\documentclass{amsart}
\usepackage{amssymb}
\usepackage{amsfonts}
\usepackage{geometry}

\newtheorem{theorem}{Theorem}
\theoremstyle{plain}

\newtheorem{definition}[theorem]{Definition}

\newtheorem{lemma}[theorem]{Lemma}
\newtheorem{notation}[theorem]{Notation}

\newtheorem{proposition}[theorem]{Proposition}
\newtheorem{remark}[theorem]{Remark}

\numberwithin{equation}{section}
\input{tcilatex}
\begin{document}
\title[Energy conditions]{Energy conditions and twisted localizations of
operators}

\begin{abstract}
We show that the energy conditions are \textbf{not} necessary for
boundedness of fractional Riesz transforms $\mathbf{R}^{\alpha ,n}$ for $%
0\leq \alpha <n$ in dimension $n\geq 2$.

We also give a weak converse, namely that the energy conditions are
necessary for boundedness of families of twisted localizations of fractional
singular integrals $\mathbf{T}^{\alpha }$ having the positive gradient
property - however, the kernels of these localizations satisfy only
one-sided Calder\'{o}n-Zygmund smoothness estimates.
\end{abstract}

\author[E.T. Sawyer]{Eric T. Sawyer}
\address{ Department of Mathematics \& Statistics, McMaster University, 1280
Main Street West, Hamilton, Ontario, Canada L8S 4K1 }
\email{sawyer@mcmaster.ca}
\thanks{Research supported in part by NSERC}
\date{January 11, 2017}
\maketitle
\tableofcontents

\section{Introduction}

An important `two weight theorem' for the Hilbert transform was obtained
early on by Nazarov, Treil and Volberg \cite{NTV4}, who proved that the
Hilbert transform $H$, with convolution kernel $K\left( x\right) =\frac{1}{x}
$, was bounded from $L^{2}\left( \sigma \right) $ to $L^{2}\left( \omega
\right) $, i.e.%
\begin{equation*}
\int_{\mathbb{R}}\left\vert H_{\sigma }f\left( x\right) \right\vert
^{2}d\omega \left( x\right) =\int_{\mathbb{R}}\left\vert \int_{\mathbb{R}}%
\frac{f\left( y\right) }{y-x}d\sigma \left( y\right) \right\vert ^{2}d\omega
\left( x\right) \leq \mathfrak{N}_{H}^{2}\int_{\mathbb{R}}\left\vert f\left(
y\right) \right\vert ^{2}d\sigma \left( y\right) ,
\end{equation*}%
for all $f\in L^{2}\left( \sigma \right) $ uniformly over suitable
truncations of the kernel $K$, provided that the following three conditions
held:

\begin{enumerate}
\item the Muckenhoupt condition,%
\begin{equation*}
\mathcal{A}_{2}\equiv \sup_{\text{intervals }I}\int_{\mathbb{R}}\frac{%
\left\vert I\right\vert }{\left\vert I\right\vert ^{2}+x^{2}}d\omega \left(
x\right) \cdot \int_{\mathbb{R}}\frac{\left\vert I\right\vert }{\left\vert
I\right\vert ^{2}+y^{2}}d\sigma \left( y\right) <\infty ,
\end{equation*}

\item the testing conditions,%
\begin{equation*}
\mathfrak{T}_{H}\equiv \sup_{\text{intervals }I}\sqrt{\frac{1}{\left\vert
I\right\vert _{\sigma }}\int_{I}\left\vert H_{\sigma }\mathbf{1}%
_{I}\right\vert ^{2}d\omega }<\infty \text{ and }\mathfrak{T}_{H}^{\ast
}\equiv \sup_{\text{intervals }I}\sqrt{\frac{1}{\left\vert I\right\vert
_{\omega }}\int_{I}\left\vert H_{\omega }\mathbf{1}_{I}\right\vert
^{2}d\sigma }<\infty ,
\end{equation*}

\item and the pivotal conditions,%
\begin{equation}
\mathcal{V}_{2}\equiv \sup_{\text{intervals }I}\sqrt{\sup_{I=\dot{\cup}I_{r}}%
\frac{1}{\left\vert I\right\vert _{\sigma }}\sum_{r=1}^{\infty }\left( \int_{%
\mathbb{R}}\frac{\left\vert I_{r}\right\vert }{\left\vert I_{r}\right\vert
^{2}+y^{2}}d\sigma \left( y\right) \right) ^{2}\left\vert I_{r}\right\vert
_{\omega }}<\infty ,  \label{piv for}
\end{equation}%
\begin{equation}
\mathcal{V}_{2}^{\ast }\equiv \sup_{\text{intervals }I}\sqrt{\sup_{I=\dot{%
\cup}I_{r}}\frac{1}{\left\vert I\right\vert _{\omega }}\sum_{r=1}^{\infty
}\left( \int_{\mathbb{R}}\frac{\left\vert I_{r}\right\vert }{\left\vert
I_{r}\right\vert ^{2}+y^{2}}d\omega \left( y\right) \right) ^{2}\left\vert
I_{r}\right\vert _{\sigma }}<\infty .  \label{piv back}
\end{equation}
\end{enumerate}

The first two conditions are necessary for boundedness of $H$, but the third
condition is not. This was established in Lacey, Sawyer and Uriarte-Tuero 
\cite{LaSaUr2}, where a substitute for the pair of pivotal conditions was
introduced, namely the pair of \textbf{energy conditions},%
\begin{equation}
\mathcal{E}_{2}\equiv \sup_{\text{intervals }I}\sqrt{\sup_{I=\dot{\cup}I_{r}}%
\frac{1}{\left\vert I\right\vert _{\sigma }}\sum_{r=1}^{\infty }\left( \int_{%
\mathbb{R}}\frac{\left\vert I_{r}\right\vert }{\left\vert I_{r}\right\vert
^{2}+y^{2}}d\sigma \left( y\right) \right) ^{2}\left\vert I_{r}\right\vert
_{\omega }\mathsf{E}\left( I_{r},\omega \right) ^{2}}<\infty ,
\label{ener for}
\end{equation}%
\begin{equation}
\mathcal{E}_{2}^{\ast }\equiv \sup_{\text{intervals }I}\sqrt{\sup_{I=\dot{%
\cup}I_{r}}\frac{1}{\left\vert I\right\vert _{\omega }}\sum_{r=1}^{\infty
}\left( \int_{\mathbb{R}}\frac{\left\vert I_{r}\right\vert }{\left\vert
I_{r}\right\vert ^{2}+y^{2}}d\omega \left( y\right) \right) ^{2}\left\vert
I_{r}\right\vert _{\sigma }\mathsf{E}\left( I_{r},\sigma \right) ^{2}}%
<\infty ,  \label{ener back}
\end{equation}%
and this pair was shown to be not only necessary for boundedness of the
Hilbert transform to hold, but in fact necessary for the Muckenhoupt and
testing conditions to hold. The quantity 
\begin{equation}
\mathsf{E}\left( J,\mu \right) ^{2}\equiv 2\frac{1}{\left\vert J\right\vert
_{\mu }}\int_{J}\frac{1}{\left\vert J\right\vert _{\mu }}\int_{J}\left\vert 
\frac{x-x^{\prime }}{\left\vert J\right\vert }\right\vert ^{2}d\mu \left(
x\right) d\mu \left( x^{\prime }\right) ,  \label{def energy}
\end{equation}%
is a one-dimensional $L^{2}$ version of the familiar normalized self-energy
of the charge distribution $\mathbf{1}_{I}\mu $ in physics, and $\mathsf{E}%
\left( \left[ a,b\right] ,\mu \right) $ takes values near $0$ for highly
concentrated distributions such as $\delta _{a}$, and values near $1$ for
highly spread out distributions such as $\mu =\delta _{a}+\delta _{b}$ (just
the opposite from self-energy in $3$-space, since the exponent $n-2$ of
Laplace's fundamental solution changes sign when $n$ goes from $3$ to $1$).

This necessity of the energy condition reinforced the $T1$ conjecture of NTV 
\cite{Vol} that boundedness of the Hilbert transform is equivalent to the
Muckenhoupt and testing conditions. And this conjecture was subsequently
proved in the two part paper \cite{LaSaShUr3}; \cite{Lac} by Lacey, Sawyer,
Shen and Uriarte-Tuero; Lacey, with the inclusion of common point masses by
Hyt\"{o}nen in \cite{Hyt2}. The energy conditions played a crucial role in
both parts \cite{LaSaShUr3} and \cite{Lac}, and their $n$-dimensional
counterparts have continued to play equally crucial roles in higher
dimensional theorems of $T1$ type, \cite{SaShUr7}-\cite{SaShUr10}, \cite%
{LaSaShUrWi}, \cite{LaWi1} and \cite{LaWi}. The known proofs of necessity of
the energy conditions broke down in higher dimensions, leaving higher
dimensional $T1$-type theorems in a state of limbo, not knowing if the
energy conditions were necessary, or if another approach was needed.

In this paper we construct families of counterexample weight pairs to show
that the energy conditions can indeed \textbf{fail} for a pair of weights,
despite boundedness of the fractional Riesz transform - the prototypical
fractional singular integral in higher dimensions. These families of
counterexamples are motivated by a weak converse result\ that we also
develop - namely that the boundedness of a large `twisted' family of
operators (related to a single `nice' singular integral) does indeed imply
the energy conditions. While this converse result may be of some theoretical
interest, it is diminished by the requirement that the testing conditions be
taken over too large a family $\left\{ \Theta _{i}\mathbf{T}_{\mathcal{J}%
}^{\alpha }\Theta _{j}\right\} _{\mathcal{J}\in \mathfrak{J}\text{ and }%
1\leq i,j\in M}$ of twisted localizations, and by the fact that the kernels
of these twisted localizations satisfy only one-sided Calder\'{o}n-Zygmund
smoothness estimates.

The counterexamples constructed here in dimension $n\geq 2$ are actually
simpler than the subtle and complicated counterexample constructed in
dimension $n=1$ by Sawyer, Shen and Uriarte-Tuero in \cite{SaShUr11}.
Indeed, the counterexample in \cite{SaShUr11} (which showed the energy
conditions are not necessary for boundedness of a certain elliptic operator
on the line) was obtained by modifying the example weight pair $\left(
\sigma ,\omega \right) $ in \cite{LaSaUr2} consisting of a Cantor measure $%
\omega $ and a discrete measure $\sigma $. In order to fail the energy
condition, the measure $\sigma $ was modified into a measure $\widehat{%
\sigma }$ by smearing out along the line each point mass in $\sigma $, so
that the local energies of the `smeared out' measure $\widehat{\sigma }$ no
longer vanished. But this `smearing out' destroyed the backward testing
condition, which then required a modification of the Hilbert transform to a
`flattened' version $H_{\flat }$, whose convolution kernel was still
elliptic $K_{\flat }\left( x\right) \approx \frac{1}{x}$, but no longer had
strictly negative derivative. This in turn forced a redistribution $\widehat{%
\omega }$ of the Cantor measure $\omega $ in order that $H_{\flat }\widehat{%
\omega }$ vanish on the support of $\widehat{\sigma }$, resulting in a
delicate and difficult recursion.

On the other hand, the extra dimension in $\mathbb{R}^{n}$ for $n\geq 2$
permits a `spreading out' of each point mass in $\sigma $ into a \emph{new}
dimension, which then requires a matching `spreading out' of the Cantor
measure $\omega $, something much simpler to deal with than that just
outlined in dimension $n=1$.

\subsection{Statements of theorems}

We state here our main theorem and proposition, but defer the definitions of
some of the terminology used in the statements, until the sections where
they are developed. First we show that the deep energy conditions are not
necessary for boundedness of fractional Riesz transforms in general.

\begin{theorem}
\label{counterexample}Let $0\leq \alpha <n$. Then there is a sequence of
pairs $\left\{ \left( \widehat{\sigma }_{N},\widehat{\omega }_{N}\right)
\right\} _{N=1}^{\infty }$ of locally finite positive Borel measures on $%
\mathbb{R}^{n}$ (actually finite sums of point masses) such that the
backward deep energy constants $\mathcal{E}_{2}^{\alpha ,\ast }\left( 
\widehat{\sigma }_{N},\widehat{\omega }_{N}\right) $ are unbounded in $N\geq
1$, and yet such that the vector Riesz transform $\mathbf{R}^{\alpha ,n}$ of
order $\alpha $ is bounded from $L^{2}\left( \widehat{\sigma }_{N}\right) $
to $L^{2}\left( \widehat{\omega }_{N}\right) $ uniformly in $N\geq 1$.
\end{theorem}

In the converse direction, we can derive the deep energy conditions, and
also the bounded overlap energy conditions, from uniform boundedness of a
large enough family of operators with uniform one-sided Calder\'{o}n-Zygmund
norms.

\begin{proposition}
\label{necessity}Let $\left( \sigma ,\omega \right) $ be a pair of locally
finite positive Borel measures on $\mathbb{R}^{n}$. Let $0\leq \alpha <n$
and suppose that $\mathbf{T}^{\alpha }$ is a standard $\alpha $-fractional
singular integral on $\mathbb{R}^{n}$ that is both \emph{strongly elliptic}
and satisfies the \emph{positive gradient property}. In particular we can
take $\mathbf{T}^{\alpha }=\mathbf{R}^{\alpha ,n}$. If the family $\left\{
\Theta _{i}\mathbf{T}_{\mathcal{J}}^{\alpha }\Theta _{j}\right\} _{\mathcal{J%
}\in \mathfrak{J}\text{ and }1\leq i,j\in M}$ of \emph{twisted localizations}
of the operator $\mathbf{T}^{\alpha }$ satisfies the testing conditions
uniformly in $\mathcal{J}\in \mathfrak{J}$ and $1\leq i,j\in M$, then the
deep and bounded overlap energy conditions hold, and moreover there is a
positive constant $C$, depending only on $n$, $\alpha $, $\left\Vert \mathbf{%
T}^{\alpha }\right\Vert _{CZ_{\alpha }}$ and the constants in the
definitions of strongly elliptic and positive gradient condition, such that%
\begin{equation*}
\mathcal{E}_{2}^{\alpha ,\limfunc{deep}}\leq \mathcal{E}_{2}^{\alpha ,%
\limfunc{overlap}}\leq C\sup_{\mathcal{J}\in \mathfrak{J}\text{ and }1\leq
i,j\leq M}\mathfrak{T}_{\Theta _{i}\mathbf{T}_{\mathcal{J}}^{\alpha }\Theta
_{j}}+C\beta \mathcal{A}_{2}^{\alpha }.
\end{equation*}%
{}
\end{proposition}

We also show in Lemma \ref{CZalpha twist} below, that the kernels $\Theta
_{i}\mathbf{K}_{\mathcal{J}}^{\alpha }\Theta _{j}$ of the operators $\Theta
_{i}\mathbf{T}_{\mathcal{J}}^{\alpha }\Theta _{j}$ satisfy one-sided Calder%
\'{o}n-Zygmund estimates%
\begin{eqnarray*}
\left\vert \Theta _{i}\mathbf{K}_{\mathcal{J}}^{\alpha }\Theta
_{j}^{-1}\left( x,y\right) \right\vert &\lesssim &\left\Vert \mathbf{K}%
^{\alpha }\right\Vert _{CZ_{\alpha }}\left\vert x-y\right\vert ^{\alpha -n},
\\
\left\vert \nabla _{y}^{\ell }\Theta _{i}\mathbf{K}_{\mathcal{J}}^{\alpha
}\Theta _{j}^{-1}\left( x,y\right) \right\vert &\lesssim &\left\Vert \mathbf{%
K}^{\alpha }\right\Vert _{CZ_{\alpha }}\left\vert x-y\right\vert ^{\alpha
-n-\ell },\ \ \ \ \ \ell =1,2,
\end{eqnarray*}%
with smoothness in the $y$-variable only, and then we show in Lemma \ref%
{Riesz pgp} below, that the negative of the Riesz transform $\mathbf{R}%
^{\alpha ,n}$ has the positive gradient property.

Proposition \ref{necessity} is proved in Part 1 of the paper, while Theorem %
\ref{counterexample} is proved in Part 2. Each of these parts can
essentially be read independently of the other.

\part{Necessity of energy conditions for twisted localizations}

In the first part of this paper, we prove Proposition \ref{necessity} by
deriving the deep and bounded overlap energy conditions, as defined below,
from testing conditions for the family of \emph{twisted localizations} of
the $\alpha $-fractional Riesz transform $\mathbf{R}^{\alpha ,n}$ in
dimension $n$, and more generally for strongly elliptic convolution vector
operators $\mathbf{T}^{\alpha }$ in place of $\mathbf{R}^{\alpha ,n}$ that
enjoy the \emph{positive gradient property}.

\section{Standard fractional singular integrals}

Let $0\leq \alpha <n$ and $0<\delta \leq 1$. We define a $\delta $-standard $%
\alpha $-fractional Calder\'{o}n-Zygmund kernel $\mathbf{K}^{\alpha }(x,y)$
to be a vector-valued function defined on $\mathbb{R}^{n}\times \mathbb{R}%
^{n}$ whose components uniformly satisfy the following fractional size and
smoothness conditions: For $x\neq y$ in $\mathbb{R}^{n}$,%
\begin{eqnarray}
\left\vert \mathbf{K}^{\alpha }\left( x,y\right) \right\vert &\leq
&C_{CZ_{\alpha }}\left\vert x-y\right\vert ^{\alpha -n}\text{ and }%
\left\vert \nabla \mathbf{K}^{\alpha }\left( x,y\right) \right\vert \leq
C_{CZ_{\alpha }}\left\vert x-y\right\vert ^{\alpha -n-1},  \label{size} \\
\left\vert \nabla \mathbf{K}^{\alpha }\left( x,y\right) -\nabla \mathbf{K}%
^{\alpha }\left( x^{\prime },y\right) \right\vert &\leq &C_{CZ_{\alpha
}}\left( \frac{\left\vert x-x^{\prime }\right\vert }{\left\vert
x-y\right\vert }\right) ^{\delta }\left\vert x-y\right\vert ^{\alpha -n-1},\
\ \ \ \ \frac{\left\vert x-x^{\prime }\right\vert }{\left\vert
x-y\right\vert }\leq \frac{1}{2},  \notag
\end{eqnarray}%
and where the last inequality also holds for the adjoint kernel in which $x$
and $y$ are interchanged. We define the Calder\'{o}n-Zygmund norm $%
\left\Vert \mathbf{K}^{\alpha }\right\Vert _{CZ_{\alpha }}$ of $\mathbf{K}%
^{\alpha }$ to be the least constant $C_{CZ_{\alpha }}$ for which the above
display holds.

\subsection{Defining the norm inequality and testing conditions}

We now recall the precise definition of the weighted norm inequality%
\begin{equation}
\left\Vert \mathbf{T}_{\sigma }^{\alpha }f\right\Vert _{L^{2}\left( \omega
\right) }\leq \mathfrak{N}_{\mathbf{T}^{\alpha }}\left\Vert f\right\Vert
_{L^{2}\left( \sigma \right) },\ \ \ \ \ f\in L^{2}\left( \sigma \right) ,
\label{two weight'}
\end{equation}%
as in \cite{SaShUr10} for example. Let $\left\{ \eta _{\delta ,R}^{\alpha
}\right\} _{0<\delta <R<\infty }$ be a family of nonnegative functions on $%
\left[ 0,\infty \right) $ so that the truncated kernels $\mathbf{K}_{\delta
,R}^{\alpha }\left( x,y\right) =\eta _{\delta ,R}^{\alpha }\left( \left\vert
x-y\right\vert \right) \mathbf{K}^{\alpha }\left( x,y\right) $ of the
operator $\mathbf{T}^{\alpha }$ are bounded with compact support for fixed $%
x $ or $y$. Then the truncated operators 
\begin{equation*}
\mathbf{T}_{\sigma ,\delta ,R}^{\alpha }f\left( x\right) \equiv \int_{%
\mathbb{R}}\mathbf{K}_{\delta ,R}^{\alpha }\left( x,y\right) f\left(
y\right) d\sigma \left( y\right) ,\ \ \ \ \ x\in \mathbb{R}^{n},
\end{equation*}%
are pointwise well-defined, and we will refer to the pair $\left( K^{\alpha
},\left\{ \eta _{\delta ,R}^{\alpha }\right\} _{0<\delta <R<\infty }\right) $
as an $\alpha $-fractional singular integral operator, which we typically
denote by $\mathbf{T}^{\alpha }$, suppressing the dependence on the
truncations. When $\mathbf{K}^{\alpha }\left( x,y\right) $ is a $\delta $%
-standard $\alpha $-fractional Calder\'{o}n-Zygmund kernel, we say that $%
\mathbf{T}^{\alpha }=\left( \mathbf{K}^{\alpha },\left\{ \eta _{\delta
,R}^{\alpha }\right\} _{0<\delta <R<\infty }\right) $ is a $\delta $%
-standard $\alpha $-fractional singular integral.

\begin{definition}
We say that an $\alpha $-fractional singular integral operator $T^{\alpha
}=\left( K^{\alpha },\left\{ \eta _{\delta ,R}^{\alpha }\right\} _{0<\delta
<R<\infty }\right) $ satisfies the norm inequality (\ref{two weight'})
provided%
\begin{equation*}
\left\Vert T_{\sigma ,\delta ,R}^{\alpha }f\right\Vert _{L^{2}\left( \omega
\right) }\leq \mathfrak{N}_{T_{\sigma }^{\alpha }}\left\Vert f\right\Vert
_{L^{2}\left( \sigma \right) },\ \ \ \ \ f\in L^{2}\left( \sigma \right)
,0<\delta <R<\infty .
\end{equation*}
\end{definition}

It turns out that, in the presence of the Muckenhoupt conditions (\ref{def
A2}) below, the norm inequality (\ref{two weight'}) is essentially
independent of the choice of truncations used (see e.g. \cite{LaSaShUr3} in
dimension $n=1$), and this is explained in some detail in \cite{SaShUr10}.
Thus, as in \cite{SaShUr10}, we are free to use the tangent line truncations
described there throughout this paper, and in particular we interpret the
testing conditions below using the tangent line truncations:.%
\begin{equation}
\mathfrak{T}_{T^{\alpha }}\equiv \sup_{\text{cubes }I}\sqrt{\frac{1}{%
\left\vert I\right\vert _{\sigma }}\int_{I}\left\vert T_{\sigma }^{\alpha }%
\mathbf{1}_{I}\right\vert ^{2}d\omega }<\infty \text{ and }\mathfrak{T}%
_{T^{\alpha ,\ast }}\equiv \sup_{\text{cubels }I}\sqrt{\frac{1}{\left\vert
I\right\vert _{\omega }}\int_{I}\left\vert T_{\omega }^{\alpha ,\ast }%
\mathbf{1}_{I}\right\vert ^{2}d\sigma }<\infty .  \label{testing conditions}
\end{equation}

\subsection{Strong ellipticity and the positive gradient property}

Recall from \cite{SaShUr7} that a standard $\alpha $-fractional vector
singular integral $\mathbf{T}^{\alpha }$ on $\mathbb{R}^{n}$ with vector
kernel $\mathbf{K}^{\alpha }=\left( K_{j}^{\alpha }\right) _{j=1}^{J}$ is 
\emph{strongly elliptic} if for each $m\in \left\{ 1,-1\right\} ^{n}$, there
is a sequence of coefficients $\left\{ \lambda _{j}^{m}\right\} _{j=1}^{J}$
such that%
\begin{equation}
\left\vert \sum_{j=1}^{J}\lambda _{j}^{m}K_{j}^{\alpha }\left( x,x+t\mathbf{u%
}\right) \right\vert \geq ct^{\alpha -n},\ \ \ \ \ t\in \mathbb{R},
\label{Ktalpha strong}
\end{equation}%
holds for \emph{all} unit vectors $\mathbf{u}$ in the $n$-ant $V_{m}$ (i.e.
an $n$-dimensional quadrant) where%
\begin{equation*}
V_{m}=\left\{ x\in \mathbb{R}^{n}:m_{i}x_{i}>0\text{ for }1\leq i\leq
n\right\} ,\ \ \ \ \ m\in \left\{ 1,-1\right\} ^{n}.
\end{equation*}

We now define the positive gradient property. We say that a strongly
elliptic standard $\alpha $-fractional \textbf{convolution} singular
integral $\mathbf{T}^{\alpha }$ has the \emph{positive gradient property} if
in addition there is a finite sequence of closed sectors $\left\{
S_{j}\right\} _{j=1}^{M}$ such that:

\begin{enumerate}
\item we have $\mathbb{R}^{n}=\bigcup_{j=1}^{M}S_{j}$ and there is a
positive constant $\theta =\theta _{\mathbf{T}^{\alpha }}>0$ so that each
sector $S_{j}$ is a rotation $\Theta _{j}$ of the unit sector $S$ of
aperture $\theta $, where 
\begin{equation*}
S\equiv \left\{ y\in \mathbb{R}^{n}:y^{\prime }=\frac{y}{\left\vert
y\right\vert }\in B_{\mathbb{S}^{n-1}}\left( \mathbf{e}_{1},\theta \right)
\right\} .
\end{equation*}

\item for each sector $S_{j}=\Theta _{j}S$ there is a sequence of
coefficients $\left\{ \lambda _{j}^{m}\right\} _{j=1}^{J}$ such that the
scalar convolution kernel $K\left( \xi \right) \equiv \sum_{j=1}^{J}\lambda
_{j}^{m}K_{j}^{\alpha }\left( \Theta _{j}^{-1}\xi \right) $ satisfies 
\begin{equation*}
\frac{K\left( \xi \right) -K\left( \eta \right) }{\xi _{1}-\eta _{1}}\approx
\left\vert \xi \right\vert ^{\alpha -n-1},\ \ \ \ \ \text{for }\xi =\left(
\xi _{1},\widetilde{\xi }\right) ,\eta =\left( \eta _{1},\widetilde{\eta }%
\right) \in S\text{ with }\frac{\left\vert \widetilde{\xi }-\widetilde{\eta }%
\right\vert }{\left\vert \xi _{1}-\eta _{1}\right\vert }\leq \tan \theta .
\end{equation*}
\end{enumerate}

It is obvious that the vector Riesz transform $\mathbf{R}^{\alpha ,n}$ is a
strongly elliptic convolution singular integral, and we prove in Lemma \ref%
{Riesz pgp} below that its negative has the positive gradient property. The
rotation invariance of $\mathbf{R}^{\alpha ,n}$ makes each of these two
properties easy to establish.

\begin{remark}
In dimension $n=1$ the positive gradient property reduces to $\frac{d}{dx}%
K^{\alpha }\left( x\right) \approx -\frac{1}{x^{2-\alpha }}$, and in the
case of the kernel $K\left( x\right) =\frac{1}{x}$ for the Hilbert
transform, we actually have equality, $\frac{d}{dx}K\left( x\right) =-\frac{1%
}{x^{2}}$, a property exploited extensively in the proof of the NTV
conjecture in \cite{LaSaShUr3}, \cite{Lac}.
\end{remark}

\section{Poisson integrals and Muckenhoupt conditions}

Let $\mu $ be a locally finite positive Borel measure on $\mathbb{R}^{n}$,
and suppose $Q$ is a cube in $\mathbb{R}^{n}$. Recall that $\left\vert
Q\right\vert =\ell \left( Q\right) ^{n}$ where $\ell \left( Q\right) $ is
the side length of a cube $Q$. The two $\alpha $-fractional Poisson
integrals of $\mu $ on a cube $Q$ are given by the following expressions:%
\begin{eqnarray*}
\mathrm{P}^{\alpha }\left( Q,\mu \right) &\equiv &\int_{\mathbb{R}^{n}}\frac{%
\left\vert Q\right\vert ^{\frac{1}{n}}}{\left( \left\vert Q\right\vert ^{%
\frac{1}{n}}+\left\vert x-c_{Q}\right\vert \right) ^{n+1-\alpha }}d\mu
\left( x\right) , \\
\mathcal{P}^{\alpha }\left( Q,\mu \right) &\equiv &\int_{\mathbb{R}%
^{n}}\left( \frac{\left\vert Q\right\vert ^{\frac{1}{n}}}{\left( \left\vert
Q\right\vert ^{\frac{1}{n}}+\left\vert x-c_{Q}\right\vert \right) ^{2}}%
\right) ^{n-\alpha }d\mu \left( x\right) ,
\end{eqnarray*}%
where $\left\vert x-c_{Q}\right\vert $ denotes distance between $x$ and the
center $c_{Q}$ of $Q$, and $\left\vert Q\right\vert $ denotes the Lebesgue
measure of $Q$. We refer to $\mathrm{P}^{\alpha }$ as the \emph{standard}
Poisson integral and to $\mathcal{P}^{\alpha }$ as the \emph{reproducing}
Poisson integral. Note that for $n-1\leq \alpha <n$, these two kernels
satisfy 
\begin{equation*}
\mathrm{P}^{\alpha }\left( Q,\mu \right) \leq \mathcal{P}^{\alpha }\left(
Q,\mu \right) ,\ \ \ \ \text{for all intervals }Q\text{ and positive
measures }\mu ,
\end{equation*}%
and that the inequality is reversed for $0\leq \alpha \leq n-1$.

We now define the \emph{one-tailed} $\mathcal{A}_{2}^{\alpha }$ constant
using $\mathcal{P}^{\alpha }$. The energy constants $\mathcal{E}_{\alpha }$
introduced in the next section will use the standard Poisson integral $%
\mathrm{P}^{\alpha }$. We denote the collection of cubes in $\mathbb{R}^{n}$
with edges parallel to the coordinate axes by $\mathcal{P}^{n}$ (not to be
confused with the Poisson integral $\mathcal{P}^{\alpha }$).

\begin{definition}
The one-sided constants $\mathcal{A}_{2}^{\alpha }$ and $\mathcal{A}%
_{2}^{\alpha ,\ast }$ for the weight pair $\left( \sigma ,\omega \right) $
are given by%
\begin{eqnarray}
\mathcal{A}_{2}^{\alpha } &\equiv &\sup_{Q\in \mathcal{P}^{n}}\mathcal{P}%
^{\alpha }\left( Q,\mathbf{1}_{Q^{c}}\sigma \right) \frac{\left\vert
Q\right\vert _{\omega }}{\left\vert Q\right\vert ^{1-\frac{\alpha }{n}}}%
<\infty ,  \label{def A2} \\
\mathcal{A}_{2}^{\alpha ,\ast } &\equiv &\sup_{Q\in \mathcal{P}^{n}}\mathcal{%
P}^{\alpha }\left( Q,\mathbf{1}_{Q^{c}}\omega \right) \frac{\left\vert
Q\right\vert _{\sigma }}{\left\vert Q\right\vert ^{1-\frac{\alpha }{n}}}%
<\infty .  \notag
\end{eqnarray}
\end{definition}

Note that these definitions are the analogues of the corresponding
conditions with `holes' introduced by Hyt\"{o}nen \cite{Hyt2} in dimension $%
n=1$ - the supports of the measures $\mathbf{1}_{Q^{c}}\sigma $ and $\mathbf{%
1}_{Q}\omega $ in the definition of $\mathcal{A}_{2}^{\alpha }$ are
disjoint, and so the common point masses of $\sigma $ and $\omega $ do not
appear simultaneously in each factor.

\section{Strong, deep and bounded overlap energy constants}

We begin with the \emph{strong} energy constants (see e.g. \cite{LaSaUr2}
and \cite{SaShUr7}).

\begin{definition}
\label{def strong quasienergy}Let $0\leq \alpha <n$. Suppose $\sigma $ and $%
\omega $ are locally finite positive Borel measures on $\mathbb{R}^{n}$.
Then the \emph{strong} energy constant $\mathcal{E}_{2}^{\alpha }$ is
defined by 
\begin{equation}
\left( \mathcal{E}_{2}^{\alpha }\right) ^{2}\equiv \sup_{I=\dot{\cup}I_{r}}%
\frac{1}{\left\vert I\right\vert _{\sigma }}\sum_{r=1}^{\infty }\left( \frac{%
\mathrm{P}^{\alpha }\left( I_{r},\mathbf{1}_{I}\sigma \right) }{\left\vert
I_{r}\right\vert }\right) ^{2}\left\Vert x-m_{I_{r}}^{\omega }\right\Vert
_{L^{2}\left( \mathbf{1}_{I_{r}}\omega \right) }^{2}\ ,
\label{strong b* energy}
\end{equation}%
where the supremum is taken over arbitrary decompositions of a cube $I$
using a pairwise disjoint union of subcubes $I_{r}$. Similarly, we define
the dual \emph{strong} energy constant $\mathcal{E}_{2}^{\alpha ,\ast }$ by
switching the roles of $\sigma $ and $\omega $:%
\begin{equation}
\left( \mathcal{E}_{2}^{\alpha ,\ast }\right) ^{2}\equiv \sup_{I=\dot{\cup}%
I_{r}}\frac{1}{\left\vert I\right\vert _{\omega }}\sum_{r=1}^{\infty }\left( 
\frac{\mathrm{P}^{\alpha }\left( I_{r},\mathbf{1}_{I}\omega \right) }{%
\left\vert I_{r}\right\vert }\right) ^{2}\left\Vert x-m_{I_{r}}^{\sigma
}\right\Vert _{L^{2}\left( \mathbf{1}_{I_{r}}\sigma \right) }^{2}\ .
\label{strong b energy}
\end{equation}
\end{definition}

In order to define the weaker notions of deep and bounded overlap energy
constants, we must introduce additional notation. We say that a dyadic cube $%
J$ is $\left( \mathbf{r},\varepsilon \right) $-\emph{deeply embedded} in a
(not necessarily dyadic) quasicube $K$, which we write as $J\Subset _{%
\mathbf{r},\varepsilon }K$, when $J\subset K$ and both 
\begin{eqnarray}
\ell \left( J\right) &\leq &2^{-\mathbf{r}}\ell \left( K\right) ,
\label{def deep embed} \\
\limfunc{dist}\left( J,\partial K\right) &\geq &\frac{1}{2}\ell \left(
J\right) ^{\varepsilon }\ell \left( K\right) ^{1-\varepsilon },  \notag
\end{eqnarray}%
Recall the collection 
\begin{equation*}
\mathcal{M}_{\left( \mathbf{r},\varepsilon \right) -\limfunc{deep}}\left(
K\right) \equiv \left\{ \text{maximal }J\Subset _{\mathbf{r},\varepsilon
}K\right\}
\end{equation*}%
of \emph{maximal} $\left( \mathbf{r},\varepsilon \right) $-deeply embedded
dyadic subcubes of a cube $K$ (a subcube $J$ of $K$ is a \emph{dyadic}
subcube of $K$ if $J\in \mathcal{D}$ when $\mathcal{D}$ is a dyadic grid
containing $K$). This collection of dyadic subcubes of $K$ is of course a
pairwise disjoint decomposition of $K$. Recall also the refinement and
extension of the collection $\mathcal{M}_{\left( \mathbf{r},\varepsilon
\right) -\limfunc{deep}}\left( K\right) $ given in \cite{SaShUr7} for
certain $K$ and each $\ell \geq 1$ (where $\pi ^{\ell }K^{\prime }$ denotes
the $\ell ^{th}$ ancestor of $K^{\prime }$ in the grid): 
\begin{eqnarray*}
\mathcal{M}_{\left( \mathbf{r},\varepsilon \right) -\limfunc{deep}}^{\ell
}\left( K\right) &\equiv &\left\{ J\in \mathcal{M}_{\left( \mathbf{r}%
,\varepsilon \right) -\limfunc{deep}}\left( \pi ^{\ell }K^{\prime }\right) 
\text{ for some }K^{\prime }\in \mathfrak{C}_{\mathcal{D}}\left( K\right)
:\right. \\
&&\ \ \ \ \ \ \ \ \ \ \ \ \ \ \ \ \ \ \ \ \ \ \ \ \ \ \ \ \ \ \left.
J\subset L\text{ for some }L\in \mathcal{M}_{\left( \mathbf{r},\varepsilon
\right) -\limfunc{deep}}\left( K\right) \right\} ,
\end{eqnarray*}%
where $\mathfrak{C}_{\mathcal{D}}\left( K\right) $ is the set of $\mathcal{D}
$-dyadic children of $K$. Thus $\mathcal{M}_{\left( \mathbf{r},\varepsilon
\right) -\limfunc{deep}}^{\ell }\left( K\right) $ is the union, over all
children $K^{\prime }$ of $K$, of those cubes in $\mathcal{M}_{\left( 
\mathbf{r},\varepsilon \right) -\limfunc{deep}}\left( \pi ^{\ell }K^{\prime
}\right) $ that happen to be contained in some $L\in \mathcal{M}_{\left( 
\mathbf{r},\varepsilon \right) -\limfunc{deep}}\left( K\right) $. These
collections of cubes satisfy the bounded overlap property (see e.g. \cite%
{SaShUr7}), 
\begin{equation*}
\sum_{J\in \mathcal{M}_{\left( \mathbf{r},\varepsilon \right) -\limfunc{deep}%
}^{\ell }\left( K\right) }\mathbf{1}_{\gamma J}\leq \beta \mathbf{1}_{K},\ \
\ \ \ \text{for each }\ell \geq 1.
\end{equation*}

Finally, let $\mathsf{P}_{M}^{\omega }\equiv \sum_{J\in \mathcal{D}:\
J\subset M}\bigtriangleup _{J}^{\omega }$ be Haar projection onto the
subspace of $L^{2}\left( \omega \right) $ consisting of those functions $%
f\in L^{2}\left( \omega \right) $ supported in $M$ with $\int_{M}fd\omega =0$
- see e.g. \cite{SaShUr7} for more detail on Haar expansions in $L^{2}\left(
\omega \right) $.

\begin{definition}
\label{energy condition}Suppose $\sigma $ and $\omega $ are locally finite
positive Borel measures on $\mathbb{R}^{n}$ and fix $\gamma >1$. Then the\
deep energy condition constant $\mathcal{E}_{\alpha }^{\limfunc{deep}}$ is
given by%
\begin{equation*}
\left( \mathcal{E}_{2}^{\alpha ,\limfunc{deep}}\right) ^{2}\equiv \sup_{\ell
\geq 1}\sup_{\mathcal{D}}\sup_{I=\dot{\cup}I_{r}}\frac{1}{\left\vert
I\right\vert _{\sigma }}\sum_{r=1}^{\infty }\sum_{M\in \mathcal{M}_{\left( 
\mathbf{r},\varepsilon \right) -\limfunc{deep}}^{\ell }\left( I_{r}\right)
}\left( \frac{\mathrm{P}^{\alpha }\left( M,\mathbf{1}_{I\setminus \gamma
M}\sigma \right) }{\left\vert M\right\vert }\right) ^{2}\left\Vert \mathsf{P}%
_{M}^{\omega }x\right\Vert _{L^{2}\left( \omega \right) }^{2},
\end{equation*}%
where $\sup_{\mathcal{D}}\sup_{I=\dot{\cup}I_{r}}$ is taken over

\begin{enumerate}
\item all dyadic grids $\mathcal{D}$,

\item all $\mathcal{D}$-dyadic cubes $I$,

\item and all subpartitions $\left\{ I_{r}\right\} _{r=1}^{N\text{ or }%
\infty }$ of the cube $I$ into $\mathcal{D}$-dyadic subcubes $I_{r}$.
\end{enumerate}
\end{definition}

The exact value of $\gamma >1$ above is not too important in general, but
when we wish to emphasize the value of $\gamma $, we will refer to $\mathcal{%
E}_{\alpha }^{\limfunc{deep}}$ as the $\gamma $-deep energy condition
constant.

Note that we could also define a slightly less restrictive notion of energy
condition as in \cite{LaWi} by taking the supremum over $I=\dot{\cup}I_{r}$
for which there is bounded overlap of the expansions $\gamma I_{r}$,%
\begin{equation}
\left( \mathcal{E}_{2}^{\alpha ,\limfunc{overlap}}\right) ^{2}\equiv \sup_{%
\mathcal{D}}\sup_{\substack{ I=\dot{\cup}I_{r}  \\ \sum_{r=1}^{\infty }%
\mathbf{1}_{\gamma I_{r}}\leq \beta \mathbf{1}_{I}}}\frac{1}{\left\vert
I\right\vert _{\sigma }}\sum_{r=1}^{\infty }\left( \frac{\mathrm{P}^{\alpha
}\left( I_{r},\mathbf{1}_{I\setminus \gamma I_{r}}\sigma \right) }{%
\left\vert I_{r}\right\vert ^{\frac{1}{n}}}\right) ^{2}\left\Vert \mathsf{P}%
_{I_{r}}^{\omega }x\right\Vert _{L^{2}\left( \omega \right) }^{2}\ ,
\label{deep overlap energy}
\end{equation}%
and we refer to finiteness of $\mathcal{E}_{2}^{\alpha ,\limfunc{overlap}}$
as the \emph{bounded overlap} energy condition, or more precisely as the $%
\gamma $-overlap energy condition when we want to emphasize the choice of $%
\gamma $.

Later, in Part 2 of the paper, we will have reason to consider the
corresponding forward (bounded overlap) \emph{pivotal} constant, which is
defined by replacing $\left\Vert \mathsf{P}_{I_{r}}^{\omega }x\right\Vert
_{L^{2}\left( \omega \right) }^{2}$ with its upper bound $\left\vert
I_{r}\right\vert ^{\frac{2}{n}}\left\vert I_{r}\right\vert _{\omega }$ in (%
\ref{deep overlap energy}):%
\begin{equation}
\left( \mathcal{V}_{2}^{\alpha ,\limfunc{overlap}}\right) ^{2}\equiv \sup_{%
\mathcal{D}}\sup_{\substack{ I=\dot{\cup}I_{r} \\ \sum_{r=1}^{\infty }%
\mathbf{1}_{\gamma I_{r}}\leq \beta \mathbf{1}_{I}}}\frac{1}{\left\vert
I\right\vert _{\sigma }}\sum_{r=1}^{\infty }\mathrm{P}^{\alpha }\left( I_{r},%
\mathbf{1}_{I\setminus \gamma I_{r}}\sigma \right) ^{2}\left\vert
I_{r}\right\vert _{\omega }\ .  \label{deep overlap pivotal}
\end{equation}%
The backward \emph{pivotal} constant $\mathcal{V}_{2}^{\alpha ,\limfunc{%
overlap},\ast }$ is defined by interchanging the roles of the measures $%
\sigma $ and $\omega $.

\section{Twisted localizations and necessity of the deep energy condition}

Let $\left\{ \Theta _{j}\right\} _{j=1}^{M}$ be a finite set of rotations
such that $\mathbb{R}^{n}=\overset{\cdot }{\bigcup }_{j=1}^{M}\Theta _{j}%
\widehat{Q}$ where $\widehat{Q}$ is the sector centered on the positive $%
x_{1}$-axis with aperture angle $\theta >0$ as in the positive gradient
property for the strongly elliptic convolution singular integral $\mathbf{T}%
^{\alpha }$ in $\mathbb{R}^{n}$. Our goal here is to prove the (forward) $%
\gamma $-overlap energy condition with constant 
\begin{equation*}
\mathcal{E}_{2}^{\alpha ,\limfunc{overlap}}\lesssim \sup_{\mathcal{J}%
}\sup_{1\leq i,j\leq N}\mathfrak{T}_{\Theta _{i}\mathbf{T}_{\mathcal{J}%
}^{\alpha }\Theta _{j}}+\mathcal{A}_{2}^{\alpha },
\end{equation*}%
where $\left\{ \Theta _{i}\mathbf{T}_{\mathcal{J}}^{\alpha }\Theta
_{j}\right\} $ is a family of standard fractional singular integrals
associated with $\mathbf{T}^{\alpha }$. More precisely we will show%
\begin{equation*}
\sum_{r=1}^{\infty }\left( \frac{\mathrm{P}^{\alpha }\left( J,\mathbf{1}%
_{I\setminus \gamma I_{r}}\sigma \right) }{\left\vert I_{r}\right\vert ^{%
\frac{1}{n}}}\right) ^{2}\left\Vert \mathsf{P}_{I_{r}}^{\omega }\mathbf{x}%
\right\Vert _{L^{2}\left( \omega \right) }^{2}\leq \left\{ \sup_{\mathcal{J}%
}\sup_{1\leq i,j\leq N}\left( \mathfrak{T}_{\Theta _{i}\mathbf{T}%
_{J}^{\alpha }\Theta _{j}}\right) ^{2}+\beta \mathcal{A}_{2}^{\alpha
}\right\} \ \left\vert I\right\vert _{\sigma }\ ,
\end{equation*}%
for all partitions of a dyadic cube $I=\overset{\cdot }{\dbigcup\limits_{r%
\geq 1}}I_{r}$ into subcubes $I_{r}$ with $\sum_{r=1}^{\infty }\mathbf{1}%
_{\gamma I_{r}}\lesssim \beta \mathbf{1}_{I}$. We now turn to defining the
twisted localizations $\Theta _{i}\mathbf{T}_{J}^{\alpha }\Theta _{j}$
appearing on the right hand side of the inequality displayed above.

Let $Q\equiv \left[ -\frac{1}{2},\frac{1}{2}\right] ^{n}$ be the \emph{unit
cube} of side length $1$ centered at the origin, and let 
\begin{eqnarray*}
\widehat{Q} &\equiv &\left\{ y=\left( y_{1},y^{\prime }\right) \in \left( 
\mathbb{R}\times \mathbb{R}^{n-1}\right) \setminus \gamma Q:\left\vert
y^{\prime }\right\vert \leq \lambda \left\vert y_{1}\right\vert \right\} \\
&=&\left\{ y\in \mathbb{R}^{n}\setminus \gamma Q:y^{\prime }=\frac{y}{%
\left\vert y\right\vert }\in B_{\mathbb{S}^{n-1}}\left( \mathbf{e}%
_{1},\theta \right) \right\}
\end{eqnarray*}%
be the \emph{unit truncated sector} of separation $\gamma $ and aperture $%
\theta =\arctan \lambda $ for $\gamma >1$ and $\lambda >0$ chosen as needed
below. Let $\varphi $ be a smooth bump function that equals $1$ on $\widehat{%
Q}$ and vanishes off an appropriate $\rho $-expansion%
\begin{equation*}
\widehat{Q}^{\ast }=\left\{ y\in \mathbb{R}^{n}\setminus \frac{\gamma }{\rho 
}Q:y^{\prime }=\frac{y}{\left\vert y\right\vert }\in \rho B_{\mathbb{S}%
^{n-1}}\left( \mathbf{e}_{1},\theta \right) \right\} ,
\end{equation*}%
where $1<\rho <\gamma $, and such that%
\begin{equation*}
\left\vert \nabla \varphi \left( y\right) \right\vert \leq 1\text{ and }%
\left\vert \nabla \varphi \left( y\right) \right\vert \leq C_{\varphi
}\left\vert y\right\vert ^{-1},\ \ \ \ \ y\in \mathbb{R}^{n}.
\end{equation*}%
In particular, we can choose the bump function $\varphi $ so that the
localized kernel $\mathbf{1}_{J}\left( x\right) \mathbf{K}^{\alpha }\left(
x,y\right) \varphi \left( y\right) $ satisfies a \emph{one-sided} Calder\'{o}%
n-Zygmund condition, in which there is smoothness only in the $y$-variable.
See below.

We also define such a bump function for each `rotated' sector%
\begin{equation*}
\widehat{Q^{j}}\equiv \left\{ y\in \mathbb{R}^{n}\setminus \gamma
Q:y^{\prime }=\frac{y}{\left\vert y\right\vert }\in B_{\mathbb{S}%
^{n-1}}\left( \Theta _{j}\mathbf{e}_{1},\theta \right) \right\} ,
\end{equation*}%
which with a small abuse of notation we denote by $\Theta _{j}\widehat{Q}$,
despite the fact that $\widehat{Q^{j}}$ is not exactly a rotation of $%
\widehat{Q}$. But since the cube $Q$ is not rotation invariant, we cannot
simply take a rotation of $\varphi $. Thus for each $1\leq j\leq M$, we
choose a bump function $\varphi ^{j}$ that is equals $1$ on the sector $%
\Theta _{j}\widehat{Q}=\widehat{Q^{j}}$ and is supported in the $\rho $%
-expansion of the sector,%
\begin{equation*}
\left\{ y\in \mathbb{R}^{n}\setminus \frac{\gamma }{\rho }Q:y^{\prime }=%
\frac{y}{\left\vert y\right\vert }\in \rho B_{\mathbb{S}^{n-1}}\left( \Theta
_{j}\mathbf{e}_{1},\theta \right) \right\} ,
\end{equation*}%
and satisfies appropriate estimates. To avoid clutter of notation. we will
typically suppress the superscript $j$ and simply write $\varphi $ for each
of these bump functions $\varphi _{1},...,\varphi _{M}$.

\begin{lemma}
With notation as above,%
\begin{eqnarray}
\left\vert \mathbf{1}_{Q}\left( x\right) \mathbf{K}^{\alpha }\left(
x,y\right) \varphi \left( y\right) \right\vert &\leq &\left\Vert \mathbf{K}%
^{\alpha }\right\Vert _{CZ_{\alpha }}\left\vert x-y\right\vert ^{\alpha -n},
\label{CZ one-sided} \\
\left\vert \nabla _{y}^{1}\mathbf{1}_{Q}\left( x\right) \mathbf{K}^{\alpha
}\left( x,y\right) \varphi \left( y\right) \right\vert &\lesssim &\left\Vert 
\mathbf{K}^{\alpha }\right\Vert _{CZ_{\alpha }}\left\vert x-y\right\vert
^{\alpha -n-1},  \notag \\
\left\vert \nabla _{y}^{2}\mathbf{1}_{Q}\left( x\right) \mathbf{K}^{\alpha
}\left( x,y\right) \varphi \left( y\right) \right\vert &\lesssim &\left\Vert 
\mathbf{K}^{\alpha }\right\Vert _{CZ_{\alpha }}\left\vert x-y\right\vert
^{\alpha -n-2}.  \notag
\end{eqnarray}
\end{lemma}

\begin{proof}
We trivially have the first line in (\ref{CZ one-sided}),%
\begin{equation*}
\left\vert \mathbf{1}_{Q}\left( x\right) \mathbf{K}^{\alpha }\left(
x,y\right) \varphi \left( y\right) \right\vert \leq \left\vert \mathbf{K}%
^{\alpha }\left( x,y\right) \right\vert \leq \left\Vert \mathbf{K}^{\alpha
}\right\Vert _{CZ_{\alpha }}\left\vert x-y\right\vert ^{\alpha -n}.
\end{equation*}%
If in addition $\mathbf{1}_{Q}\left( x\right) \varphi \left( y\right) \neq 0$%
, then 
\begin{equation*}
3\sqrt{n}\left\vert y\right\vert \leq \left\vert x-y\right\vert \leq 5\sqrt{n%
}\left\vert y\right\vert ,
\end{equation*}%
and so 
\begin{eqnarray*}
\left\vert \nabla _{y}\mathbf{1}_{Q}\left( x\right) \mathbf{K}^{\alpha
}\left( x,y\right) \varphi \left( y\right) \right\vert &\leq &\left\vert
\nabla _{y}\mathbf{K}^{\alpha }\left( x,y\right) \right\vert \ \left\vert
\varphi \left( y\right) \right\vert +\left\vert \mathbf{K}^{\alpha }\left(
x,y\right) \right\vert \ \left\vert \nabla _{y}\varphi \left( y\right)
\right\vert \\
&\leq &\left\Vert \mathbf{K}^{\alpha }\right\Vert _{CZ_{\alpha }}\left\vert
x-y\right\vert ^{\alpha -n-1}+\left\Vert \mathbf{K}^{\alpha }\right\Vert
_{CZ_{\alpha }}\left\vert x-y\right\vert ^{\alpha -n}\left\vert \nabla
\varphi \left( y\right) \right\vert \\
&\leq &\left\Vert \mathbf{K}^{\alpha }\right\Vert _{CZ_{\alpha }}\left(
1+\left( 5\sqrt{n}\right) ^{n-\alpha }C_{\varphi }\right) \left\vert
x-y\right\vert ^{\alpha -n-1}\lesssim \left\Vert \mathbf{K}^{\alpha
}\right\Vert _{CZ_{\alpha }}\left\vert x-y\right\vert ^{\alpha -n-1},
\end{eqnarray*}%
Similarly $\left\vert \nabla _{y}^{2}\mathbf{1}_{Q}\left( x\right) \mathbf{K}%
^{\alpha }\left( x,y\right) \varphi \left( y\right) \right\vert \lesssim
\left\Vert \mathbf{K}^{\alpha }\right\Vert _{CZ_{\alpha }}\left\vert
x-y\right\vert ^{\alpha -n-2}$.
\end{proof}

Thus the localized kernel $\mathbf{1}_{Q}\left( x\right) \mathbf{K}^{\alpha
}\left( x,y\right) \varphi \left( y\right) $ satisfies Calder\'{o}n-Zygmund
smoothness in the $y$-variable, but it fails to satisfy Calder\'{o}n-Zygmund
smoothness in the $x$-variable. This unfortunate omission diminishes the
significance of the derivation of energy from localized families, but does
help somewhat to narrow the focus on difficulties in obtaining necessity of
energy from boundedness of families of operators.

\subsection{Family of localizations of an operator}

For any $\alpha $-fractional singular integral operator $\mathbf{T}^{\alpha
} $ with kernel $\mathbf{K}^{\alpha }\left( x,y\right) $, and any cube $J$
with center $c_{J}$ and side length $\ell \left( J\right) $, we consider the
vector operator $\mathbf{T}_{J}^{\alpha }$ with kernel%
\begin{eqnarray*}
\mathbf{K}_{J}^{\alpha }\left( x,y\right) &\equiv &\mathbf{1}_{J}\left(
x\right) \ \mathbf{K}^{\alpha }\left( x,y\right) \ \varphi _{J}\left(
y\right) ; \\
\varphi _{J}\left( y\right) &=&\varphi \left( \frac{y-c_{J}}{\ell \left(
J\right) }\right) ,
\end{eqnarray*}%
which we refer to as a localization of $\mathbf{T}^{\alpha }$ to the cube $J$
and sector $\widehat{J}$, where $\widehat{J}=\delta _{\ell \left( J\right) }%
\widehat{Q}+c_{J}$ is the dilate by $\ell \left( J\right) $ and translate by 
$c_{J}$ of the unit sector $\widehat{Q}$ with aperture $\theta $ defined
above.

Now we define the operator $\mathbf{T}_{J}^{\alpha }\Theta _{j}^{-1}$ with
kernel%
\begin{equation*}
\mathbf{K}_{J}^{\alpha }\Theta _{j}^{-1}\left( x,y\right) \equiv \mathbf{1}%
_{J}\left( x\right) \ \mathbf{K}^{\alpha }\left( x,y\right) \ \varphi
_{J}\left( \left( \Theta _{j}^{J}\right) ^{-1}y\right) ,
\end{equation*}%
but where we must of course use $\varphi _{J}^{j}$ in place of $\varphi _{J}$
for each $1\leq j\leq M$, since cubes are not invariant under rotations. As
mentioned earlier, we will typically suppress the superscript $j$ here. This
operator $\mathbf{T}_{J}^{\alpha }\Theta _{j}^{-1}$ is referred to as a
localization of $\mathbf{T}^{\alpha }$ to the cube $J$ and sector $\widehat{%
\Theta _{j}^{J}J}$, where $\widehat{\Theta _{j}^{J}J}=\delta _{\ell \left(
J\right) }\widehat{\Theta _{j}Q}+c_{J}$ is the dilate by $\ell \left(
J\right) $ and translate by $c_{J}$ of the `rotation' $\widehat{\Theta _{j}Q}
$ of the unit sector $\widehat{Q}$ with aperture $\theta $ (we say
`rotation' despite the fact that this is only approximately true).

Now let $\mathcal{J}=\left\{ J_{k}\right\} _{k=1}^{\infty }$ be a sequence
of pairwise disjoint subcubes of a cube $I$ satisfying the bounded overlap
condition,%
\begin{equation}
\sum_{r=1}^{\infty }\mathbf{1}_{\gamma I_{r}}\lesssim \beta \mathbf{1}_{I}\ ,
\label{overlap}
\end{equation}%
and define the vector operators%
\begin{eqnarray*}
\mathbf{T}_{\mathcal{J}}^{\alpha } &=&\sum_{k=1}^{\infty }\mathbf{T}%
_{J_{k}}^{\alpha }, \\
\mathbf{T}_{\mathcal{J}}^{\alpha }\Theta _{j}^{-1} &=&\sum_{k=1}^{\infty }%
\mathbf{T}_{J_{k}}^{\alpha }\Theta _{j}^{-1},
\end{eqnarray*}%
to have kernels%
\begin{eqnarray}
\mathbf{K}_{\mathcal{J}}^{\alpha }\left( x,y\right) &=&\sum_{k=1}^{\infty }%
\mathbf{1}_{J_{k}}\left( x\right) \ \mathbf{K}^{\alpha }\left( x,y\right) \
\varphi _{J_{k}}\left( y\right) ,  \label{kernels} \\
\mathbf{K}_{\mathcal{J}}^{\alpha }\Theta _{j}^{-1}\left( x,y\right)
&=&\sum_{k=1}^{\infty }\mathbf{1}_{J_{k}}\left( x\right) \ \mathbf{K}%
^{\alpha }\left( x,y\right) \ \varphi _{J_{k}}\left( \left( \Theta
_{j}^{J_{k}}\right) ^{-1}y\right) .  \notag
\end{eqnarray}%
respectively. Here, for any cube $J$, 
\begin{equation*}
\Theta _{i}^{J}\left( x\right) \equiv \Theta _{i}\left( x-c_{J}\right)
+c_{J},\ \ \ \ \ x\in \mathbb{R}^{n},
\end{equation*}%
is the conjugation by translation by $c_{J}$ of the rotation $\Theta _{i}$,
resulting in a rotation about the point $c_{J}$. We refer to the operator $%
\mathbf{T}_{\mathcal{J}}^{\alpha }\Theta _{j}^{-1}$ as the localization
rotated by $\Theta _{j}$ of $\mathbf{T}^{\alpha }$ to the collection $%
\mathcal{J}$. Denote by $\mathfrak{J}$, the infinite family of such
collections of cubes, namely those collections $\mathcal{J}$ of pairwise
disjoint subcubes of a cube $I$, whose expansions have bounded overlap (\ref%
{overlap}). The corresponding infinite family of operators $\left\{ \mathbf{T%
}_{\mathcal{J}}^{\alpha }\Theta _{j}^{-1}\right\} _{\mathcal{J}\in \mathfrak{%
J}\text{ and }1\leq j\leq M}$, taken over all cubes $I$ and decompositions $%
\mathcal{J}$ satisfying (\ref{overlap}) and all $1\leq j\leq M$, is called
the \emph{family of localizations} of the operator $\mathbf{T}^{\alpha }$.
The kernels $\left\{ \mathbf{K}_{\mathcal{J}}^{\alpha }\Theta
_{j}^{-1}\right\} _{\mathcal{J}\in \mathfrak{J}\text{ and }1\leq j\leq M}$
uniformly satisfy a \emph{one-sided} Calder\'{o}n-Zygmund condition (in the $%
y$-variable only).

\begin{lemma}
\label{CZalpha}Let $\mathbf{K}_{\mathcal{J}}^{\alpha }\Theta _{j}^{-1}$ be
as in the second line of (\ref{kernels}). Then for all $\mathcal{J}\in 
\mathfrak{J}$ and $1\leq j\leq M$ we have 
\begin{eqnarray}
\left\vert \mathbf{K}_{\mathcal{J}}^{\alpha }\Theta _{j}^{-1}\left(
x,y\right) \right\vert &\lesssim &\left\Vert \mathbf{K}^{\alpha }\right\Vert
_{CZ_{\alpha }}\left\vert x-y\right\vert ^{\alpha -n},  \label{CZ est} \\
\left\vert \nabla _{y}^{\ell }\mathbf{K}_{\mathcal{J}}^{\alpha }\Theta
_{j}^{-1}\left( x,y\right) \right\vert &\lesssim &\left\Vert \mathbf{K}%
^{\alpha }\right\Vert _{CZ_{\alpha }}\left\vert x-y\right\vert ^{\alpha
-n-\ell },\ \ \ \ \ \ell =1,2.  \notag
\end{eqnarray}
\end{lemma}

\begin{proof}
The first line in (\ref{CZ est}) is automatic since the cubes $J_{k}$ are
pairwise disjoint:%
\begin{equation*}
\left\vert \mathbf{K}_{\mathcal{J}}^{\alpha }\Theta _{1}^{-1}\left(
x,y\right) \right\vert =\left\vert \sum_{k=1}^{\infty }\mathbf{1}%
_{J_{k}}\left( x\right) \ \mathbf{K}^{\alpha }\left( x,y\right) \ \varphi
_{J_{k}}\left( y\right) \right\vert \leq \left\vert \mathbf{K}^{\alpha
}\left( x,y\right) \right\vert \leq \left\Vert \mathbf{K}^{\alpha
}\right\Vert _{CZ_{\alpha }}\left\vert x-y\right\vert ^{\alpha -n}.
\end{equation*}%
Now note that 
\begin{eqnarray*}
\left\vert \nabla _{y}\varphi _{J}\left( y\right) \right\vert &=&\left\vert
\nabla _{y}\left[ \varphi \left( \frac{y-c_{J}}{\ell \left( J\right) }%
\right) \right] \right\vert =\left\vert \nabla \varphi \left( \frac{y-c_{J}}{%
\ell \left( J\right) }\right) \frac{1}{\ell \left( J\right) }\right\vert \\
&\leq &C_{\varphi }\left\vert \frac{y-c_{J}}{\ell \left( J\right) }%
\right\vert ^{-1}\frac{1}{\ell \left( J\right) }=C_{\varphi }\left\vert
y-c_{J}\right\vert ^{-1}.
\end{eqnarray*}%
For the second line we may suppose without loss of generality that $j=1$ so
that $\left( \Theta _{1}^{J_{k}}\right) ^{-1}$ is the identity rotation
about $c_{J_{k}}$, i.e. the identity map, and thus $\mathbf{K}_{\mathcal{J}%
}^{\alpha }\Theta _{1}^{-1}=\mathbf{K}_{\mathcal{J}}^{\alpha }$. If $\mathbf{%
K}_{\mathcal{J}}^{\alpha }\left( x,y\right) \neq 0$, then $x\in J_{k}$ for a
unique $k\geq 1$ and%
\begin{equation*}
3\sqrt{n}\left\vert y-c_{J_{k}}\right\vert \leq \left\vert x-y\right\vert
\leq 5\sqrt{n}\left\vert y-c_{J_{k}}\right\vert ,
\end{equation*}%
and then we have%
\begin{eqnarray*}
\left\vert \nabla _{y}\mathbf{K}_{\mathcal{J}}^{\alpha }\Theta
_{j}^{-1}\left( x,y\right) \right\vert &=&\left\vert \mathbf{1}%
_{J_{k}}\left( x\right) \ \nabla _{y}\left\{ \mathbf{K}^{\alpha }\left(
x,y\right) \ \widehat{\varphi }_{J_{k}}\left( y\right) \right\} \right\vert
\\
&\leq &\left\vert \nabla _{y}\mathbf{K}^{\alpha }\left( x,y\right)
\right\vert \ \left\vert \widehat{\varphi }_{J_{k}}\left( y\right)
\right\vert +\left\vert \mathbf{K}^{\alpha }\left( x,y\right) \right\vert \
\left\vert \nabla _{y}\varphi _{J_{k}}\left( y\right) \right\vert \\
&\lesssim &\left\Vert \mathbf{K}^{\alpha }\right\Vert _{CZ_{\alpha
}}\left\vert x-y\right\vert ^{\alpha -n-1}+\left\Vert \mathbf{K}^{\alpha
}\right\Vert _{CZ_{\alpha }}\left\vert x-y\right\vert ^{\alpha -n}C_{\varphi
}\left\vert y-c_{J_{k}}\right\vert ^{-1} \\
&\lesssim &\left\Vert \mathbf{K}^{\alpha }\right\Vert _{CZ_{\alpha
}}\left\vert x-y\right\vert ^{\alpha -n-1}.
\end{eqnarray*}%
Similarly we have $\left\vert \nabla _{y}^{2}\mathbf{K}_{\mathcal{J}%
}^{\alpha }\Theta _{j}^{-1}\left( x,y\right) \right\vert \lesssim \left\Vert 
\mathbf{K}^{\alpha }\right\Vert _{CZ_{\alpha }}\left\vert x-y\right\vert
^{\alpha -n-2}$.
\end{proof}

\subsection{Family of twisted localizations}

In order to derive the deep energy condition, it is \emph{not} enough to
assume the uniform boundedness of the family $\left\{ \mathbf{T}_{\mathcal{J}%
}^{\alpha }\Theta _{j}^{-1}\right\} _{\mathcal{J}\in \mathfrak{J}}$ of
localizations of $\mathbf{T}^{\alpha }$, see Remark \ref{loc unif} at the
end of the paper, but rather we must assume uniform boundedness of the
larger family $\left\{ \Theta _{i}\mathbf{T}_{\mathcal{J}}^{\alpha }\Theta
_{j}^{-1}\right\} _{\mathcal{J}\in \mathfrak{J}}$ of \emph{twisted
localizations} of $\mathbf{T}^{\alpha }$ given by%
\begin{eqnarray}
\left[ \Theta _{i}\mathbf{T}_{\mathcal{J}}^{\alpha }\Theta _{j}^{-1}\right]
_{\sigma }f\left( x\right) &\equiv &\int \Theta _{i}\mathbf{K}_{\mathcal{J}%
}^{\alpha }\Theta _{j}^{-1}\left( x,y\right) \ f\left( y\right) d\sigma
\left( y\right) ,  \label{kernel twist} \\
\Theta _{i}\mathbf{K}_{\mathcal{J}}^{\alpha }\Theta _{j}^{-1}\left(
x,y\right) &\equiv &\sum_{J\in \mathcal{J}}\mathbf{1}_{J}\left( x\right) \ 
\mathbf{K}^{\alpha }\left( \Theta _{i}^{J}x,y\right) \ \varphi _{J}\left(
\left( \Theta _{j}^{J}\right) ^{-1}y\right) ,  \notag
\end{eqnarray}%
where we have pre-rotated the kernel by a rotation $\Theta _{j}^{J}$
centered at $c_{J}$, and post-rotated the kernel by a rotation $\Theta
_{i}^{J}$ centered at $c_{J}$. For a single cube $J$, we refer to $\Theta
_{i}\mathbf{T}_{J}^{\alpha }\Theta _{j}^{-1}$ as a \emph{twisted localization%
} of $\mathbf{T}^{\alpha }$ to the cube $J$ and sector $\widehat{J}%
=c_{J}+\Theta _{j}\widehat{Q}$, which is \emph{twisted} by the post-
rotation $\Theta _{i}$. For a collection of cubes $\mathcal{J}\in \mathfrak{J%
}$, we refer to the infinite sum $\Theta _{i}\mathbf{T}_{\mathcal{J}%
}^{\alpha }\Theta _{j}^{-1}\equiv \sum_{J\in \mathcal{J}}\Theta _{i}\mathbf{T%
}_{J_{k}}^{\alpha }\Theta _{j}^{-1}$ as a \emph{twisted localization} of $%
\mathbf{T}^{\alpha }$ to the collection of cubes $\mathcal{J}$. Finally, we
then refer to the family of operators $\left\{ \Theta _{i}\mathbf{T}_{%
\mathcal{J}}^{\alpha }\Theta _{j}\right\} _{\mathcal{J}\in \mathfrak{J}\text{
and }1\leq i,j\in M}$ as the family of \emph{twisted localizations} of the
operator $\mathbf{T}^{\alpha }$. Again, using $\left\vert \Theta
_{i}^{J}x-c_{J}\right\vert =\left\vert x-c_{J}\right\vert $ together with
the argument for the localized kernels $\mathbf{K}_{\mathcal{J}}^{\alpha
}\Theta _{j}^{-1}$in the proof of Lemma \ref{CZalpha} above, it is easy to
obtain a one-sided Calder\'{o}n-Zygmund kernel estimate for the twisted
localizations.

\begin{lemma}
\label{CZalpha twist}Let $\Theta _{i}\mathbf{K}_{\mathcal{J}}^{\alpha
}\Theta _{j}^{-1}$ be as in the second line of (\ref{kernel twist}). Then 
\begin{eqnarray*}
\left\vert \Theta _{i}\mathbf{K}_{\mathcal{J}}^{\alpha }\Theta
_{j}^{-1}\left( x,y\right) \right\vert &\lesssim &\left\Vert \mathbf{K}%
^{\alpha }\right\Vert _{CZ_{\alpha }}\left\vert x-y\right\vert ^{\alpha -n},
\\
\left\vert \nabla _{y}^{\ell }\Theta _{i}\mathbf{K}_{\mathcal{J}}^{\alpha
}\Theta _{j}^{-1}\left( x,y\right) \right\vert &\lesssim &\left\Vert \mathbf{%
K}^{\alpha }\right\Vert _{CZ_{\alpha }}\left\vert x-y\right\vert ^{\alpha
-n-\ell },\ \ \ \ \ \ell =1,2.
\end{eqnarray*}
\end{lemma}

In applications to the necessity of the strong energies $\mathcal{E}%
_{2}^{\alpha }$ and $\mathcal{E}_{2}^{\alpha ,\ast }$\ in Definition \ref%
{def strong quasienergy}, one would take $\mathcal{J}=\left\{ I_{r}\right\}
_{r=1}^{\infty }$.

Even more generally, given a sequence $\mathcal{J}=\left\{ J_{k}\right\}
_{k=1}^{\infty }$ of pairwise disjoint subcubes of a cube $I$ satisfying (%
\ref{overlap}), \textbf{and} a choice of pre- and post-rotations $\widetilde{%
\Theta }_{\limfunc{pre}}\equiv \left\{ \Theta _{j_{k}}\right\}
_{k=1}^{\infty }$ and $\widetilde{\Theta }_{\limfunc{post}}\equiv \left\{
\Theta _{i_{k}}\right\} _{k=1}^{\infty }$, we define the vector operator%
\begin{equation*}
\widetilde{\Theta }_{\limfunc{pre}}\mathbf{T}_{\mathcal{J}}^{\alpha }%
\widetilde{\Theta }_{\limfunc{post}}f=\sum_{k=1}^{\infty }\left[ \Theta
_{i_{k}}\mathbf{T}_{J_{k}}^{\alpha }\Theta _{j_{k}}^{-1}\right] f,\ \ \ \ \
1\leq i,j\leq M,
\end{equation*}%
which has kernel%
\begin{equation*}
\widetilde{\Theta }_{\limfunc{pre}}\mathbf{K}_{\mathcal{J}}^{\alpha }%
\widetilde{\Theta }_{\limfunc{post}}\left( x,y\right) =\sum_{k=1}^{\infty
}\varphi _{J_{k}}\left( x\right) \ \mathbf{K}^{\alpha }\left( \Theta
_{i_{k}}^{J_{k}}x,y\right) \ \varphi _{J}\left( \left( \Theta
_{j_{k}}^{J_{k}}\right) ^{-1}y\right) ,
\end{equation*}%
whose rotations now vary with the subcube $J_{k}$. We will show that for
appropriate operators $\mathbf{T}^{\alpha }$, including the Riesz transform
vector $\mathbf{R}^{\alpha ,n}$, we can actually use reversal of energy for
the \emph{single} operator $\widetilde{\Theta }_{\limfunc{pre}}\mathbf{T}_{%
\mathcal{J}}^{\alpha }\widetilde{\Theta }_{\limfunc{post}}$ to deduce the 
\emph{single} inequality%
\begin{equation*}
\sum_{k=1}^{\infty }\left( \frac{\mathrm{P}^{\alpha }\left( J_{k},\mathbf{1}%
_{I}\sigma \right) }{\left\vert J_{k}\right\vert ^{\frac{1}{n}}}\right)
^{2}\left\Vert \mathsf{P}_{J_{k}}^{\omega }\mathbf{x}\right\Vert
_{L^{2}\left( \omega \right) }^{2}\leq \left( \left( \mathfrak{T}_{%
\widetilde{\Theta }_{\limfunc{pre}}\mathbf{T}_{\mathcal{J}}^{\alpha }%
\widetilde{\Theta }_{\limfunc{post}}}\right) ^{2}+\mathcal{A}_{2}^{\alpha
}\right) \left\vert I\right\vert _{\sigma }\ ,
\end{equation*}%
when $\mathcal{J}$ is taken to be $\left\{ J_{k}\right\} _{k=1}^{\infty }$,
and $\widetilde{\Theta }_{\limfunc{pre}}$ and $\widetilde{\Theta }_{\limfunc{%
post}}$ are chosen appropriately depending on $\sigma $ and $\omega $
respectively.

\subsection{Reversal of energy}

Fix a cube $J_{k}$ and indices $1\leq i,j\leq M$. Let $\theta =\theta _{%
\mathbf{T}^{\alpha }}$ be the angle in the positive gradient property for
the operator $\mathbf{T}^{\alpha }$, and set $\lambda =\tan \theta $. Then
set $B_{\mathbb{S}^{n-1}}\equiv B_{\mathbb{S}^{n-1}}\left( \mathbf{e}%
_{1},\theta \right) $ and take $x,z\in J_{k}$ with $x-z\in \Theta
_{i}^{-1}\Theta _{j}B_{\mathbb{S}^{n-1}}$ so that 
\begin{eqnarray*}
&&\Theta _{i}^{J_{k}}x-\Theta _{i}^{J_{k}}z=\Theta _{i}\left( x-z\right) \in
\Theta _{j}B_{\mathbb{S}^{n-1}}, \\
&&\Theta _{i}^{J_{k}}x-y\text{ and }\Theta _{i}^{J_{k}}z-y\in \Theta _{j}B_{%
\mathbb{S}^{n-1}}.
\end{eqnarray*}%
Let $p=\Theta _{i}^{J_{k}}x$ and $q=\Theta _{i}^{J_{k}}z$. Without loss of
generality we can take $\Theta _{j}=Id$ the identity for this argument. Then
for $x=\left( x^{1},\widetilde{x}\right) $ and $z=\left( z^{1},\widetilde{z}%
\right) $ in $J_{k}$ with $\left\vert \widetilde{x}-\widetilde{z}\right\vert
\leq \lambda \left\vert x_{1}-z_{1}\right\vert $ (equivalently $\frac{x-z}{%
\left\vert x-z\right\vert }\in B_{\mathbb{S}^{n-1}}$), we claim the
following `strong reversal' of energy. Since $\mathbf{1}_{J_{k}}\left(
x\right) =1=\mathbf{1}_{J_{k}}\left( z\right) $, we can compute%
\begin{eqnarray*}
&&\frac{\left[ \left( \Theta _{i}\mathbf{T}_{J_{k}}^{\alpha }\Theta
_{j}\right) _{\sigma }\right] _{1}\mathbf{1}_{I}\left( x\right) -\left[
\left( \Theta _{i}\mathbf{T}_{J_{k}}^{\alpha }\Theta _{j}\right) _{\sigma }%
\right] _{1}\mathbf{1}_{I}\left( z\right) }{\left[ \Theta _{i}\left(
x-z\right) \right] _{1}} \\
&=&\int \left\{ \frac{\mathbf{1}_{J_{k}}\left( x\right) \mathbf{K}^{\alpha
}\left( p,y\right) -\mathbf{1}_{J_{k}}\left( z\right) \mathbf{K}^{\alpha
}\left( q,y\right) }{\left[ \Theta _{i}^{J_{k}}x\right] _{1}-\left[ \Theta
_{i}^{J_{k}}z\right] _{1}}\right\} \widehat{\varphi }\left( \frac{y-c_{J_{k}}%
}{\ell \left( J_{k}\right) }\right) \mathbf{1}_{I}\left( y\right) d\sigma
\left( y\right) \\
&=&\int \left\{ \frac{K_{1}^{\alpha }\left( p,y\right) -K_{1}^{\alpha
}\left( q,y\right) }{p_{1}-q_{1}}\right\} \widehat{\varphi }\left( \frac{%
y-c_{J_{k}}}{\ell \left( J_{k}\right) }\right) \mathbf{1}_{I}\left( y\right)
d\sigma \left( y\right) ,
\end{eqnarray*}%
and since $K_{1}^{\alpha }$ is a convolution operator, the term in braces
satisfies%
\begin{eqnarray*}
\frac{K_{1}^{\alpha }\left( p-y\right) -K_{1}^{\alpha }\left( q-y\right) }{%
p_{1}-q_{1}} &=&\frac{K_{1}^{\alpha }\left( p_{1}-y_{1},\widetilde{p}-%
\widetilde{y}\right) -K_{1}^{\alpha }\left( q_{1}-y_{1},\widetilde{q}-%
\widetilde{y}\right) }{p_{1}-q_{1}} \\
&=&\frac{K_{1}^{\alpha }\left( s,\widetilde{p}-\widetilde{y}\right)
-K_{1}^{\alpha }\left( t,\widetilde{q}-\widetilde{y}\right) }{s-t}
\end{eqnarray*}%
with $y=\left( y_{1},\widetilde{y}\right) \in \mathbb{R}\times \mathbb{R}%
^{n-1}$, $s=p_{1}-y_{1}$ and $t=q_{1}-y_{1}$. Here $K_{1}^{\alpha }\left(
\xi \right) =K_{1}^{\alpha }\left( \xi _{1},\widetilde{\xi }\right) $ is the
first component of the convolution kernel $\mathbf{K}^{\alpha }\left( \xi
\right) $ for $\xi =\left( \xi _{1},\widetilde{\xi }\right) $.

Now we invoke the \emph{positive gradient property} of $K_{1}^{\alpha }$: 
\begin{equation}
\frac{K_{1}^{\alpha }\left( \xi \right) -K_{1}^{\alpha }\left( \eta \right) 
}{\xi _{1}-\eta _{1}}\approx \left\vert \xi \right\vert ^{\alpha -n-1},\ \ \
\ \ \text{for }\xi ,\eta \in S\text{ with }\frac{\left\vert \widetilde{\xi }-%
\widetilde{\eta }\right\vert }{\left\vert \xi _{1}-\eta _{1}\right\vert }%
\leq \lambda .  \label{grad pos prop}
\end{equation}%
In particular, we then have%
\begin{equation*}
\frac{K_{1}^{\alpha }\left( s,\widetilde{p}-\widetilde{y}\right)
-K_{1}^{\alpha }\left( t,\widetilde{q}-\widetilde{y}\right) }{s-t}\approx
\left\vert \left( s,\widetilde{p}-\widetilde{y}\right) \right\vert ^{\alpha
-n-1}\approx \left\vert c_{J_{k}}-y\right\vert ^{\alpha -n-1},
\end{equation*}%
since%
\begin{equation*}
\widetilde{p}-\widetilde{y}=\widetilde{\Theta _{i}^{J_{k}}x}-\widetilde{y}=%
\widetilde{\Theta _{i}\left( x-c_{J_{k}}\right) }+\widetilde{c_{J_{k}}}-%
\widetilde{y}
\end{equation*}%
satisfies%
\begin{eqnarray*}
\left\Vert \widetilde{p}-\widetilde{y}\right\Vert &\leq &\left\Vert 
\widetilde{\Theta _{i}\left( x-c_{J_{k}}\right) }\right\Vert +\left\Vert 
\widetilde{c_{J_{k}}}-\widetilde{y}\right\Vert \\
&\leq &\left\Vert x-c_{J_{k}}\right\Vert +\left\Vert \widetilde{c_{J}}-%
\widetilde{y}\right\Vert \lesssim \lambda \left\vert p_{1}-y_{1}\right\vert ,
\end{eqnarray*}%
and similarly $\left\Vert \widetilde{q}-\widetilde{y}\right\Vert \lesssim
\lambda \left\vert p_{1}-y_{1}\right\vert $.

Thus we have%
\begin{eqnarray*}
&&\frac{\left[ \left( \Theta _{i}\mathbf{T}_{J_{k}}^{\alpha }\Theta
_{j}\right) _{\sigma }\right] _{1}\mathbf{1}_{I}\left( x\right) -\left[
\left( \Theta _{i}\mathbf{T}_{J_{k}}^{\alpha }\Theta _{j}\right) _{\sigma }%
\right] _{1}\mathbf{1}_{I}\left( z\right) }{\left[ \Theta _{i}\left(
x-z\right) \right] _{1}} \\
&=&\frac{K_{1}^{\alpha }\left( s,\widetilde{p}-\widetilde{y}\right)
-K_{1}^{\alpha }\left( t,\widetilde{q}-\widetilde{y}\right) }{s-t}\approx
\left\vert c_{J_{k}}-y\right\vert ^{\alpha -n-1},
\end{eqnarray*}%
and so in general,%
\begin{eqnarray*}
&&\left\vert \frac{\left[ \left( \Theta _{i}\mathbf{T}_{J_{k}}^{\alpha
}\Theta _{j}\right) _{\sigma }\right] _{1}\mathbf{1}_{I}\left( x\right) -%
\left[ \left( \Theta _{i}\mathbf{T}_{J_{k}}^{\alpha }\Theta _{j}\right)
_{\sigma }\right] _{1}\mathbf{1}_{I}\left( z\right) }{\left[ \Theta
_{i}\left( x-z\right) \right] _{1}}\right\vert \gtrsim \frac{\mathrm{P}%
^{\alpha }\left( J_{k},\mathbf{1}_{\Theta _{j}\widehat{Q}}\right) }{%
\left\vert J_{k}\right\vert ^{\frac{1}{n}}}, \\
&&\ \ \ \ \ \ \ \ \ \ \text{for }x,z\in J_{k}\text{ with }\frac{x-z}{%
\left\vert x-z\right\vert }\in \Theta _{i}B_{\mathbb{S}^{n-1}}.
\end{eqnarray*}

Thus\ with $\Phi _{i}\equiv \left\{ \left( x,z\right) \in \mathbb{R}%
^{n}\times \mathbb{R}^{n}:\frac{x-z}{\left\vert x-z\right\vert }\in \Theta
_{i}B_{\mathbb{S}^{n-1}}\right\} $ and%
\begin{equation*}
\mathsf{F}_{i}\left( J_{k},\omega \right) ^{2}\equiv \frac{1}{\left\vert
J_{k}\right\vert _{\omega }}\int \int_{J_{k}\times J_{k}\cap \Phi
_{i}}\left\vert x-z\right\vert ^{2}d\omega \left( x\right) d\omega \left(
z\right) ,
\end{equation*}%
we have%
\begin{equation*}
\frac{1}{\left\vert J_{k}\right\vert _{\omega }}\int \int_{J_{k}\times
J_{k}\cap \Phi _{i}}\left\vert x-z\right\vert ^{2}d\omega \left( x\right)
d\omega \left( z\right) =\sum_{i=1}^{N}\mathsf{F}_{i}\left( J_{k},\omega
\right) ^{2}
\end{equation*}%
and so%
\begin{eqnarray*}
&&\sum_{k=1}^{\infty }\left( \frac{\mathrm{P}^{\alpha }\left( J_{k},\mathbf{1%
}_{c_{J_{k}}+\Theta _{j}\widehat{Q}}\mathbf{1}_{I}\sigma \right) }{%
\left\vert J_{k}\right\vert ^{\frac{1}{n}}}\right) ^{2}\left( \frac{1}{%
\left\vert J_{k}\right\vert _{\omega }}\int \int_{J_{k}\times
J_{k}}\left\vert x-z\right\vert ^{2}d\omega \left( x\right) d\omega \left(
z\right) \right) \\
&=&\sum_{i=1}^{N}\sum_{k=1}^{\infty }\left( \frac{\mathrm{P}^{\alpha }\left(
J_{k},\mathbf{1}_{c_{J_{k}}+\Theta _{j}\widehat{Q}}\mathbf{1}_{I}\sigma
\right) }{\left\vert J_{k}\right\vert ^{\frac{1}{n}}}\right) ^{2}\mathsf{F}%
_{i}\left( J_{k},\omega \right) ^{2} \\
&\lesssim &\sum_{i=1}^{N}\sum_{k=1}^{\infty }\frac{1}{\left\vert
J_{k}\right\vert _{\omega }}\int \int_{J_{k}\times J_{k}\cap \Phi
_{i}}\left\vert \left[ \left( \Theta _{i}\mathbf{T}_{J_{k}}^{\alpha }\Theta
_{j}\right) _{\sigma }\right] _{1}\mathbf{1}_{I}\left( x\right) -\left[
\left( \Theta _{i}\mathbf{T}_{J_{k}}^{\alpha }\Theta _{j}\right) _{\sigma }%
\right] _{1}\mathbf{1}_{I}\left( z\right) \right\vert ^{2} \\
&\lesssim &\sum_{i=1}^{N}\sum_{k=1}^{\infty }\int_{J_{k}}\left\vert \left[
\left( \Theta _{i}\mathbf{T}_{J_{k}}^{\alpha }\Theta _{j}\right) _{\sigma }%
\right] _{1}\mathbf{1}_{I}\left( x\right) \right\vert ^{2}\lesssim
\sup_{1\leq i\leq N}\left( \mathfrak{T}_{\Theta _{i}\mathbf{T}_{\mathcal{J}%
}^{\alpha }\Theta _{j}}\right) ^{2}\left\vert I\right\vert _{\sigma }\ .
\end{eqnarray*}%
Finally then using $\mathbf{1}_{I\setminus \gamma J_{k}}\leq \sum_{j=1}^{M}%
\mathbf{1}_{c_{J_{k}}+\Theta _{j}\widehat{Q}}\mathbf{1}_{I}$, we have%
\begin{eqnarray*}
&&\sum_{k=1}^{\infty }\left( \frac{\mathrm{P}^{\alpha }\left( J_{k},\mathbf{1%
}_{I\setminus \gamma J_{k}}\sigma \right) }{\left\vert J_{k}\right\vert ^{%
\frac{1}{n}}}\right) ^{2}\left\Vert \mathsf{P}_{J_{k}}^{\omega }x\right\Vert
_{L^{2}\left( \omega \right) }^{2} \\
&\lesssim &\sup_{1\leq j\leq N}\sum_{k=1}^{\infty }\left( \frac{\mathrm{P}%
^{\alpha }\left( J_{k},\mathbf{1}_{c_{J_{k}}+\Theta _{j}\widehat{Q}}\mathbf{1%
}_{I}\sigma \right) }{\left\vert J_{k}\right\vert ^{\frac{1}{n}}}\right)
^{2}\left( \frac{1}{\left\vert J_{k}\right\vert _{\omega }}\int
\int_{J_{k}\times J_{k}}\left\vert x-z\right\vert ^{2}d\omega \left(
x\right) d\omega \left( z\right) \right) \\
&\lesssim &\sup_{1\leq i,j\leq N}\left( \mathfrak{T}_{\Theta _{i}\mathbf{T}_{%
\mathcal{J}}^{\alpha }\Theta _{j}}\right) ^{2}\left\vert I\right\vert
_{\sigma }\ ,
\end{eqnarray*}%
which proves the forward deep and bounded overlap energy conditions with 
\begin{equation*}
\mathcal{E}_{2}^{\alpha ,\limfunc{deep}}\leq \mathcal{E}_{2}^{\alpha ,%
\limfunc{overlap}}\lesssim \sup_{\mathcal{J}}\sup_{1\leq i,j\leq N}\mathfrak{%
T}_{\Theta _{i}\mathbf{T}_{\mathcal{J}}^{\alpha }\Theta _{j}},
\end{equation*}%
where the supremum in $\mathcal{J}$ is taken over all sequences $\left\{
J_{k}\right\} _{k=1}^{\infty }$ of subcubes of $I$ such that $%
\sum_{k=1}^{\infty }\mathbf{1}_{J_{k}^{\ast }}\leq \beta \mathbf{1}_{I}$.
Indeed, we have%
\begin{equation*}
\sup_{I\supset \dot{\cup}_{r=1}^{\infty }I_{r}}\sup_{\ell \geq
0}\sum_{r=1}^{\infty }\sum_{J\in \mathcal{M}_{\limfunc{deep}}^{\ell }\left(
I_{r}\right) }\left( \frac{\mathrm{P}^{\alpha }\left( J,\mathbf{1}_{I}\sigma
\right) }{\left\vert J\right\vert ^{\frac{1}{n}}}\right) ^{2}\left\Vert 
\mathsf{P}_{J}^{\omega }\mathbf{x}\right\Vert _{L^{2}\left( \omega \right)
}^{2}\leq \left\{ \sup_{\mathcal{J}}\sup_{1\leq i,j\leq N}\left( \mathfrak{T}%
_{\Theta _{i}\mathbf{T}_{\mathcal{J}}^{\alpha }\Theta _{j}}\right)
^{2}+\beta A_{2}^{\alpha }\right\} \left\vert I\right\vert _{\sigma }\ ,
\end{equation*}%
after writing $\mathbf{1}_{I}=\mathbf{1}_{I\setminus \gamma J}+\mathbf{1}%
_{\gamma J}$, and similarly for the bounded overlap energy condition. Thus
we see that the deep and bounded overlap energy constants $\mathcal{E}%
_{2}^{\alpha ,\limfunc{deep}}$ and $\mathcal{E}_{2}^{\alpha ,\limfunc{overlap%
}}$ are controlled by the testing constants $\mathfrak{T}_{\Theta _{i}%
\mathbf{T}_{\mathcal{J}}^{\alpha }\Theta _{j}}$ for the family $\left\{
\Theta _{i}\mathbf{T}_{\mathcal{J}}^{\alpha }\Theta _{j}\right\} _{\mathcal{J%
},i,j}$ of twisted localizations of an operator $\mathbf{T}^{\alpha }$ with
the positive gradient property. Proposition \ref{necessity} is now proved
save for the assertion regarding the Riesz transform, to which we now turn.

\subsubsection{Positive gradient property of the Riesz transform}

Finally we establish the positive gradient property for the negative of the
vector Riesz transform $\mathbf{R}^{\alpha ,n}$ with kernel $\mathbf{K}%
^{\alpha ,n}$.

\begin{lemma}
\label{Riesz pgp}The operator $-\mathbf{R}^{\alpha ,n}$ has the positive
gradient property.
\end{lemma}

\begin{proof}
For this we compute the gradient of the first component $K_{1}^{\alpha
,n}\left( u,w\right) $ for $\left( u,w\right) \in \mathbb{R}\times \mathbb{R}%
^{n-1}$. First, the $u$ partial derivative of $K_{1}^{\alpha ,n}\left(
u,w\right) $ is%
\begin{eqnarray*}
\frac{\partial }{\partial u}K_{1}^{\alpha ,n}\left( u,w\right) &=&\left(
u^{2}+\left\vert w\right\vert ^{2}\right) ^{-\frac{n+1-\alpha }{2}}-\frac{%
n+1-\alpha }{2}\left( u^{2}+\left\vert w\right\vert ^{2}\right) ^{-\frac{%
n+1-\alpha }{2}-1}2u^{2} \\
&=&\left( u^{2}+\left\vert w\right\vert ^{2}\right) ^{-\frac{n+1-\alpha }{2}%
-1}\left\{ \left( u^{2}+\left\vert \xi \right\vert ^{2}\right) -\left(
n+1-\alpha \right) u^{2}\right\} \\
&=&\left( u^{2}+\left\vert w\right\vert ^{2}\right) ^{-\frac{n+1-\alpha }{2}%
-1}\left\{ \left\vert w\right\vert ^{2}-\left( n-\alpha \right)
u^{2}\right\} ,
\end{eqnarray*}%
which satisfies 
\begin{equation*}
\frac{\partial }{\partial u}K_{1}^{\alpha ,n}\left( u,w\right) \approx \frac{%
\left( \alpha -n\right) u^{2}}{\left( u^{2}+\left\vert w\right\vert
^{2}\right) ^{\frac{n+1-\alpha }{2}+1}}\approx \left( \alpha -n\right)
u^{\alpha -n-1}
\end{equation*}%
provided $\left\vert w\right\vert \leq \lambda u$, where $\lambda >0$ is
chosen sufficiently small depending on $\gamma $ and $\rho $. We also have,%
\begin{eqnarray*}
\nabla _{w}K_{1}^{\alpha ,n}\left( u,w\right) &=&\nabla _{w}K_{1}^{\alpha
,n}\left( u^{2}+\left\vert w\right\vert ^{2}\right) ^{-\frac{n+1-\alpha }{2}}
\\
&=&-\frac{n+1-\alpha }{2}u\left( u^{2}+\left\vert \xi \right\vert
^{2}\right) ^{-\frac{n+1-\alpha }{2}-1}2\xi \\
&=&\left( \alpha -n-1\right) uw\left( u^{2}+\left\vert \xi \right\vert
^{2}\right) ^{-\frac{n+1-\alpha }{2}-1},
\end{eqnarray*}%
which satisfies%
\begin{equation*}
\left\vert \nabla _{w}K_{1}^{\alpha ,n}\left( u,w\right) \right\vert
\lesssim \frac{\left\vert uw\right\vert }{\left( u^{2}+\left\vert
w\right\vert ^{2}\right) ^{\frac{n+1-\alpha }{2}+1}}\lesssim \lambda \frac{%
u^{2}}{\left( u^{2}+\left\vert w\right\vert ^{2}\right) ^{\frac{n+1-\alpha }{%
2}+1}},
\end{equation*}%
since $\left\vert w\right\vert \leq \lambda u$.

Altogether then 
\begin{eqnarray*}
&&K_{1}^{\alpha ,n}\left( p_{1}-y_{1},\widetilde{p}-\widetilde{y}\right)
-K_{1}^{\alpha ,n}\left( q_{1}-y_{1},\widetilde{q}-\widetilde{y}\right) \\
&=&K_{1}^{\alpha ,n}\left[ \theta p_{1}+\left( 1-\theta \right)
q_{1}-y_{1},\theta \widetilde{p}+\left( 1-\theta \right) \widetilde{q}-%
\widetilde{y}\right] \mid _{0}^{1} \\
&=&\int_{0}^{1}\frac{d}{d\theta }K_{1}^{\alpha ,n}\left[ \theta p_{1}+\left(
1-\theta \right) q_{1}-y_{1},\theta \widetilde{p}+\left( 1-\theta \right) 
\widetilde{q}-\widetilde{y}\right] \ d\theta \\
&=&\int_{0}^{1}\left( p_{1}-q_{1}\right) \left( \frac{\partial }{\partial u}%
K_{1}^{\alpha ,n}\right) \left[ \theta p_{1}+\left( 1-\theta \right)
q_{1}-y_{1},\theta \widetilde{p}+\left( 1-\theta \right) \widetilde{q}-%
\widetilde{y}\right] \ d\theta \\
&&+\int_{0}^{1}\left( \widetilde{p}-\widetilde{q}\right) \cdot \left( \nabla
_{w}K_{1}^{\alpha ,n}\right) \left[ \theta p_{1}+\left( 1-\theta \right)
q_{1}-y_{1},\theta \widetilde{p}+\left( 1-\theta \right) \widetilde{q}-%
\widetilde{y}\right] \ d\theta ,
\end{eqnarray*}%
and since 
\begin{eqnarray*}
&&\left\vert \theta \widetilde{p}+\left( 1-\theta \right) \widetilde{q}-%
\widetilde{y}\right\vert \lesssim \lambda \left\vert p_{1}-q_{1}\right\vert ,
\\
&&\left\Vert \widetilde{p}-\widetilde{q}\right\Vert \lesssim \lambda
\left\vert p_{1}-q_{1}\right\vert , \\
&&u=\theta p_{1}+\left( 1-\theta \right) q_{1}-y_{1}\approx \left\vert
c_{J_{k}}-y\right\vert ,
\end{eqnarray*}%
the above estimates give%
\begin{eqnarray*}
&&K_{1}^{\alpha ,n}\left( p_{1}-y_{1},\widetilde{p}-\widetilde{y}\right)
-K_{1}^{\alpha ,n}\left( q_{1}-y_{1},\widetilde{q}-\widetilde{y}\right) \\
&\approx &\left( p_{1}-q_{1}\right) \int_{0}^{1}\left( \alpha -n\right)
u^{\alpha -n-1}d\theta +o\left( \frac{u^{2}\left\vert p_{1}-q_{1}\right\vert 
}{\left( u^{2}+\left\vert \xi \right\vert ^{2}\right) ^{\frac{n+1-\alpha }{2}%
+1}}\right) \approx \left( p_{1}-q_{1}\right) \left( \alpha -n\right)
u^{\alpha -n-1},
\end{eqnarray*}%
provided $\lambda >0$ is chosen sufficiently small. This completes the proof
of Lemma \ref{Riesz pgp}.
\end{proof}

We have now completed the proof of Proposition \ref{necessity}.

\part{Failure of necessity of the energy condition for Riesz transforms}

In the second part of this paper, we prove Theorem \ref{counterexample} by
constructing the families of counterexample weight pairs that demonstrate
the failure of necessity of the energy conditions in higher dimensions.

In \cite{LaSaUr2}, the authors constructed a weight pair $\left( \sigma
,\omega \right) $ on the real line which demonstrated that the backward
pivotal condition of NTV was not necessary for boundedness of the Hilbert
transform. This pair was then modified in \cite{SaShUr11}, to demonstrate
failure of necessity of the backward energy condition for boundedness of an
elliptic operator on the line, by `smearing out' the point masses of $\sigma 
$ in order that the backward energy condition became equivalent with the
backward pivotal condition. But this change then destroyed the backward
testing condition for the Hilbert transform, and this necessitated a
flattening of the kernel of the Hilbert transform, along with a delicate
redistribution of the Cantor measure.

In this paper, we instead modify the weight pair $\left( \sigma ,\omega
\right) $ on the real line to obtain a family of weight pairs $\left\{
\left( \widehat{\sigma }_{N},\widehat{\omega }_{N}\right) \right\}
_{N=1}^{\infty }$ in a two-dimensional subspace of $\mathbb{R}^{n}$, which
demonstrate that the energy conditions are not necessary for boundedness of
the vector Riesz transform $\mathbf{R}^{\alpha ,n}$. This modification is
suggested by the above derivation of the energy conditions from the testing
conditions for the family of twisted localizations of $\mathbf{R}^{\alpha
,n} $, and is accomplished by replacing the point masses of $\sigma $ on the
line with a `spread out' pair of point masses extending off the real line
(this is the twist), again resulting in failure of the backward energy
condition. While this spreading out of the point masses in $\sigma $ leaves
intact the testing conditions for the first component $R_{1}^{\alpha ,n}$ of
the Riesz transform, it destroys the backward testing condition for $%
R_{1}^{\alpha ,n}$ - consistent with the fact that the energy conditions 
\textbf{are} necessary for boundedness of $\mathbf{R}^{\alpha ,n}$ when the
measure $\omega $ is supported on a line - see \cite{SaShUr8} and \cite%
{LaSaShUrWi}. In order to circumvent this difficulty, we must carefully
reposition the Cantor measure off the line to occupy the upper and lower
half spaces of $\mathbb{R}^{2}\subset \mathbb{R}^{n}$ in such a way that
point masses associated with the repositioned Cantor measure appear near the
spreadout point masses of $\sigma $. This is needed in order to force zeroes
of the function $R_{1}^{\alpha ,n}\widehat{\omega }_{N}$ to occur where we
want them locally. Since the second component $R_{2}^{\alpha ,n}$ of the
Riesz transform is essentially controlled by the Poisson operator, and the
remaining components $R_{j}^{\alpha ,n}$, $3\leq j\leq n$, vanish on the
supports of these measures, we also obtain the testing conditions for the
remaining $R_{j}^{\alpha ,n}$ when $j\geq 2$.

\section{Construction of the counterexample pair of weights for the Cauchy
operator}

We begin the proof of Theorem \ref{counterexample} with the special case $%
\alpha =1$ in dimension $n=2$, where the components of the fractional Riesz
transform $\mathbf{R}^{1,2}=\left( R_{1}^{1,2},R_{2}^{1,2}\right) $ are the
real and imaginary parts of the Cauchy transform $\mathbf{C}$ with
convolution kernel $\frac{1}{z}$, for $z\in \mathbb{C}$. Note also that the
restriction of the first component $R_{1}^{1,2}$ to the $x$-axis in $\mathbb{%
C}$ is precisely the Hilbert transform $H$ with convolution kernel $\frac{1}{%
x}$ on the real line, which explains the relevance of the one-dimensional
weight pair in \cite{LaSaUr2}. However, it is the additional dimension
available in the plane that allows us to retain boundedness of the operator $%
\mathbf{R}^{1,2}$ while spreading out both measures off the line, and
arranging for the resulting backward energy condition to fail. The general
case $0\leq \alpha <n$ and $n\geq 2$ is considered at the very end of the
paper.

Recall the middle-third Cantor set $\mathsf{E}$ and Cantor measure $\omega $
on the closed unit interval $I_{1}^{0}=\left[ 0,1\right] $. At the $k^{th}$
generation in the construction, there is a collection $\left\{
I_{j}^{k}\right\} _{j=1}^{2^{k}}$ of $2^{k}$ pairwise disjoint closed
intervals of length $\left\vert I_{j}^{k}\right\vert =\frac{1}{3^{k}}$. With 
$K_{k}=\bigcup_{j=1}^{2^{k}}I_{j}^{k}$, the Cantor set is defined by $%
\mathsf{E}=\bigcap_{k=1}^{\infty }K_{k}=\bigcap_{k=1}^{\infty }\left(
\bigcup_{j=1}^{2^{k}}I_{j}^{k}\right) $. The Cantor measure $\omega $ is the
unique probability measure supported in $\mathsf{E}$ with the property that
it is equidistributed among the intervals $\left\{ I_{j}^{k}\right\}
_{j=1}^{2^{k}}$ at each scale $k$, i.e.%
\begin{equation}
\omega (I_{j}^{k})=2^{-k},\ \ \ \ \ k\geq 0,1\leq j\leq 2^{k}.
\label{omega measure}
\end{equation}%
Let $G_{j}^{k}=\left( a_{j}^{k},b_{j}^{k}\right) $ be the open middle third
of $I_{j}^{k}$ and let $\left( I_{j}^{k}\right) _{\limfunc{left}}$ denote
the interval $I_{j_{1}}^{k+1}$ with $j_{1}=2j-1$ that has right hand
endpoint equal to $a_{j}^{k}$, and more generally let $\left\{ I_{j_{\ell
}}^{k+\ell }\right\} _{\ell =1}^{\infty }$ be the tower of intervals with
right hand endpoint $a_{j}^{k}$. Similarly, let $\left( I_{j}^{k}\right) _{%
\limfunc{right}}$ denote the interval $I_{j_{1}+1}^{k+1}=I_{2j}^{k+1}$ that
has left hand endpoint equal to $b_{j}^{k}$, and let $\left\{ I_{j_{\ell
}+1}^{k+\ell }\right\} _{\ell =1}^{\infty }$ be the tower of intervals with
left hand endpoint $b_{j}^{k}$. Let $c_{i}^{k}\in G_{i}^{k}$ be the center
of the interval $G_{i}^{k}=\left( a_{i}^{k},b_{i}^{k}\right) $, which is
also the center of the interval $I_{i}^{k}$.

Now we recall from \cite{LaSaUr2} an important property of the Hilbert
transform $H$ with respect to the Cantor measure $\omega $. We use the
pairwise disjoint decomposition $I_{j_{1}}^{k+1}=\overset{\cdot }{\bigcup }%
_{\ell =1}^{\infty }\left( I_{j_{\ell }}^{k+\ell }\right) _{\limfunc{left}}$
to compute%
\begin{equation*}
H\left( \mathbf{1}_{\left( I_{j}^{k}\right) _{\limfunc{left}}}\omega \right)
\left( a_{j}^{k}\right) =\int_{I_{j_{1}}^{k+1}}\frac{1}{y-a_{j}^{k}}d\omega
\left( y\right) =\sum_{\ell =1}^{\infty }\int_{\left( I_{j_{\ell }}^{k+\ell
}\right) _{\limfunc{left}}}\frac{1}{y-a_{j}^{k}}d\omega \left( y\right) ,
\end{equation*}%
and hence the estimate%
\begin{equation*}
H\left( \mathbf{1}_{\left( I_{j}^{k}\right) _{\limfunc{left}}}\omega \right)
\left( a_{j}^{k}\right) \approx -\sum_{\ell =1}^{\infty }\frac{\left\vert
\left( I_{j_{\ell }}^{k+\ell }\right) _{\limfunc{left}}\right\vert _{\omega }%
}{\left\vert I_{j_{\ell }}^{k+\ell }\right\vert }d\omega \left( y\right)
=-\sum_{\ell =1}^{\infty }\frac{2^{-k-\ell }}{3^{-k-\ell }}=-\sum_{\ell
=1}^{\infty }\left( \frac{3}{2}\right) ^{k+\ell }=-\infty .
\end{equation*}%
Since $H\left( \mathbf{1}_{\left( I_{j}^{k}\right) _{\limfunc{left}%
}^{c}}\omega \right) \left( a_{j}^{k}\right) \lesssim \frac{1}{3^{-k}}%
<\infty $, we conclude that $H\omega \left( a_{j}^{k}\right) =-\infty $, and
similarly $H\omega \left( b_{j}^{k}\right) =\infty $. Thus $H\omega \left(
x\right) $ increases from $-\infty $ to $\infty $ on the interval $G_{j}^{k}$%
. We will later arrange for a similar result to hold for the first component 
$R_{1}^{1,2}$ of the Riesz transform with respect to a modification of $%
\omega $ into the plane.

We now extend certain approximations $\omega _{N}$ of the Cantor measure $%
\omega $ to the plane in the following way. Fix $N\in \mathbb{N}$. Recall
that $K_{N}=\bigcup_{j=1}^{2^{N}}I_{j}^{N}$ and that 
\begin{equation*}
I_{j}^{N}=I_{2j-1}^{N+1}\dot{\cup}G_{j}^{N}\dot{\cup}I_{2j}^{N+1}\equiv I_{j,%
\limfunc{left}}^{N}\dot{\cup}G_{j}^{N}\dot{\cup}I_{j,\limfunc{right}}^{N}\ .
\end{equation*}%
The Cantor measure $\omega $ charges each interval $I_{j,\limfunc{left}}^{N}$
and $I_{j,\limfunc{right}}^{N}$ with the same mass, namely $\left\vert I_{j,%
\limfunc{left}}^{N}\right\vert _{\omega }=\left\vert I_{j,\limfunc{right}%
}^{N}\right\vert _{\omega }=2^{-\left( N+1\right) }$, and we now define the
discrete approximation $\omega _{N}$ by%
\begin{equation*}
\omega _{N}\equiv \sum_{j=1}^{2^{N}}2^{-N-1}\left( \delta _{c_{j,\limfunc{%
left}}^{N}}+\delta _{c_{j,\limfunc{right}}^{N}}\right) ,
\end{equation*}%
where we have relabelled the intervals $I_{j,\limfunc{left}%
}^{N}=I_{2j-1}^{N+1}$ and $I_{j,\limfunc{right}}^{N}=I_{2j}^{N+1}$, and have
denoted their centers by $c_{j,\limfunc{left}}^{N}=c_{2j}^{N+1}$ and $c_{j,%
\limfunc{right}}^{N}=c_{2j}^{N+1}$ respectively.

We now embed the point mass $\delta _{c}$ on $\mathbb{R}$ as $\delta
_{\left( c,0\right) }$in the plane $\mathbb{R}^{2}$, and split each of the
point masses $\delta _{c_{j,\limfunc{left}}^{N}},\delta _{c_{j,\limfunc{right%
}}^{N}}$ for $1\leq j\leq 2^{N}$ into a sum of two point masses located at
equal distances $d_{j,\limfunc{left}}^{N}$ and $d_{j,\limfunc{right}}^{N}$
above and below the points $c_{j,\limfunc{left}}^{N}-d_{j,\limfunc{left}%
}^{N} $ and $c_{j,\limfunc{right}}^{N}+d_{j,\limfunc{right}}^{N}$
respectively. For $\delta _{c_{j,\limfunc{left}}^{N}}=\delta
_{c_{2j-1}^{N+1}}$ we define $d_{j,\limfunc{left}}^{N}$ to be one half the
length of $I_{j,\limfunc{left}}^{N}$ plus one quarter the length of the
neighbouring open middle third $G_{i}^{k}$ to the left of $I_{j,\limfunc{left%
}}^{N}$, i.e. 
\begin{equation*}
d_{j,\limfunc{left}}^{N}=\frac{1}{2}3^{-N-1}+\frac{1}{4}3^{-k-1}.
\end{equation*}%
Note that $0\leq k\leq N-1$, and that the neighbouring open middle third to
the \emph{right} of $I_{j,\limfunc{left}}^{N}$ is simply $G_{j}^{N}$.
Similarly, we define $d_{j,\limfunc{right}}^{N}$ to be one half the length
of $I_{j,\limfunc{right}}^{N}$ plus one quarter the length of the
neighbouring open middle third $G_{i^{\prime }}^{k^{\prime }}$ to the right
of $I_{j,\limfunc{right}}^{N}$, i.e. 
\begin{equation*}
d_{j,\limfunc{right}}^{N}=\frac{1}{2}3^{-N-1}+\frac{1}{4}3^{-k^{\prime }-1},
\end{equation*}%
where again $0\leq k^{\prime }\leq N-1$, and the neighbouring open middle
third to the \emph{left} of $I_{j,\limfunc{right}}^{N}$ is again $G_{j}^{N}$%
. Note that we have defined the lengths $d_{j,\limfunc{left}}^{N}$ and $d_{j,%
\limfunc{right}}^{N}$ so that%
\begin{eqnarray}
c_{j,\limfunc{left}}^{N}-d_{j,\limfunc{left}}^{N} &=&c_{i}^{k}+\frac{1}{4}%
3^{-k-1},  \label{left and right} \\
c_{j,\limfunc{right}}^{N}+d_{j,\limfunc{right}}^{N} &=&c_{i}^{k^{\prime }}-%
\frac{1}{4}3^{-k^{\prime }-1}.  \notag
\end{eqnarray}

We now define 
\begin{eqnarray*}
\widehat{\omega }_{N} &\equiv &\sum_{j=1}^{2^{N}}2^{-N-1}\left( \frac{\delta
_{\left( c_{j,\limfunc{left}}^{N}-d_{j,\limfunc{left}}^{N},\frac{1}{4}%
3^{-k-1}\right) }+\delta _{\left( c_{j,\limfunc{left}}^{N}-d_{j,\limfunc{left%
}}^{N},-\frac{1}{4}3^{-k-1}\right) }}{2}\right) \\
&&+\sum_{j=1}^{2^{N}}2^{-N-1}\left( \frac{\delta _{\left( c_{j,\limfunc{right%
}}^{N}+d_{j,\limfunc{right}}^{N},\frac{1}{4}3^{-k^{\prime }-1}\right)
}+\delta _{\left( c_{j,\limfunc{right}}^{N}+d_{j,\limfunc{right}}^{N},-\frac{%
1}{4}3^{-k^{\prime }-1}\right) }}{2}\right) .
\end{eqnarray*}%
Note in particular that the point mass $\delta _{\left( c_{j,\limfunc{left}%
}^{N},0\right) }$ has been replaced with the average of two point masses
whose locations in the plane, $\left( c_{j,\limfunc{left}}^{N}-d_{j,\limfunc{%
left}}^{N},\frac{1}{4}3^{-k-1}\right) $ and $\left( c_{j,\limfunc{left}%
}^{N}-d_{j,\limfunc{left}}^{N},-\frac{1}{4}3^{-k-1}\right) $, lie at less
than $45^{\circ }$ angles from $c_{j,\limfunc{left}}^{N}$ extending to the
left in the upper and lower half planes respectively. In similar fashion,
the point mass $\delta _{\left( c_{j,\limfunc{right}}^{N},0\right) }$ has
been replaced with the average of two point masses whose locations in the
plane, $\left( c_{j,\limfunc{right}}^{N}+d_{j,\limfunc{right}}^{N},\frac{1}{4%
}3^{-k^{\prime }-1}\right) $ and $\left( c_{j,\limfunc{right}}^{N}+d_{j,%
\limfunc{right}}^{N},-\frac{1}{4}3^{-k^{\prime }-1}\right) $, lie at less
than $45^{\circ }$ angles from $c_{j,\limfunc{left}}^{N}$ extending to the
right into the upper and lower half planes respectively.

The point of incorporating these less than $45^{\circ }$ angle translations
of locations is to obtain the following crucial property for all pairs of
points $y=\left( y_{1},y_{2}\right) $ and $z=\left( z_{1},z_{2}\right) $ in
the support of $\widehat{\omega }_{N}$ with $y_{1}\neq z_{1}$:%
\begin{equation}
\left\vert y_{2}-z_{2}\right\vert \leq \left\vert y_{1}-z_{1}\right\vert 
\text{ for all }y,z\text{ such that }y_{1}\neq z_{1}\text{, }\widehat{\omega 
}_{N}\left( y\right) \neq 0\text{ and }\widehat{\omega }_{N}\left( z\right)
\neq 0.  \label{slope}
\end{equation}%
This property is evident from another useful description of these measures
that derives from an extension of the observation that the intervals $I_{j,%
\limfunc{left}}^{N}$ and $I_{j,\limfunc{right}}^{N}$ are the left and right
neighbours of $G_{j}^{N}$ at level $N$. More precisely, the support of $%
\widehat{\omega }_{N}$ is contained in the union $\bigcup_{k=0}^{N-1}%
\bigcup_{i=1}^{2^{k}}\widehat{G_{i}^{k}}$ of the squares $\widehat{G_{i}^{k}}%
=G_{i}^{k}\times \left( \frac{1}{2}3^{-k-1},-\frac{1}{2}3^{-k-1}\right) $
corresponding to the open middle thirds of the intervals $I_{i}^{k}$ up to
level $N-1$. Moreover, for each $G_{i}^{k}$ with $0\leq k\leq N-1$ and $%
1\leq i\leq 2^{k}$, there are exactly four point masses from $\omega _{N}$
contained in $G_{i}^{k}$, and by (\ref{left and right}), they are located at
the points $\left( c_{i}^{k}\pm \frac{1}{4}3^{-k-1},\pm \frac{1}{4}%
3^{-k-1}\right) $. Thus we can rewrite $\widehat{\omega }_{N}$ as 
\begin{eqnarray}
\widehat{\omega }_{N} &=&2^{-N-1}\sum_{k=0}^{N-1}\sum_{i=1}^{2^{k}}\left( 
\frac{\delta _{\left( c_{i}^{k}+\frac{1}{4\cdot 3^{k+1}},\frac{1}{4\cdot
3^{k+1}}\right) }+\delta _{\left( c_{i}^{k}+\frac{1}{4\cdot 3^{k+1}},-\frac{1%
}{4\cdot 3^{k+1}}\right) }}{2}\right)  \label{representation} \\
&&+2^{-N-1}\sum_{k=0}^{N-1}\sum_{i=1}^{2^{k}}\left( \frac{\delta _{\left(
c_{i}^{k}-\frac{1}{4\cdot 3^{k+1}},\frac{1}{4\cdot 3^{k+1}}\right) }+\delta
_{\left( c_{i}^{k}-\frac{1}{4\cdot 3^{k+1}},-\frac{1}{4\cdot 3^{k+1}}\right)
}}{2}\right) .  \notag
\end{eqnarray}%
A simple picture in the plane of the support of $\widehat{\omega }_{N}$
using this representation of $\widehat{\omega }_{N}$ demonstrates the
property (\ref{slope}). Indeed, the slopes of the lines joining pairs of the
six points consisting of the four point supports of $\widehat{\omega }_{N}$
in $G_{i}^{k}$, namely $\left( c_{i}^{k}\pm \frac{1}{4\cdot 3^{k+1}},\pm 
\frac{1}{4\cdot 3^{k+1}}\right) $, and the two `endpoints' of $%
G_{i}^{k}\times \left\{ 0\right\} $, namely $\left( c_{i}^{k}\pm \frac{1}{%
2\cdot 3^{k+1}},0\right) $, are either infinite or at most $1$ in modulus.
In fact, the slopes of segments joining pairs of points in $\limfunc{supp}%
\widehat{\omega }_{N}$ are strictly less than $1$ in modulus unless the pair
of points lie in a common square $\widehat{G_{i}^{k}}$ on opposite sides of
the $x_{1}$-axis.

We will now define three measures $\widehat{\dot{\sigma}}_{N},\widehat{%
\sigma }_{N},\widehat{\sigma }_{N}^{+}$ in the plane loosely motivated by
the two measures $\dot{\sigma},\sigma $ on the line constructed in \cite%
{LaSaUr2}. Recall that $G_{i}^{k}$ is the removed open middle third of $%
I_{i}^{k}$. The measure $\widehat{\dot{\sigma}}$, restricted to a square $%
\widehat{G_{i}^{k}}$, will consist of a multiple of the single point mass $%
\delta _{\left( c_{i}^{k},0\right) }$ located on the real axis at the center 
$c_{i}^{k}$ of $G_{i}^{k}$, while the measure $\widehat{\sigma }$,
restricted to a square $\widehat{G_{i}^{k}}$, will consist of multiples of
two point masses lying equidistant above and below the real axis. The
measure $\widehat{\sigma }_{N}^{+}$ will be the restriction of $\widehat{%
\sigma }_{N}$ to the upper half plane. We now turn to describing these
measures explicitly.

Let $H_{i}^{k}=\frac{1}{2}G_{i}^{k}$ be the middle half of the open middle
third $G_{i}^{k}$ of $I_{i}^{k}$ (in other words the open middle sixth of $%
I_{i}^{k}$), and note that by construction $\widehat{\omega }_{N}$ does not
charge the open rectangle $H_{i}^{k}\times \mathbb{R}$. On the other hand
there are four points $\left( c_{i}^{k}\pm \frac{1}{4}3^{-k-1},\pm \frac{1}{4%
}3^{-k-1}\right) $ in the support of $\widehat{\omega }_{N}$ that lie on the
boundary of the strip $H_{i}^{k}\times \mathbb{R}$, two on the left edge and
two on the right edge. If we let $H_{i}^{k}=\left[ u_{i}^{k},v_{i}^{k}\right]
$, then these four points are $P_{i,\pm }^{k}=\left( u_{i}^{k},\pm \frac{1}{%
4\cdot 3^{k+1}}\right) $ and $Q_{i,\pm }^{k}=\left( v_{i}^{k},\pm \frac{1}{%
4\cdot 3^{k+1}}\right) $.

It is convenient to also define the minimal closed intervals $%
L_{j}^{k}\varsupsetneqq I_{j}^{k}$ so that the closed square $\widehat{%
L_{j}^{k}}=L_{j}^{k}\times \left[ \frac{1}{2}3^{-k-1},-\frac{1}{2}3^{-k-1}%
\right] $ contains all the point masses in $\widehat{\omega }_{N}$ that were
constructed from the point masses of $\omega _{N}$ lying in $I_{j}^{k}$ by
the procedure of splitting into two point masses. In other words, if $%
I_{j}^{k}$ is adjacent to $G_{i}^{\ell }$ on the left and to $G_{i^{\prime
}}^{k-1}$ on the right, then $L_{j}^{k}\equiv \left[ c_{i}^{\ell }+\frac{1}{%
4\cdot 3^{\ell +1}},c_{i^{\prime }}^{k-1}-\frac{1}{4\cdot 3^{k+1}}\right] $;
while if $I_{j}^{k}$ is adjacent to $G_{i}^{k-1}$ on the left and to $%
G_{i^{\prime }}^{\ell ^{\prime }}$ on the right, then $L_{j}^{k}\equiv \left[
c_{i}^{k-1}+\frac{1}{4\cdot 3^{k+1}},c_{i^{\prime }}^{\ell ^{\prime }-1}-%
\frac{1}{4\cdot 3^{\ell ^{\prime }+1}}\right] $. Thus $L_{j}^{k}$ sticks out
beyond $I_{j}^{k}$ on each side a distance $\frac{1}{4}\left\vert
G_{r}^{\ell }\right\vert $ determined by the length of the adjacent middle
third $G_{r}^{\ell }$ on that side.

Recall the $A_{2}^{1}$ condition in the plane $\mathbb{R}^{2}\,$:%
\begin{equation*}
A_{2}^{1}\left( \widehat{\dot{\sigma}},\widehat{\omega }\right) \equiv
\sup_{Q\in \mathcal{P}^{2}}\frac{\left\vert Q\right\vert _{\widehat{\omega }%
_{N}}}{\left\vert Q\right\vert ^{1-\frac{1}{2}}}\frac{\left\vert
Q\right\vert _{\widehat{\dot{\sigma}}_{N}}}{\left\vert Q\right\vert ^{1-%
\frac{1}{2}}}.
\end{equation*}

\begin{notation}
For an interval $I$ denote by $\widehat{I}$ the square $I\times \left[ -%
\frac{1}{2}\left\vert I\right\vert ,\left\vert I\right\vert \right] $.
\end{notation}

Define 
\begin{equation}
\widehat{\dot{\sigma}}_{N}=\sum_{k=0}^{N-1}\sum_{i=1}^{2^{k}}s_{i}^{k}\delta
_{\left( c_{i}^{k},0\right) },  \label{e.sigmaD}
\end{equation}%
where the sequence of positive numbers $s_{i}^{k}$ is chosen to satisfy the
following precursor of the $A_{2}^{1}$ condition involving the squares $%
\widehat{L_{i}^{k}}\equiv L_{i}^{k}\times \left[ -\frac{1}{2}\left\vert
L_{i}^{k}\right\vert ,\frac{1}{2}\left\vert L_{i}^{k}\right\vert \right] $:%
\begin{equation*}
\frac{s_{i}^{k}\widehat{\omega }_{N}(\widehat{L_{i}^{k}})}{|\widehat{%
L_{i}^{k}}|^{2\left( 1-\frac{1}{2}\right) }}=\frac{s_{i}^{k}\omega
_{N}(I_{i}^{k})}{\left\vert L_{i}^{k}\right\vert ^{2}}\approx \frac{%
s_{i}^{k}2^{-k}}{3^{-2k}}=1.
\end{equation*}%
Note that we also have a similar estimate for the squares $\widehat{I_{i}^{k}%
}$,%
\begin{equation*}
\frac{\widehat{\omega }_{N}(\widehat{I_{i}^{k}})}{|\widehat{I_{i}^{k}}%
|^{2\left( 1-\frac{1}{2}\right) }}=\frac{s_{i}^{k}\widehat{\omega }%
_{N}(I_{i}^{k}\times \left[ -\frac{1}{2}3^{-k},\frac{1}{2}3^{-k}\right] )}{%
\left\vert I_{i}^{k}\right\vert ^{2}}\approx \frac{2^{-k}}{3^{-2k}}\approx 
\frac{\widehat{\omega }_{N}(\widehat{L_{i}^{k}})}{|\widehat{L_{i}^{k}}%
|^{2\left( 1-\frac{1}{2}\right) }},
\end{equation*}%
since $\widehat{\omega }_{N}(I_{i}^{k}\times \left[ -\frac{1}{2}3^{-k},\frac{%
1}{2}3^{-k}\right] \approx 2^{-k}$ for $0\leq k\leq N-1$ because only a
fixed proportion of the mass of $\omega _{N}$ escapes $\widehat{I_{i}^{k}}$
when the point masses in $\omega _{N}$ at the extreme left and right inside $%
I_{i}^{k}$ were spread out at less than $45^{\circ }$ angles away from $%
I_{i}^{k}$ into the upper and lower half planes. Thus we define%
\begin{equation*}
s_{i}^{k}=\frac{2^{k}}{3^{2k}}=\left( \frac{1}{3}\right) ^{k}\left( \frac{2}{%
3}\right) ^{k}\qquad k\geq 0,1\leq i\leq 2^{k},
\end{equation*}%
which agrees with the weights $s_{i}^{k}$ used in \cite{LaSaUr2} when $%
\alpha =0$ and $n=1$. The definition of the measures $\widehat{\sigma }_{N}$
and $\widehat{\sigma }_{N}^{+}$ will depend on the fractional Riesz
transform $\mathbf{R}^{1,2}$ with convolution kernel $\mathbf{K}^{1,2}\left(
\xi \right) =\frac{\xi }{\left\vert \xi \right\vert ^{2}}$, and is closely
related to the structure of the function $R_{1}^{1,2}\widehat{\omega }_{N}$,
where $\mathbf{R}^{1,2}=\left( R_{1}^{1,2},R_{2}^{1,2}\right) $.

We focus on the kernel%
\begin{equation*}
K_{1}^{1,2}\left( \xi \right) =\frac{\xi _{1}}{\left\vert \xi \right\vert
^{2}}=\frac{\xi _{1}}{\xi _{1}^{2}+\xi _{2}^{2}},
\end{equation*}%
and the four points $P_{i,\pm }^{k}=\left( u_{i}^{k},\pm \frac{1}{4\cdot
3^{k+1}}\right) $ and $Q_{i,\pm }^{k}=\left( v_{i}^{k},\pm \frac{1}{4\cdot
3^{k+1}}\right) $ that are the vertices of the square $\widehat{H_{i}^{k}}$.
Fix a horizontal segment $H_{i}^{k}\times \left\{ x_{2}\right\} $ with $%
x_{2}\in \left\{ \pm \frac{1}{4\cdot 3^{k+1}}\right\} $, i.e. either the top
or bottom edge of the square $\widehat{H_{i}^{k}}$. Then the function 
\begin{equation*}
F\left( x_{1}\right) \equiv R_{1}^{1,2}\widehat{\omega }_{N}\left(
x_{1},x_{2}\right) =\int K_{1}^{1,2}d\widehat{\omega }_{N}=\int_{-1}^{1}%
\int_{-1}^{1}\frac{x_{1}-y_{1}}{\left( x_{1}-y_{1}\right) ^{2}+\left(
x_{2}-y_{2}\right) ^{2}}d\widehat{\omega }_{N}\left( y_{1},y_{2}\right)
\end{equation*}%
is monotonically decreasing for $x_{1}$ in $\left[ u_{i}^{k},v_{i}^{k}\right]
$ from the value 
\begin{equation*}
F\left( u_{i}^{k}\right) =\int_{-1}^{1}\int_{-1}^{1}\frac{u_{i}^{k}-y_{1}}{%
\left\vert u_{i}^{k}-y_{1}\right\vert ^{2}+\left\vert x_{2}-y_{2}\right\vert
^{2}}d\widehat{\omega }_{N}\left( y_{1},y_{2}\right)
\end{equation*}%
at the left hand endpoint of $G_{i}^{k}=\left[ u_{i}^{k},v_{i}^{k}\right] $,
to the value%
\begin{equation*}
F\left( v_{i}^{k}\right) =\int_{-1}^{1}\int_{-1}^{1}\frac{v_{i}^{k}-y_{1}}{%
\left\vert v_{i}^{k}-y_{1}\right\vert ^{2}+\left\vert x_{2}-y_{2}\right\vert
^{2}}d\widehat{\omega }_{N}\left( y_{1},y_{2}\right)
\end{equation*}%
at the right hand endpoint.

Indeed, to see this, fix $y=\left( y_{1},y_{2}\right) \in \limfunc{supp}%
\widehat{\omega }_{N}\setminus \left\{ P_{i,\pm }^{k},Q_{i,\pm }^{k}\right\} 
$. Then using (\ref{slope}) it is easy to see that $\left\vert
x_{2}-y_{2}\right\vert <\left\vert x_{1}-y_{1}\right\vert $ for all $%
x_{1}\in H_{j}^{k}$, and so%
\begin{eqnarray*}
\frac{d}{dx_{1}}\frac{x_{1}-y_{1}}{\left( x_{1}-y_{1}\right) ^{2}+\left(
x_{2}-y_{2}\right) ^{2}} &=&\frac{\left( x_{1}-y_{1}\right) ^{2}+\left(
x_{2}-y_{2}\right) ^{2}-2\left( x_{1}-y_{1}\right) ^{2}}{\left[ \left(
x_{1}-y_{1}\right) ^{2}+\left( x_{2}-y_{2}\right) ^{2}\right] ^{2}} \\
&=&\frac{\left( x_{2}-y_{2}\right) ^{2}-\left( x_{1}-y_{1}\right) ^{2}}{%
\left[ \left( x_{1}-y_{1}\right) ^{2}+\left( x_{2}-y_{2}\right) ^{2}\right]
^{2}}<0.
\end{eqnarray*}%
Now consider the integral corresponding to the sum of the two points $%
\left\{ P_{i,\pm }^{k}\right\} $ in the support of $\widehat{\omega }_{N}$.
This integral is a positive multiple of the following sum:%
\begin{eqnarray*}
&&\frac{x_{1}-u_{i}^{k}}{\left( x_{1}-u_{i}^{k}\right) ^{2}+\left( \frac{1}{%
4\cdot 3^{k+1}}-\frac{1}{4\cdot 3^{k+1}}\right) ^{2}}+\frac{x_{1}-u_{i}^{k}}{%
\left( x_{1}-u_{i}^{k}\right) ^{2}+\left( \frac{1}{4\cdot 3^{k+1}}+\frac{1}{%
4\cdot 3^{k+1}}\right) ^{2}} \\
&&\ \ \ \ \ \ \ \ \ \ \equiv \frac{1}{t}+\frac{t}{t^{2}+A^{2}}\text{,\ \ \ \
\ }t=x_{1}-u_{i}^{k}\text{ and }A=\frac{1}{2\cdot 3^{k+1}},
\end{eqnarray*}%
where 
\begin{equation*}
\frac{d}{dt}\left( \frac{1}{t}+\frac{t}{t^{2}+A^{2}}\right) =-\frac{1}{t^{2}}%
+\frac{1}{t^{2}+A^{2}}-\frac{2t^{2}}{\left( t^{2}+A^{2}\right) ^{2}}=-\frac{%
A^{2}+t^{4}}{t^{2}\left( t^{2}+A^{2}\right) ^{2}}<0.
\end{equation*}%
Similarly, the integral corresponding to the sum of the two points $\left\{
Q_{i,\pm }^{k}\right\} $ is a positive multiple of the sum%
\begin{eqnarray*}
&&\frac{x_{1}-v_{i}^{k}}{\left( x_{1}-v_{i}^{k}\right) ^{2}+\left( \frac{1}{%
4\cdot 3^{k+1}}-\frac{1}{4\cdot 3^{k+1}}\right) ^{2}}+\frac{x_{1}-v_{i}^{k}}{%
\left( x_{1}-v_{i}^{k}\right) ^{2}+\left( \frac{1}{4\cdot 3^{k+1}}+\frac{1}{%
4\cdot 3^{k+1}}\right) ^{2}} \\
&&\ \ \ \ \ \ \ \ \ \ \equiv \frac{1}{t}+\frac{t}{t^{2}+A^{2}}\text{,\ \ \ \
\ }t=x_{1}-v_{i}^{k}\text{ and }A=\frac{1}{2\cdot 3^{k+1}},
\end{eqnarray*}%
whose $t$ derivative was shown above to be negative. This completes the
proof that $F\left( x_{1}\right) $ is monotonically decreasing for $x_{1}$
in $\left[ u_{i}^{k},v_{i}^{k}\right] $.

Now we have $\lim_{x_{1}\searrow u_{i}^{k}}F\left( x_{1}\right) =\infty $
since the integrand $\frac{x_{1}-y_{1}}{\left\vert x_{1}-y_{1}\right\vert
^{2}+\left\vert x_{2}-y_{2}\right\vert ^{2}}=\nearrow \infty $ as $\left(
x_{1},\pm \frac{1}{4\cdot 3^{k+1}}\right) \rightarrow P_{i,\pm }^{k}=\left(
u_{i}^{k},\pm \frac{1}{4\cdot 3^{k+1}}\right) \in \limfunc{supp}\widehat{%
\omega }_{N}$. Similarly, $\lim_{x_{1}\nearrow v_{j}^{k}}F\left(
x_{1}\right) =-\infty $, and we conclude that $F\left( x_{1}\right) \equiv
R_{1}^{1,2}\widehat{\omega }_{N}\left( x_{1},x_{2}\right) $ strictly
decreases from $\infty $ to $-\infty $ along the horizontal segment $%
H_{i}^{k}\times \left\{ x_{2}\right\} $ when $x_{2}\in \left\{ \pm \frac{1}{%
4\cdot 3^{k+1}}\right\} $. In particular, $R_{1}^{1,2}\widehat{\omega }_{0}$
has a unique zero $z_{i}^{k}$ on the horizontal segment $H_{i}^{k}\times
\left\{ x_{2}\right\} $, which is independent of the two choices $x_{2}=\pm 
\frac{1}{4\cdot 3^{k+1}}$ by symmetry. With this choice of $z_{i}^{k}$, we
define the measures $\widehat{\sigma }_{N}$ and $\widehat{\sigma }_{N}^{+}$\
by 
\begin{eqnarray*}
\widehat{{\sigma }}_{N} &=&\sum_{k,i}s_{i}^{k}\left( \delta _{\left(
z_{i}^{k},\frac{1}{4\cdot 3^{k+1}}\right) }+\delta _{\left( z_{i}^{k},-\frac{%
1}{4\cdot 3^{k+1}}\right) }\right) , \\
\widehat{{\sigma }}_{N}^{+} &=&\sum_{k,i}s_{i}^{k}\delta _{\left( z_{i}^{k},%
\frac{1}{4\cdot 3^{k+1}}\right) }.
\end{eqnarray*}

\subsection{An estimate for the second component}

From the representation (\ref{representation}) of $\widehat{\omega }_{N}$,
it is clear that there are $2N$ horizontal lines on which $\widehat{\omega }%
_{N}$ is supported, namely the lines $\left\{ L_{\beta }\right\} _{\beta \in
\left\{ \pm \frac{1}{4\cdot 3^{k+1}}\right\} _{k=0}^{N-1}}$ where $L_{\beta
}\equiv \left\{ \left( x,\beta \right) \in \mathbb{R}^{2}:x\in \mathbb{R}%
\right\} $ is the $x$-axis translated vertically by $\beta $. We now
estimate the second component $R_{2}^{1,2}\widehat{\omega }_{N}$ of the
Riesz transform of $\widehat{\omega }_{N}$ on the support of the measure $%
\widehat{{\sigma }}_{N}^{+}$. So fix a point $\left( z_{j}^{\ell },\frac{1}{%
4\cdot 3^{\ell +1}}\right) $ in the strip $H_{j}^{\ell }\times \mathbb{R}$
that lies in the support of $\widehat{{\sigma }}_{N}^{+}$ with $N\geq 3$.
Then we have%
\begin{eqnarray*}
&&R_{2}^{1,2}\widehat{\omega }_{N}\left( z_{j}^{\ell },\frac{1}{4\cdot
3^{\ell +1}}\right) =\int_{-1}^{1}\int_{-1}^{1}\frac{\frac{1}{4\cdot 3^{\ell
+1}}-y_{2}}{\left\vert z_{j}^{\ell }-y_{1}\right\vert ^{2}+\left\vert \frac{1%
}{4\cdot 3^{\ell +1}}-y_{2}\right\vert ^{2}}d\widehat{\omega }_{N}\left(
y_{1},y_{2}\right) \\
&=&2^{-N-2}\sum_{k=0}^{N-1}\sum_{i=1}^{2^{k}}\left\{ \frac{\frac{1}{4\cdot
3^{\ell +1}}-\frac{1}{4\cdot 3^{k+1}}}{\left\vert z_{j}^{\ell }-\left( c_{%
\limfunc{left}}^{N+1}\left( G_{i}^{k}\right) +\frac{1}{4\cdot 3^{k+1}}%
\right) \right\vert ^{2}+\left\vert \frac{1}{4\cdot 3^{\ell +1}}-\frac{1}{%
4\cdot 3^{k+1}}\right\vert ^{2}}\right. \\
&&\ \ \ \ \ \ \ \ \ \ \ \ \ \ \ \ \ \ \ \ \ \ \ \ \ \ \ \ \ \ +\left. \frac{%
\frac{1}{4\cdot 3^{\ell +1}}+\frac{1}{4\cdot 3^{k+1}}}{\left\vert
z_{j}^{\ell }-\left( c_{\limfunc{left}}^{N+1}\left( G_{i}^{k}\right) +\frac{1%
}{4\cdot 3^{k+1}}\right) \right\vert ^{2}+\left\vert \frac{1}{4\cdot 3^{\ell
+1}}+\frac{1}{4\cdot 3^{k+1}}\right\vert ^{2}}\right\} \\
&&+2^{-N-2}\sum_{k=0}^{N-1}\sum_{i^{\prime }=1}^{2^{k}}\left\{ \frac{\frac{1%
}{4\cdot 3^{\ell +1}}-\frac{1}{4\cdot 3^{k+1}}}{\left\vert z_{j}^{\ell
}-\left( c_{\limfunc{right}}^{N+1}\left( G_{i}^{k}\right) -\frac{1}{4\cdot
3^{k+1}}\right) \right\vert ^{2}+\left\vert \frac{1}{4\cdot 3^{\ell +1}}-%
\frac{1}{4\cdot 3^{k+1}}\right\vert ^{2}}\right. \\
&&\ \ \ \ \ \ \ \ \ \ \ \ \ \ \ \ \ \ \ \ \ \ \ \ \ \ \ \ \ \ +\left. \frac{%
\frac{1}{4\cdot 3^{\ell +1}}+\frac{1}{4\cdot 3^{k+1}}}{\left\vert
z_{j}^{\ell }-\left( c_{\limfunc{right}}^{N+1}\left( G_{i}^{k}\right) -\frac{%
1}{4\cdot 3^{k+1}}\right) \right\vert ^{2}+\left\vert \frac{1}{4\cdot
3^{\ell +1}}+\frac{1}{4\cdot 3^{k+1}}\right\vert ^{2}}\right\} .
\end{eqnarray*}%
The negative terms are those with numerator $\frac{1}{4\cdot 3^{\ell +1}}-%
\frac{1}{4\cdot 3^{k+1}}$ for $0\leq k\leq \ell $, and the sum of those
terms corresponding to the left edge of $G_{i}^{k}$ is%
\begin{equation*}
2^{-N-2}\sum_{k=0}^{\ell }\sum_{i=1}^{2^{k}}\frac{\frac{1}{4\cdot 3^{\ell +1}%
}-\frac{1}{4\cdot 3^{k+1}}}{\left\vert z_{j}^{\ell }-\left( c_{\limfunc{left}%
}^{N+1}\left( G_{i}^{k}\right) +\frac{1}{4\cdot 3^{k+1}}\right) \right\vert
^{2}+\left\vert \frac{1}{4\cdot 3^{\ell +1}}-\frac{1}{4\cdot 3^{k+1}}%
\right\vert ^{2}}.
\end{equation*}%
The analogous sum of positive parts for $0\leq k\leq \ell $ corresponding to
the left edge of $G_{i}^{k}$ is given by%
\begin{equation*}
2^{-N-2}\sum_{k=0}^{\ell }\sum_{i=1}^{2^{k}}\frac{\frac{1}{4\cdot 3^{\ell +1}%
}+\frac{1}{4\cdot 3^{k+1}}}{\left\vert z_{j}^{\ell }-\left( c_{\limfunc{left}%
}^{N+1}\left( G_{i}^{k}\right) +\frac{1}{4\cdot 3^{k+1}}\right) \right\vert
^{2}+\left\vert \frac{1}{4\cdot 3^{\ell +1}}+\frac{1}{4\cdot 3^{k+1}}%
\right\vert ^{2}}.
\end{equation*}%
Adding the two fractions appearing in these sums gives, with $A_{\limfunc{%
left}}^{\left( \ell ,j\right) ,\left( k,i\right) }=z_{j}^{\ell }-\left( c_{%
\limfunc{left}}^{N+1}\left( G_{i}^{k}\right) +\frac{1}{4\cdot 3^{k+1}}%
\right) $, 
\begin{eqnarray*}
&&\frac{\frac{1}{4\cdot 3^{\ell +1}}-\frac{1}{4\cdot 3^{k+1}}}{\left\vert A_{%
\limfunc{left}}^{\left( \ell ,j\right) ,\left( k,i\right) }\right\vert
^{2}+\left\vert \frac{1}{4\cdot 3^{\ell +1}}-\frac{1}{4\cdot 3^{k+1}}%
\right\vert ^{2}}+\frac{\frac{1}{4\cdot 3^{\ell +1}}+\frac{1}{4\cdot 3^{k+1}}%
}{\left\vert A_{\limfunc{left}}^{\left( \ell ,j\right) ,\left( k,i\right)
}\right\vert ^{2}+\left\vert \frac{1}{4\cdot 3^{\ell +1}}+\frac{1}{4\cdot
3^{k+1}}\right\vert ^{2}} \\
&=&\frac{1}{4\cdot 3^{\ell +1}}\left\{ \frac{1}{\left\vert A_{\limfunc{left}%
}^{\left( \ell ,j\right) ,\left( k,i\right) }\right\vert ^{2}+\left\vert 
\frac{1}{4\cdot 3^{\ell +1}}-\frac{1}{4\cdot 3^{k+1}}\right\vert ^{2}}+\frac{%
1}{\left\vert A_{\limfunc{left}}^{\left( \ell ,j\right) ,\left( k,i\right)
}\right\vert ^{2}+\left\vert \frac{1}{4\cdot 3^{\ell +1}}+\frac{1}{4\cdot
3^{k+1}}\right\vert ^{2}}\right\} \\
&&-\frac{1}{4\cdot 3^{k+1}}\left\{ \frac{1}{\left\vert A_{\limfunc{left}%
}^{\left( \ell ,j\right) ,\left( k,i\right) }\right\vert ^{2}+\left\vert 
\frac{1}{4\cdot 3^{\ell +1}}-\frac{1}{4\cdot 3^{k+1}}\right\vert ^{2}}-\frac{%
1}{\left\vert A_{\limfunc{left}}^{\left( \ell ,j\right) ,\left( k,i\right)
}\right\vert ^{2}+\left\vert \frac{1}{4\cdot 3^{\ell +1}}+\frac{1}{4\cdot
3^{k+1}}\right\vert ^{2}}\right\} \\
&\approx &\frac{1}{4\cdot 3^{\ell +1}}\frac{2}{\left\vert A_{\limfunc{left}%
}^{\left( \ell ,j\right) ,\left( k,i\right) }\right\vert ^{2}+\left\vert 
\frac{1}{4\cdot 3^{k+1}}\right\vert ^{2}}-\frac{1}{4\cdot 3^{k+1}}\left\{ 
\frac{4\frac{1}{4\cdot 3^{\ell +1}}\frac{1}{4\cdot 3^{k+1}}}{\left[
\left\vert A_{\limfunc{left}}^{\left( \ell ,j\right) ,\left( k,i\right)
}\right\vert ^{2}+\left\vert \frac{1}{4\cdot 3^{k+1}}\right\vert ^{2}\right]
^{2}}\right\} ,
\end{eqnarray*}%
which equals%
\begin{eqnarray*}
&&\frac{1}{4\cdot 3^{\ell +1}}\left[ \frac{2}{\left\vert A_{\limfunc{left}%
}^{\left( \ell ,j\right) ,\left( k,i\right) }\right\vert ^{2}+\left\vert 
\frac{1}{4\cdot 3^{k+1}}\right\vert ^{2}}-\frac{4}{\left[ \left\vert A_{%
\limfunc{left}}^{\left( \ell ,j\right) ,\left( k,i\right) }\right\vert
^{2}+\left\vert \frac{1}{4\cdot 3^{k+1}}\right\vert ^{2}\right] ^{2}}\right]
\\
&\approx &\frac{1}{4\cdot 3^{\ell +1}}\frac{2}{\left\vert A_{\limfunc{left}%
}^{\left( \ell ,j\right) ,\left( k,i\right) }\right\vert ^{2}+\left\vert 
\frac{1}{4\cdot 3^{k+1}}\right\vert ^{2}}.
\end{eqnarray*}%
Finally, if we sum this over $\sum_{k=0}^{\ell }\sum_{i=1}^{2^{k}}$ and
multiply by $2^{-N-2}$ we get%
\begin{equation*}
2^{-N-2}\frac{1}{4\cdot 3^{\ell +1}}\sum_{k=0}^{\ell }\sum_{i=1}^{2^{k}}%
\frac{2}{\left\vert A_{\limfunc{left}}^{\left( \ell ,j\right) ,\left(
k,i\right) }\right\vert ^{2}+\left\vert \frac{1}{4\cdot 3^{k+1}}\right\vert
^{2}}\approx 2^{-N}3^{\ell },
\end{equation*}%
since the main term here occurs when $k=\ell $ and $i$ and $j$ are such that 
$A_{\limfunc{left}}^{\left( \ell ,j\right) ,\left( \ell ,i\right)
}=z_{j}^{\ell }-\left( c_{\limfunc{left}}^{N+1}\left( G_{i}^{\ell }\right) +%
\frac{1}{4\cdot 3^{\ell +1}}\right) \approx \frac{1}{3^{\ell }}$. A similar
estimate, with $A_{\limfunc{right}}^{\left( \ell ,j\right) ,\left(
k,i\right) }=z_{j}^{\ell }-\left( c_{\limfunc{right}}^{N+1}\left(
G_{i}^{k}\right) -\frac{1}{4\cdot 3^{k+1}}\right) $, is obtained for the
negative terms corresponding to the right edge of $G_{i}^{k}$, namely%
\begin{equation*}
2^{-N-2}\frac{1}{4\cdot 3^{\ell +1}}\sum_{k=0}^{\ell }\sum_{i=1}^{2^{k}}%
\frac{2}{\left\vert A_{\limfunc{right}}^{\left( \ell ,j\right) ,\left(
k,i\right) }\right\vert ^{2}+\left\vert \frac{1}{4\cdot 3^{k+1}}\right\vert
^{2}}\approx 2^{-N}3^{\ell }.
\end{equation*}

Now if we sum over the remaining terms for $\ell <k\leq N-1$, and use the
crude estimate $\left\vert \frac{1}{4\cdot 3^{\ell +1}}\pm \frac{1}{4\cdot
3^{k+1}}\right\vert \leq 2\frac{1}{4\cdot 3^{\ell +1}}$ for $\ell <k$, we
get approximately%
\begin{eqnarray*}
&&2^{-N-2}\sum_{k=\ell +1}^{N-1}\sum_{i=1}^{2^{k}}\left\{ \frac{2\frac{1}{%
4\cdot 3^{\ell +1}}}{\left\vert A_{\limfunc{left}}^{\left( \ell ,j\right)
,\left( k,i\right) }\right\vert ^{2}+\left\vert \frac{1}{4\cdot 3^{\ell +1}}+%
\frac{1}{4\cdot 3^{k+1}}\right\vert ^{2}}+\frac{2\frac{1}{4\cdot 3^{\ell +1}}%
}{\left\vert A_{\limfunc{right}}^{\left( \ell ,j\right) ,\left( k,i\right)
}\right\vert ^{2}+\left\vert \frac{1}{4\cdot 3^{\ell +1}}+\frac{1}{4\cdot
3^{k+1}}\right\vert ^{2}}\right\} \\
&\approx &2^{-N-2}\sum_{k=\ell +1}^{N-1}\sum_{i=1}^{2^{k}}\frac{4}{4\cdot
3^{\ell +1}}\left\{ \frac{1}{\left\vert A_{\limfunc{left}}^{\left( \ell
,j\right) ,\left( k,i\right) }\right\vert ^{2}+\left\vert \frac{1}{4\cdot
3^{\ell +1}}+\frac{1}{4\cdot 3^{k+1}}\right\vert ^{2}}+\frac{1}{\left\vert
A_{\limfunc{right}}^{\left( \ell ,j\right) ,\left( k,i\right) }\right\vert
^{2}+\left\vert \frac{1}{4\cdot 3^{\ell +1}}+\frac{1}{4\cdot 3^{k+1}}%
\right\vert ^{2}}\right\} \approx 2^{-N}3^{\ell },
\end{eqnarray*}%
since the main term here occurs when $k=\ell +1$ and $i$ and $j$ are such
that $A_{\limfunc{left}}^{\left( \ell ,j\right) ,\left( \ell +1,i\right)
}=z_{j}^{\ell }-\left( c_{\limfunc{left}}^{N+1}\left( G_{i}^{\ell +1}\right)
+\frac{1}{4\cdot 3^{\ell +2}}\right) \approx \frac{1}{3^{\ell }}$. So
altogether we have the estimate%
\begin{equation}
R_{2}^{1,2}\widehat{\omega }_{N}\left( z_{j}^{\ell },\frac{1}{4\cdot 3^{\ell
+1}}\right) \approx 2^{-N}3^{\ell },  \label{estimate}
\end{equation}%
which will suffice to prove the backward testing condition for $R_{2}^{1,2}$
below.

\subsection{The plan of attack}

\begin{itemize}
\item From our choice of $s_{j}^{k}$ we will obtain the Muckenhoupt
conditions $\mathcal{A}_{2}^{1}$ and $\mathcal{A}_{2}^{1,\ast }$ for the
measure pairs $\left( \widehat{{\sigma }}_{N},\widehat{\omega }_{N}\right) $
uniformly in $N\geq 1$.

\item The points $\left( z_{i}^{k}\left( x_{2}\right) ,\pm \frac{1}{4\cdot
3^{-k-1}}\right) $ lie in the zero set of $R_{1}^{1,2}\widehat{\omega }_{N}$%
, and the resulting cancellation in the backward testing condition for $%
R_{1}^{1,2}$ with respect to the weight pair $\left( \widehat{{\sigma }}_{N},%
\widehat{\omega }_{N}\right) $ is enough to obtain it uniformly in $N\geq 1$.

\item The self-similarity of the measure $\widehat{\dot{\sigma}}_{N}$ will
aid in computing the forward testing condition for $R_{1}^{1,2}$ with
respect to the measure pairs $\left( \widehat{\dot{\sigma}}_{N},\widehat{%
\omega }_{N}\right) $, and then a perturbation argument will establish the
forward testing condition for $R_{1}^{1,2}$ with respect to the measure
pairs $\left( \widehat{{\sigma }}_{N}^{+},\widehat{\omega }_{N}\right) $
uniformly in $N\geq 1$.

\item Then we will use the estimate (\ref{estimate}) to show that the
testing conditions for $R_{1}^{1,2}$ hold uniformly in $N\geq 1$ for the
measure pairs $\left( \widehat{{\sigma }}_{N}^{+},\widehat{\omega }%
_{N}\right) $.

\item We next establish the forward and backward energy conditions uniformly
in $N$ for the measure pairs $\left( \widehat{{\sigma }}_{N}^{+},\widehat{%
\omega }_{N}\right) $.

\item Then we establish the forward and backward testing conditions
uniformly in $N$ for the second component $R_{2}^{1,2}$ of the Riesz
transform.

\item Using $\mathcal{A}_{2}^{1}\left( \widehat{{\sigma }}_{N}^{+},\widehat{%
\omega }_{N}\right) \leq \mathcal{A}_{2}^{1}\left( \widehat{{\sigma }}_{N},%
\widehat{\omega }_{N}\right) \leq C<\infty $, we can then conclude from the $%
T1$ theorem in \cite{SaShUr7} that $\mathfrak{N}_{\mathbf{R}^{1,2}}\left( 
\widehat{{\sigma }}_{N}^{+},\widehat{\omega }_{N}\right) <\infty $ uniformly
in $N$.

\item Finally, we show by a direct computation that $\mathfrak{N}_{\mathbf{R}%
^{1,2}}\left( \widehat{{\sigma }}_{N},\widehat{\omega }_{N}\right) \leq 2%
\mathfrak{N}_{\mathbf{R}^{1,2}}\left( \widehat{{\sigma }}_{N}^{+},\widehat{%
\omega }_{N}\right) $ for all $N$, and that the backward energy condition
fails with respect to the measure pair $\left( \widehat{{\sigma }}_{N},%
\widehat{\omega }_{N}\right) $ for each $N$.

\item The result is that the two-dimensional Riesz transform $\mathbf{R}%
^{1,2}$ is bounded from $L^{2}\left( \widehat{{\sigma }}_{N}\right) $ to $%
L^{2}\left( \widehat{\omega }_{N}\right) $ uniformly in $N\geq 1$, yet the
energy constants for the weight pairs $\left( \widehat{{\sigma }}_{N},%
\widehat{\omega }_{N}\right) $ are unbounded for $N\geq 1$.
\end{itemize}

In order to execute this strategy in the next four subsections, although not
necessarily in the order specified above, we will follow as closely as we
can the line of argument in \cite{LaSaUr2}, adapting to the plane as
necessary. We begin by calculating the rate at which $R_{1}^{1,2}\widehat{%
\omega }_{N}\left( \cdot ,\frac{1}{4\cdot 3^{k+1}}\right) $ blows up at the
endpoints of the intervals $H_{j}^{k}$.

\begin{lemma}
\label{l.gkj}Let $H_{j}^{k}=(u_{j}^{k},v_{j}^{k})$. We have 
\begin{equation}
R_{1}^{1,2}\widehat{\omega }_{N}\left( u_{j}^{k}-c3^{-k},\frac{1}{4\cdot
3^{k+1}}\right) \approx \left( \frac{3}{2}\right) ^{k}\,,\ \ \ \ \ k\geq
0\,,\ 1\leq j\leq 2^{k}\,,  \label{e.rapid}
\end{equation}%
and a similar approximate equality, with signs reversed, holds for $%
v_{j}^{k} $.
\end{lemma}

This in particular shows that the zeros $z_{j}^{k}$ cannot move too far from
the middle: 
\begin{equation}
\sup_{j,k}\frac{\lvert z_{j}^{k}-c_{j}^{k}\rvert }{\lvert H_{j}^{k}\rvert }%
<\zeta <1\,.  \label{e.closeToMiddle}
\end{equation}

\begin{proof}
Fix $k$, and consider the numbers $R_{1}^{1,2}\widehat{\omega }_{N}\left(
u_{j}^{k}\pm c3^{-k},\frac{1}{4\cdot 3^{k+1}}\right) $ for $1\leq j\leq
2^{k} $. These numbers are monotonically increasing as the point of
evaluation $u_{j}^{k}\pm c3^{-k}$ moves from left to right across the
interval $[0,1]$. So it suffices to verify that 
\begin{equation}
C_{1}\left( \frac{3}{2}\right) ^{k}\leq R_{1}^{1,2}\widehat{\omega }%
_{N}\left( u_{1}^{k}-c3^{-k},\frac{1}{4\cdot 3^{k+1}}\right) \leq R_{1}^{1,2}%
\widehat{\omega }_{N}\left( u_{2^{k}}^{k}+c3^{-k},\frac{1}{4\cdot 3^{k+1}}%
\right) \leq C_{2}\left( \frac{3}{2}\right) ^{k}  \label{e..H<}
\end{equation}

We consider first the right hand inequality, and write 
\begin{align*}
R_{1}^{1,2}\widehat{\omega }_{N}\left( u_{2^{k}}^{k}+c3^{-k},\frac{1}{4\cdot
3^{k+1}}\right) & =\int_{(H_{2^{k}}^{k})^{c}}\frac{%
a_{2^{k}}^{k}+c3^{-k}-y_{1}}{\left\vert \left( a_{2^{k}}^{k}+c3^{-k},\frac{1%
}{4\cdot 3^{k+1}}\right) -\left( y_{1},y_{2}\right) \right\vert ^{2}}d%
\widehat{\omega }_{N}\left( y_{1},y_{2}\right) \\
& \leq \int \int_{\widehat{\left[ 0,u_{2^{k}}^{k}\right] }}\frac{%
a_{2^{k}}^{k}+c3^{-k}-y_{1}}{\left\vert \left( a_{2^{k}}^{k}+c3^{-k},\frac{1%
}{4\cdot 3^{k+1}}\right) -\left( y_{1},y_{2}\right) \right\vert ^{2}}d%
\widehat{\omega }_{N}\left( y_{1},y_{2}\right) .
\end{align*}%
Here we have discarded that part of the domain of the integral where the
integrand is nonpositive. Now, on the square $\widehat{\left[ 0,u_{2^{k}}^{k}%
\right] }$, the support of $\widehat{\omega }_{N}$ is contained in the set $%
\bigcup_{\ell =1}^{k}\widehat{L_{2^{\ell }-1}^{\ell }}$. Using this, we
continue the estimate above as 
\begin{align*}
R_{1}^{1,2}\widehat{\omega }_{N}\left( u_{2^{k}}^{k}+c3^{-k},\frac{1}{4\cdot
3^{k+1}}\right) & \leq \sum_{\ell =1}^{k}\widehat{\omega }_{N}\left( 
\widehat{L_{2^{\ell }-1}^{\ell }}\right) \sup_{y\in \widehat{I_{2^{\ell
}-1}^{\ell }}}\frac{a_{2^{k}}^{k}+c3^{-k}-y_{1}}{\left\vert \left(
a_{2^{k}}^{k}+c3^{-k},\frac{1}{4\cdot 3^{k+1}}\right) -\left(
y_{1},y_{2}\right) \right\vert ^{2}} \\
& \lesssim c^{-1}\widehat{\omega }_{N}\left( \widehat{L_{2^{\ell }-1}^{\ell }%
}\right) \sup_{y\in \widehat{I_{2^{k}-1}^{k}}}\frac{%
a_{2^{k}}^{k}+c3^{-k}-y_{1}}{\left\vert \left( a_{2^{k}}^{k}+c3^{-k},\frac{1%
}{4\cdot 3^{k+1}}\right) -\left( y_{1},y_{2}\right) \right\vert ^{2}} \\
& \lesssim c^{-1}2^{-k}3^{k}=c^{-1}\left( \frac{3}{2}\right) ^{k}\,.
\end{align*}%
It is useful to record for use below, that in this sum, the summand
associated with $\ell =k$ is the dominant one.

\smallskip Now we consider the left hand inequality in \eqref{e..H<}. We
split the support of $\widehat{\omega }_{N}$ into the sets $I_{1}^{k}\,,\
I_{2}^{k}\,I_{2}^{k-1},\dotsc ,I_{2}^{1}$. By the argument above, we have 
\begin{equation*}
\left\vert \sum_{\ell =1}^{k-1}R_{1}^{1,2}\left( \mathbf{1}_{\widehat{%
L_{2}^{\ell }}}\widehat{\omega }_{N}\right) (a_{1}^{k}+c3^{-k},\frac{1}{%
4\cdot 3^{k+1}})\right\vert \leq A\left( \frac{3}{2}\right) ^{k}\,,
\end{equation*}%
where $A$ is an absolute constant, and we have yet to select the constant $c$%
. But we also have 
\begin{align*}
R_{1}^{1,2}(\mathbf{1}_{\widehat{L_{1}^{k}}\cup \widehat{L_{2}^{k}}}\widehat{%
\omega }_{N})& =\int_{I_{1}^{k}}\left\{ \frac{a_{2^{k}}^{k}+c3^{-k}-y_{1}}{%
\left\vert \left( a_{2^{k}}^{k}+c3^{-k},\frac{1}{4\cdot 3^{k+1}}\right)
-\left( y_{1},y_{2}\right) \right\vert ^{2}}-\frac{a_{2^{k}}^{k}+\left(
1+c\right) 3^{-k}-y_{1}}{\left\vert \left( a_{2^{k}}^{k}+\left( 1+c\right)
3^{-k},\frac{1}{4\cdot 3^{k+1}}\right) -\left( y_{1},y_{2}\right)
\right\vert ^{2}}\right\} d\widehat{\omega }_{N}\left( y\right) \\
& \gtrsim c^{-1}3^{k}{\omega (L_{1}^{k})=}c^{-1}\left( \frac{3}{2}\right)
^{k}.
\end{align*}%
The choice $0<c\ll (2A)^{-1}$ then concludes the proof of Lemma \ref{l.gkj}.
\end{proof}

\section{The $A_{2}^{1}$ Condition}

We recall from (\ref{def A2}) the \emph{one-tailed} \emph{constant with
holes }$\mathcal{A}_{2}^{1}$ in the plane $\mathbb{R}^{2}$ using the
reproducing Poisson kernel $\mathcal{P}^{1}$. Suppose $\sigma $ and $\omega $
are locally finite positive Borel measures on $\mathbb{R}$. Then the
one-tailed constants $\mathcal{A}_{2}^{1}$ and $\mathcal{A}_{2}^{1,\ast }$
with holes for the weight pair $\left( \sigma ,\omega \right) $ are given by%
\begin{eqnarray*}
\mathcal{A}_{2}^{1} &\equiv &\sup_{Q\in \mathcal{P}}\mathcal{P}^{1}\left( Q,%
\mathbf{1}_{Q^{c}}\sigma \right) \frac{\left\vert Q\right\vert _{\omega }}{%
\left\vert Q\right\vert ^{1-\frac{1}{2}}}<\infty , \\
\mathcal{A}_{2}^{1,\ast } &\equiv &\sup_{Q\in \mathcal{P}}\mathcal{P}%
^{1}\left( Q,\mathbf{1}_{Q^{c}}\omega \right) \frac{\left\vert Q\right\vert
_{\sigma }}{\left\vert Q\right\vert ^{1-\frac{1}{2}}}<\infty .
\end{eqnarray*}%
For pairs of measures that share no common point masses, we will also use
the classical Muckenhoupt condition%
\begin{equation*}
A_{2}^{1}\equiv \sup_{Q\in \mathcal{P}}\frac{\left\vert Q\right\vert
_{\omega }\left\vert Q\right\vert _{\sigma }}{\left\vert Q\right\vert
^{2\left( 1-\frac{1}{2}\right) }}<\infty .
\end{equation*}

We will first verify that the $A_{2}^{1}$ condition holds for the weight
pair $\left( \widehat{\omega }_{N},\widehat{\dot{\sigma}}_{N}\right) $. The
same argument will apply to the weight pair $\left( \widehat{\omega }_{N},%
\widehat{\sigma }_{N}\right) $. Recall that we have $\left\vert \widehat{%
L_{i}^{k}}\right\vert _{\widehat{\omega }}\approx 2^{-k}$. Now we use the
definition 
\begin{equation*}
s_{j}^{k}=\left( \frac{1}{3}\right) ^{k}\left( \frac{2}{3}\right) ^{k},
\end{equation*}%
to compute the estimate 
\begin{equation}
\widehat{\dot{\sigma}}_{N}\left( \widehat{L_{r}^{\ell }}\right)
=\sum_{\left( k,j\right) :\ z_{j}^{k}\in L_{r}^{\ell
}}s_{j}^{k}=\sum_{k=\ell }^{\infty }2^{k-\ell }\left( \frac{1}{3}\right)
^{k}\left( \frac{2}{3}\right) ^{k}\approx \left( \frac{1}{3}\right) ^{\ell
}\left( \frac{2}{3}\right) ^{\ell }=s_{r}^{\ell }\,,  \label{sigma measure}
\end{equation}%
by the ratio test since%
\begin{equation}
\frac{2^{k+1}\left( \frac{1}{3}\right) ^{\left( k+1\right) }\left( \frac{2}{3%
}\right) ^{k+1}}{2^{k}\left( \frac{1}{3}\right) ^{k}\left( \frac{2}{3}%
\right) ^{k}}=4\left( \frac{1}{3}\right) ^{2}<1.  \label{ratio}
\end{equation}%
From this, it follows that we have 
\begin{equation}
\frac{\widehat{\dot{\sigma}}_{N}\left( \widehat{L_{j}^{k}}\right) \widehat{%
\omega }_{N}\left( \widehat{L_{j}^{k}}\right) }{\left\vert \widehat{L_{j}^{k}%
}\right\vert }\approx \frac{s_{j}^{k}\widehat{\omega }_{N}\left( \widehat{%
L_{j}^{k}}\right) }{\left\vert I_{j}^{k}\right\vert ^{2}}\approx 1.
\label{e.sA2}
\end{equation}%
The analogous condition $\mathcal{A}_{2}^{1}$ with a tail also holds, namely%
\begin{equation*}
\mathcal{A}_{2}^{1}\left( \widehat{\dot{\sigma}}_{N},\widehat{\omega }%
_{N}\right) =\sup_{Q\in \mathcal{P}^{2}}\mathcal{P}^{1}\left( Q,\widehat{%
\dot{\sigma}}_{N}\right) \mathcal{P}^{1}\left( Q,\widehat{\omega }%
_{N}\right) <\infty ,
\end{equation*}%
where%
\begin{equation*}
\mathcal{P}^{1}\left( Q,\mu \right) \equiv \int_{\mathbb{R}^{2}}\frac{\ell
\left( Q\right) }{\ell \left( Q\right) ^{2}+\left\vert x-c_{Q}\right\vert
^{2}}d\mu \left( x\right) =\int_{\mathbb{R}^{2}}\frac{\left\vert
Q\right\vert ^{\frac{1}{2}}}{\left\vert Q\right\vert +\left\vert
x-c_{Q}\right\vert ^{2}}d\mu \left( x\right) .
\end{equation*}%
Indeed, using $\left\vert \widehat{L_{r}^{\ell }}\right\vert _{\widehat{%
\omega }_{N}}\approx 2^{-\ell }$, one can verify 
\begin{eqnarray*}
\mathcal{P}^{1}\left( \widehat{L_{r}^{\ell }},\widehat{\omega }_{N}\right)
&=&\left\vert \widehat{L_{r}^{\ell }}\right\vert ^{\frac{1}{2}-1}\left\vert 
\widehat{L_{r}^{\ell }}\right\vert _{\widehat{\omega }_{N}}+\sum_{k=1}^{%
\infty }2^{-2k}\left\vert \widehat{L_{r}^{\ell }}\right\vert ^{\frac{1}{2}%
-1}\left\vert 2^{k}\widehat{L_{r}^{\ell }}\setminus 2^{k-1}\widehat{%
L_{r}^{\ell }}\right\vert _{\widehat{\omega }_{N}}\lesssim \frac{\widehat{%
\omega }_{N}\left( \widehat{L_{r}^{\ell }}\right) }{\left\vert \widehat{%
L_{r}^{\ell }}\right\vert ^{1-\frac{1}{2}}}, \\
\mathcal{P}^{\alpha }\left( \widehat{L_{r}^{\ell }},\widehat{\dot{\sigma}}%
_{N}\right) &\lesssim &\sum_{m=0}^{\infty }2^{-2m}\frac{\widehat{\dot{\sigma}%
}_{N}\left( 2^{m}\widehat{L_{r}^{\ell }}\right) }{\left\vert 2^{m}\widehat{%
L_{r}^{\ell }}\right\vert ^{1-\frac{1}{2}}}\lesssim \sum_{m=0}^{\infty
}2^{-2m}\frac{\left( \frac{1}{3}\right) ^{\left( \ell +m\right) }\left( 
\frac{2}{3}\right) ^{\ell +m}}{\left( \frac{1}{3}\right) ^{\left( \ell
+m\right) }} \\
&\lesssim &\frac{\left( \frac{1}{3}\right) ^{\ell }\left( \frac{2}{3}\right)
^{\ell }}{\left( \frac{1}{3}\right) ^{\ell }}\approx \frac{\widehat{\dot{%
\sigma}}_{N}\left( \widehat{L_{r}^{\ell }}\right) }{\left\vert \widehat{%
L_{r}^{\ell }}\right\vert ^{1-\frac{1}{2}}}.
\end{eqnarray*}%
From this and \eqref{e.sA2}, we see that 
\begin{equation*}
\mathcal{P}^{1}\left( \widehat{L_{r}^{\ell }},\widehat{\omega }_{N}\right) 
\mathcal{P}^{1}\left( \widehat{L_{r}^{\ell }},\widehat{\dot{\sigma}}%
_{N}\right) \lesssim \frac{\widehat{\omega }_{N}\left( \widehat{L_{r}^{\ell }%
}\right) \widehat{\dot{\sigma}}_{N}\left( \widehat{L_{r}^{\ell }}\right) }{%
\left\vert \widehat{L_{r}^{\ell }}\right\vert }\lesssim 1\,.
\end{equation*}%
The case of a general square $Q$ now follows easily.

\section{The pivotal and energy conditions}

In this subsection, we show that the backward energy constants with respect
to the weight pairs $\left( \widehat{\sigma }_{N},\widehat{\omega }%
_{N}\right) $ are unbounded in $N$. On the other hand, we show that the
weight pairs $\left( \widehat{\sigma }_{N},\widehat{\omega }_{N}\right) $
satisfy the forward energy condition uniformly in $N\geq 1$. Recall that $%
\widehat{I}\equiv I\times \left[ -\frac{\ell \left( I\right) }{2},\frac{\ell
\left( I\right) }{2}\right] $ is the square centered on the $x_{1}$-axis
whose intersection with the $x_{1}$-axis is the interval $I\subset \mathbb{R}
$. Recall also that%
\begin{equation*}
\mathrm{P}^{1}\left( Q,\mu \right) \equiv \int_{\mathbb{R}^{2}}\frac{%
\left\vert Q\right\vert ^{\frac{1}{2}}}{\left( \left\vert Q\right\vert ^{%
\frac{1}{2}}+\left\vert x-c_{Q}\right\vert \right) ^{2}}d\mu \left( x\right)
,
\end{equation*}%
and that the forward pivotal constant $\mathcal{V}_{2}^{1,\limfunc{overlap}}$
in the plane $\mathbb{R}^{2}$ for the weight pair $\left( \widehat{\sigma }%
_{N},\widehat{\omega }_{N}\right) $ is given by%
\begin{equation}
\left( \mathcal{V}_{2}^{1,\limfunc{overlap}}\left( \widehat{\sigma }_{N},%
\widehat{\omega }_{N}\right) \right) ^{2}\equiv \sup_{\mathcal{D}}\sup 
_{\substack{ I=\dot{\cup}I_{r}  \\ \sum_{r=1}^{\infty }\mathbf{1}_{\gamma
I_{r}}\leq \beta \mathbf{1}_{I}}}\frac{1}{\left\vert I\right\vert _{\widehat{%
\sigma }_{N}}}\sum_{r=1}^{\infty }\mathrm{P}^{\alpha }\left( I_{r},\mathbf{1}%
_{I\setminus \gamma I_{r}}\widehat{\sigma }_{N}\right) ^{2}\left\vert
I_{r}\right\vert _{\widehat{\omega }_{N}}\ .  \label{for piv con}
\end{equation}%
The backward pivotal constant $\mathcal{V}_{2}^{1,\limfunc{overlap}}\left( 
\widehat{\sigma }_{N},\widehat{\omega }_{N}\right) $ is obtained by
interchanging the roles of $\widehat{\sigma }$ and $\widehat{\omega }_{N}$.

\subsection{Failure of the backward pivotal and backward energy conditions 
\label{Subsubfails}}

Failure of the backward pivotal condition uniformly in $N\geq 1$ is
straightforward. Indeed, $I_{1}^{\ell }\subset I_{1}^{\ell -1}\subset
...\subset I_{1}^{0}$ and so 
\begin{equation*}
\mathrm{P}^{1}\left( \widehat{H_{1}^{\ell }},\widehat{\omega }_{N}\right)
\approx \mathrm{P}^{1}\left( \widehat{I_{1}^{\ell }},\widehat{\omega }%
_{N}\right) \approx \sum_{k=0}^{\ell }\frac{\left\vert I_{1}^{\ell
}\right\vert }{\left\vert I_{1}^{k}\right\vert ^{2}}\widehat{\omega }%
_{N}\left( \widehat{I_{r}^{k}}\right) \approx \sum_{k=0}^{\ell }\frac{%
3^{-\ell }}{3^{-2k}}2^{-k}\approx \left( \frac{3}{2}\right) ^{\ell },
\end{equation*}%
and similarly%
\begin{equation*}
\mathrm{P}^{1}\left( \widehat{H_{r}^{\ell }},\widehat{\omega }_{N}\right)
\approx \mathrm{P}^{1}\left( \widehat{I_{r}^{\ell }},\widehat{\omega }%
_{N}\right) \approx \left( \frac{3}{2}\right) ^{\ell },\ \ \ \ \ \text{for
all }r.
\end{equation*}%
We also have $\left\vert \widehat{H_{r}^{\ell }}\right\vert _{\widehat{%
\sigma }_{N}}\approx \left( \frac{1}{3}\right) ^{\ell }\left( \frac{2}{3}%
\right) ^{\ell }$. Considering the decomposition $\overset{\cdot }{\bigcup }%
_{\ell ,r}\widehat{H_{r}^{\ell }}\subset \widehat{\left[ 0,1\right] }=\left[
0,1\right] \times \left[ -\frac{1}{2},\frac{1}{2}\right] $ where the squares 
$\gamma \widehat{H_{r}^{\ell }}$ are contained in $\widehat{G_{r}^{\ell }}$
and the $\widehat{G_{r}^{\ell }}$ are pairwise disjoint, we thus have%
\begin{equation*}
\sum_{\ell ,r}\left\vert \widehat{H_{r}^{\ell }}\right\vert _{\widehat{%
\sigma }_{N}}\mathrm{P}^{1}\left( \widehat{H_{r}^{\ell }},\widehat{\omega }%
_{N}\right) ^{2}\approx \sum_{\ell =0}^{N-1}2^{\ell }\left( \frac{1}{3}%
\right) ^{\ell }\left( \frac{2}{3}\right) ^{\ell }\left( \frac{3}{2}\right)
^{2\ell }\mathbb{\approx }\sum_{\ell =0}^{N-1}1=N,
\end{equation*}%
which shows that the backward pivotal constants $\mathcal{V}_{1}^{\limfunc{%
overlap},\ast }\left( \widehat{\sigma }_{N},\widehat{\omega }_{N}\right) $
are unbounded in $N\geq 1$.

Now we consider the backward energy condition. The sum corresponding to the
above is%
\begin{equation*}
\sum_{\ell ,r}\left\vert \widehat{H_{r}^{\ell }}\right\vert _{\widehat{%
\sigma }_{N}}\mathsf{E}\left( \widehat{H_{r}^{\ell }},\widehat{\sigma }%
_{N}\right) ^{2}\mathrm{P}^{1}\left( \widehat{H_{r}^{\ell }},\widehat{\omega 
}_{N}\right) ^{2}
\end{equation*}%
where 
\begin{equation*}
\mathsf{E}\left( \widehat{H_{r}^{\ell }},\widehat{\sigma }_{N}\right) ^{2}=%
\frac{1}{\left\vert \widehat{H_{r}^{\ell }}\right\vert _{\widehat{\sigma }%
_{N}}}\frac{1}{\left\vert \widehat{H_{r}^{\ell }}\right\vert _{\widehat{%
\sigma }_{N}}}\int_{\widehat{H_{r}^{\ell }}}\int_{\widehat{H_{r}^{\ell }}%
}\left\vert \frac{x-z}{\left\vert H_{r}^{\ell }\right\vert }\right\vert ^{2}d%
\widehat{\sigma }_{N}\left( x\right) \widehat{\sigma }_{N}d\left( z\right)
\approx 1
\end{equation*}%
since $\mathbf{1}_{\widehat{H_{r}^{\ell }}}\widehat{\sigma }_{N}$ consists
of four point masses separated by a distance of approximately $\left\vert 
\widehat{H_{r}^{\ell }}\right\vert ^{\frac{1}{2}}=\left\vert H_{r}^{\ell
}\right\vert $. Thus the backward energy constants fail to be bounded in $%
N\geq 1$ as well.

\subsection{The forward energy condition}

It remains to verify that the pair of measures $\left( \widehat{\sigma }_{N},%
\widehat{\omega }_{N}\right) $ satisfy the forward energy conditions. We
will actually establish the stronger forward pivotal condition (\ref{for piv
con}), which then implies that $\mathcal{E}_{2}^{1}<\infty $. For this it
suffices to show that the forward maximal inequality%
\begin{equation}
\int \mathcal{M}\left( f\widehat{\sigma }_{N}\right) ^{2}d\widehat{\omega }%
_{N}\leq C\int \left\vert f\right\vert ^{2}d\widehat{\sigma }_{N}
\label{M2weight}
\end{equation}%
holds for the pair $\left( \widehat{\sigma }_{N},\widehat{\omega }%
_{N}\right) $, and (\ref{M2weight}) in turn follows from the testing
condition%
\begin{equation}
\int \mathcal{M}\left( \mathbf{1}_{Q}\widehat{\sigma }_{N}\right) ^{2}d%
\widehat{\omega }_{N}\leq C\int_{Q}d\widehat{\sigma }_{N},  \label{testmax}
\end{equation}%
for all squares $Q$ (see \cite{MR676801}). We will show (\ref{testmax}) when 
$Q=\widehat{I_{r}^{\ell }}$, the remaining cases being an easy consequence
of this one. For this we use the fact that%
\begin{equation}
\mathcal{M}\left( \mathbf{1}_{\widehat{I_{r}^{\ell }}}\widehat{\sigma }%
_{N}\right) \left( x\right) \leq C\left( \frac{2}{3}\right) ^{\ell },\ \ \ \
\ x=\left( x_{1},x_{2}\right) \in \limfunc{supp}\widehat{\omega }_{N}.
\label{Msigma bounded}
\end{equation}%
To see (\ref{Msigma bounded}), note that for each $x=\left(
x_{1},x_{2}\right) \in \limfunc{supp}\widehat{\omega }_{N}\cap \widehat{%
L_{r}^{\ell }}$ , we have%
\begin{equation*}
\mathcal{M}\left( \mathbf{1}_{\widehat{L_{r}^{\ell }}}\widehat{\sigma }%
_{N}\right) \left( x\right) \leq \sup_{\left( k,j\right) :x\in \widehat{%
L_{j}^{k}}}\frac{1}{\left\vert \widehat{L_{j}^{k}}\right\vert }\int_{%
\widehat{L_{j}^{k}}\cap \widehat{I_{r}^{\ell }}}d\widehat{\sigma }%
_{N}\approx \sup_{\left( k,j\right) :x\in \widehat{L_{j}^{k}}}\frac{\left( 
\frac{1}{3}\right) ^{k\vee \ell }\left( \frac{2}{3}\right) ^{k\vee \ell }}{%
\left( \frac{1}{3}\right) ^{k}}\approx \left( \frac{2}{3}\right) ^{\ell }.
\end{equation*}%
Thus we have%
\begin{equation*}
\int_{\widehat{L_{r}^{\ell }}}\mathcal{M}\left( \mathbf{1}_{\widehat{%
L_{r}^{\ell }}}\widehat{\sigma }_{N}\right) ^{2}d\widehat{\omega }_{N}\leq
C\left( \frac{2}{3}\right) ^{2\ell }\widehat{\omega }_{N}\left( \widehat{%
L_{r}^{\ell }}\right) \approx C\left( \frac{1}{3}\right) ^{2\ell }\left( 
\frac{1}{2}\right) ^{\ell }=Cs_{r}^{\ell }\approx C\int_{\widehat{%
I_{r}^{\ell }}}d\widehat{\sigma }_{N}.
\end{equation*}%
This yields the case $Q=\widehat{L_{r}^{\ell }}$ of (\ref{testmax}), and
completes our proof of the pivotal condition, and hence also of the forward
energy conditions uniformly in $N\geq 1$.

\section{Testing conditions for the first component $R_{1}^{1,2}$}

In this section we establish both testing conditions for $R_{1}^{1,2}$ with
respect to the weight pairs $\left( \widehat{\sigma }_{N},\widehat{\omega }%
_{N}\right) $ uniformly in $N\geq 1$. We consider first the forward testing
condition.

\subsection{The forward testing condition}

As an initial step in verifying the forward testing condition in (\ref%
{testing conditions}) with respect to the weight pair $\left( \widehat{%
\sigma }_{N},\widehat{\omega }_{N}\right) $, namely 
\begin{equation*}
\mathfrak{T}_{R_{1}^{1,2}}\left( \widehat{\sigma }_{N},\widehat{\omega }%
_{N}\right) \equiv \sup_{\text{squares }I}\sqrt{\frac{1}{\left\vert
I\right\vert _{\widehat{\sigma }_{N}}}\int_{I}\left\vert H_{\widehat{\sigma }%
_{N}}\mathbf{1}_{I}\right\vert ^{2}d\widehat{\omega }_{N}}<\infty ,
\end{equation*}
we replace $\widehat{\sigma }_{N}$ by the self-similar measure $\widehat{%
\dot{\sigma}}_{N}$, and exploit the self-similarity of both measures $%
\widehat{\dot{\sigma}}_{N}$ and $\widehat{\omega }_{N}$ in the following
replicating identities: 
\begin{eqnarray}
\widehat{\omega }_{N} &=&\frac{1}{2}\func{Dil}_{\frac{1}{3}}\widehat{\omega }%
_{N}+\frac{1}{2}\func{Trans}_{\left( \frac{2}{3},0\right) }\func{Dil}_{\frac{%
1}{3}}\widehat{\omega }_{N}\equiv \widehat{\omega }_{N,1}+\widehat{\omega }%
_{N,2},  \label{self-similar} \\
\widehat{\dot{\sigma}}_{N} &=&\frac{2}{9}\func{Dil}_{\frac{1}{3}}\widehat{%
\dot{\sigma}}_{N}+\delta _{\frac{1}{2}}+\frac{2}{9}\func{Trans}_{\left( 
\frac{2}{3},0\right) }\func{Dil}_{\frac{1}{3}}\widehat{\dot{\sigma}}%
_{N}\equiv \widehat{\dot{\sigma}}_{N,1}+\delta _{\frac{1}{2}}+\widehat{\dot{%
\sigma}}_{N,2},
\end{eqnarray}%
where $\func{Trans}_{\left( \frac{2}{3},0\right) }$ is translation in the
plane by the vector $\left( \frac{2}{3},0\right) $, and $\func{Dil}_{\frac{1%
}{3}}$ is dilation in the plane by the factor $\frac{1}{3}$. But now we note
that%
\begin{eqnarray*}
\int \left\vert R_{1}^{1,2}\widehat{\dot{\sigma}}_{N,1}\right\vert ^{2}%
\widehat{\omega }_{N,1} &=&\frac{1}{2}\int \left\vert R_{1}^{1,2}\widehat{%
\dot{\sigma}}_{N,1}\left( x\right) \right\vert ^{2}\func{Dil}_{\frac{1}{3}}%
\widehat{\omega }_{N}\left( x\right) =\frac{1}{2}\int \left\vert R_{1}^{1,2}%
\widehat{\dot{\sigma}}_{N,1}\left( \frac{x}{3}\right) \right\vert ^{2}%
\widehat{\omega }_{N}\left( x\right) \\
&=&\frac{1}{2}\int \left\vert \int \frac{z_{1}-\frac{x_{1}}{3}}{\left( z_{1}-%
\frac{x_{1}}{3}\right) ^{2}+\left( z_{2}-\frac{x_{2}}{3}\right) ^{2}}\frac{2%
}{9}\func{Dil}_{\frac{1}{3}}\widehat{\dot{\sigma}}_{N}\left( z\right)
\right\vert ^{2}\widehat{\omega }_{N}\left( x\right) \\
&=&\frac{1}{2}\left( \frac{2}{9}\right) ^{2}\int \left\vert \int \frac{\frac{%
z_{1}}{3}-\frac{x_{1}}{3}}{\left( \frac{z_{1}}{3}-\frac{x_{1}}{3}\right)
^{2}+\left( \frac{z_{2}}{3}-\frac{x_{2}}{3}\right) ^{2}}\widehat{\dot{\sigma}%
}_{N}\left( z\right) \right\vert ^{2}\widehat{\omega }_{N}\left( x\right) \\
&=&\frac{1}{2}\left( \frac{2}{9}\right) ^{2}9\int \left\vert R_{1}^{1,2}%
\widehat{\dot{\sigma}}_{N}\left( x\right) \right\vert ^{2}\widehat{\omega }%
_{N}\left( x\right) =\frac{2}{9}\int \left\vert R_{1}^{1,2}\widehat{\dot{%
\sigma}}_{N}\right\vert ^{2}\widehat{\omega }_{N},
\end{eqnarray*}%
and similarly $\int \left\vert R_{1}^{1,2}\widehat{\dot{\sigma}}%
_{N,2}\right\vert ^{2}\widehat{\omega }_{N,2}=\frac{2}{9}\int \left\vert
R_{1}^{1,2}\widehat{\dot{\sigma}}_{N}\right\vert ^{2}\widehat{\omega }_{N}$.
We then claim as in \cite{LaSaUr2} that%
\begin{eqnarray*}
\int \left\vert R_{1}^{1,2}\widehat{\dot{\sigma}}_{N}\right\vert ^{2}\omega
&=&\int \left\vert R_{1}^{1,2}\left( \widehat{\dot{\sigma}}_{N,1}+\delta _{%
\frac{1}{2}}+\widehat{\dot{\sigma}}_{N,2}\right) \right\vert ^{2}\widehat{%
\omega }_{N,1}+\int \left\vert R_{1}^{1,2}\left( \widehat{\dot{\sigma}}%
_{N,1}+\delta _{\frac{1}{2}}+\widehat{\dot{\sigma}}_{N,2}\right) \right\vert
^{2}\widehat{\omega }_{N,2} \\
&=&\left( 1+\varepsilon \right) \left\{ \int \left\vert R_{1}^{1,2}\widehat{%
\dot{\sigma}}_{N,1}\right\vert ^{2}\widehat{\omega }_{N,1}+\int \left\vert
R_{1}^{1,2}\widehat{\dot{\sigma}}_{N,2}\right\vert ^{2}\widehat{\omega }%
_{N,2}\right\} +\mathcal{R}_{\varepsilon },
\end{eqnarray*}%
for any $\varepsilon >0$ where the remainder term $\mathcal{R}_{\varepsilon
} $ is easily seen to satisfy 
\begin{equation*}
\mathcal{R}_{\varepsilon }\lesssim _{\varepsilon }\mathcal{A}_{2}^{2}\left(
\int \widehat{\dot{\sigma}}_{N}\right) ,
\end{equation*}%
since the supports of $\delta _{\frac{1}{2}}+\widehat{\dot{\sigma}}_{N,2}$
and $\widehat{\omega }_{N,1}$ are well separated, as are those of $\delta _{%
\frac{1}{2}}+\widehat{\dot{\sigma}}_{N,1}$ and $\widehat{\omega }_{N,2}$.
This is proved exactly as in \cite{LaSaUr2} so we will be brief. We have%
\begin{eqnarray*}
&&\int \left\vert R_{1}^{1,2}\left( \widehat{\dot{\sigma}}_{N,1}+\delta _{%
\frac{1}{2}}+\widehat{\dot{\sigma}}_{N,2}\right) \right\vert ^{2}\widehat{%
\omega }_{N,1} \\
&\lesssim &\int \left\{ \left( 1+\varepsilon \right) \left\vert R_{1}^{1,2}%
\widehat{\dot{\sigma}}_{N,1}\right\vert ^{2}+\left( 1+\frac{1}{\varepsilon }%
\right) \left\vert R_{1}^{1,2}\left( \delta _{\frac{1}{2}}+\widehat{\dot{%
\sigma}}_{N,2}\right) \right\vert ^{2}\right\} \widehat{\omega }_{N,1},
\end{eqnarray*}%
and then,%
\begin{equation*}
\int \left\vert R_{1}^{1,2}\widehat{\dot{\sigma}}_{N,2}\right\vert ^{2}%
\widehat{\omega }_{N,1}=\int_{\left[ 0,\frac{1}{3}\right] }\left\vert \int_{%
\left[ \frac{2}{3},1\right] }\frac{1}{y-x}\widehat{\dot{\sigma}}_{N,2}\left(
y\right) \right\vert ^{2}\widehat{\omega }_{N}\left( x\right) \lesssim 
\mathcal{A}_{2}^{2}\int \widehat{\dot{\sigma}}_{N,2}.
\end{equation*}

But now we note that%
\begin{equation*}
\int \left\vert R_{1}^{1,2}\widehat{\dot{\sigma}}_{N,1}\right\vert ^{2}%
\widehat{\omega }_{N,1}=\frac{2}{9}\int \left\vert R_{1}^{1,2}\widehat{\dot{%
\sigma}}_{N}\right\vert ^{2}\widehat{\omega }_{N},
\end{equation*}%
and similarly $\int \left\vert R_{1}^{1,2}\widehat{\dot{\sigma}}%
_{N,2}\right\vert ^{2}\widehat{\omega }_{N,2}=\frac{2}{9}\int \left\vert
R_{1}^{1,2}\widehat{\dot{\sigma}}_{N}\right\vert ^{2}\widehat{\omega }_{N}$.
Thus we have

\begin{equation}
\int \left\vert R_{1}^{1,2}\widehat{\dot{\sigma}}_{N}\right\vert ^{2}%
\widehat{\omega }_{N}=\frac{2}{9}\left( 1+\varepsilon \right) \int
\left\vert R_{1}^{1,2}\widehat{\dot{\sigma}}_{N}\right\vert ^{2}\widehat{%
\omega }_{N}+\frac{2}{9}\left( 1+\varepsilon \right) \int \left\vert
R_{1}^{1,2}\widehat{\dot{\sigma}}_{N}\right\vert ^{2}\widehat{\omega }_{N}+%
\mathcal{R}_{\varepsilon },  \label{reproduce}
\end{equation}%
and since $\int \left\vert R_{1}^{1,2}\widehat{\dot{\sigma}}_{N}\right\vert
^{2}\widehat{\omega }_{N}$ is finite we conclude that%
\begin{equation*}
\int \left\vert R_{1}^{1,2}\widehat{\dot{\sigma}}_{N}\right\vert ^{2}%
\widehat{\omega }_{N}=\frac{1}{1-\frac{4}{9}\left( 1+\varepsilon \right) }%
\mathcal{R}_{\varepsilon }\lesssim _{\varepsilon }\mathcal{A}_{2}^{2}\left(
\int \widehat{\dot{\sigma}}_{N}\right) ,
\end{equation*}%
for $\varepsilon >0$ so small that $1-\frac{4}{9}\left( 1+\varepsilon
\right) >0$. This completes the proof of the forward testing condition in (%
\ref{testing conditions}) for the cube $\widehat{I}=\widehat{\left[ 0,1%
\right] }$ with respect to the weight pair $\left( \widehat{\dot{\sigma}}%
_{N},\widehat{\omega }_{N}\right) $ uniformly in $N\geq 1$. The proof for
the case $\widehat{L}=\widehat{L_{j}^{k}}$ is similar using $\mathcal{R}%
_{\varepsilon }\left( \widehat{L_{j}^{k}}\right) \leq C_{\varepsilon }%
\mathcal{A}_{2}^{2}\left( \int_{\widehat{L_{j}^{k}}}\dot{\sigma}\right) $,
and the general case now follows without much extra work.

Having verified the forward testing condition for the weight pair $\left( 
\widehat{\dot{\sigma}}_{N},\widehat{\omega }_{N}\right) $, we now show that
the forward testing condition in (\ref{testing conditions}) holds for $%
\left( \widehat{\sigma }_{N},\widehat{\omega }_{N}\right) $. For this, we
estimate the difference as in \cite{LaSaUr2}$\ $by%
\begin{equation*}
\int_{\widehat{L_{r}^{\ell }}}\left\vert R_{1}^{1,2}\mathbf{1}_{\widehat{%
L_{r}^{\ell }}}\left( \widehat{\sigma }_{N}-\widehat{\dot{\sigma}}%
_{N}\right) \right\vert ^{2}\widehat{\omega }_{N}\lesssim C\int_{\widehat{%
L_{r}^{\ell }}}\left\vert \sum_{\left( k,j\right) :z_{j}^{k}\in I_{r}^{\ell
}}s_{j}^{k}\left( \frac{\left\vert I_{j}^{k}\right\vert }{\left\vert
x-\left( z_{j}^{k},\pm \frac{1}{4\cdot 3^{k+1}}\right) \right\vert ^{2}}%
\right) \right\vert ^{2}\widehat{\omega }_{N}\left( x\right) .
\end{equation*}%
Now for any fixed $x$ in the support of $\widehat{\omega }_{N}$ inside $%
\widehat{L_{r}^{\ell }}$, we have just as in \cite{LaSaUr2} that%
\begin{equation*}
\sum_{\left( k,j\right) :z_{j}^{k}\in I_{r}^{\ell }}s_{j}^{k}\left( \frac{%
\left\vert I_{j}^{k}\right\vert }{\left\vert x-\left( z_{j}^{k},\pm \frac{1}{%
4\cdot 3^{k+1}}\right) \right\vert ^{2}}\right) \lesssim \left( \frac{2}{3}%
\right) ^{\ell }.
\end{equation*}%
Thus we get%
\begin{equation*}
\int_{\widehat{L_{r}^{\ell }}}\left\vert R_{1}^{1,2}\mathbf{1}_{\widehat{%
L_{r}^{\ell }}}\left( \widehat{\sigma }_{N}-\widehat{\dot{\sigma}}%
_{N}\right) \right\vert ^{2}\widehat{\omega }_{N}\lesssim \left( \frac{2}{3}%
\right) ^{2\ell }\widehat{\omega }_{N}(L_{r}^{\ell })=C^{2}\left( \frac{2}{3}%
\right) ^{2\ell }2^{-\ell }\approx \widehat{\sigma }_{N}(\widehat{%
L_{r}^{\ell }}),
\end{equation*}%
which yields%
\begin{equation*}
\left( \int_{\widehat{L_{r}^{\ell }}}\left\vert R_{1}^{1,2}\mathbf{1}_{%
\widehat{L_{r}^{\ell }}}\widehat{\sigma }_{N}\right\vert ^{2}\widehat{\omega 
}_{N}\right) ^{\frac{1}{2}}\lesssim \left( \int_{\widehat{L_{r}^{\ell }}%
}\left\vert R_{1}^{1,2}\mathbf{1}_{\widehat{L_{r}^{\ell }}}\widehat{\dot{%
\sigma}}_{N}\right\vert ^{2}\widehat{\omega }_{N}\right) ^{\frac{1}{2}%
}+\left( \int_{\widehat{L_{r}^{\ell }}}\left\vert R_{1}^{1,2}\mathbf{1}_{%
\widehat{L_{r}^{\ell }}}\left( \widehat{\sigma }_{N}-\widehat{\dot{\sigma}}%
_{N}\right) \right\vert ^{2}\widehat{\omega }_{N}\right) ^{\frac{1}{2}%
}\lesssim C\sqrt{\widehat{\sigma }_{N}(\widehat{L_{r}^{\ell }})}.
\end{equation*}%
This is the case $I=I_{r}^{\ell }$ of the forward testing condition in (\ref%
{testing conditions}) for the weight pair $\left( \widehat{\sigma }_{N},%
\widehat{\omega }_{N}\right) $ uniformly in $N\geq 1$, and the general case
follows from this by an additional argument. This additional argument is
given explicitly in a related situation in \cite[Subsubsection 5.2.4 \emph{%
Completion of the proof for general intervals}]{SaShUr11}, to which we refer
the reader for details.

\subsection{The backward testing condition}

Finally, we turn to the dual testing condition for $R_{1}^{1,2}$ in (\ref%
{testing conditions}) with respect to the weight pair $\left( \widehat{%
\sigma }_{N},\widehat{\omega }_{N}\right) $, namely 
\begin{equation*}
\mathfrak{T}_{R_{1}^{1,2}}^{\ast }\equiv \sup_{\text{squares }I}\sqrt{\frac{1%
}{\left\vert I\right\vert _{\widehat{\omega }_{N}}}\int_{I}\left\vert H_{%
\widehat{\omega }_{N}}\mathbf{1}_{I}\right\vert ^{2}d\widehat{\sigma }_{N}}%
<\infty .
\end{equation*}%
For an interval $L_{r}^{\ell }$ with $z_{j}^{k}\in L_{r}^{\ell }$, we claim
that 
\begin{equation}
\left\vert R_{1}^{1,2}\left( \mathbf{1}_{\widehat{L_{r}^{\ell }}}\widehat{%
\omega }_{N}\right) \left( z_{j}^{k},\pm \frac{1}{4\cdot 3^{k+1}}\right)
\right\vert \lesssim \mathrm{P}^{1}\left( \widehat{L_{r}^{\ell }},\widehat{%
\omega }_{N}\right) .  \label{claim Poisson}
\end{equation}%
This is a substantial improvement over the estimate $\left\vert
R_{1}^{1,2}\left( \mathbf{1}_{\widehat{L_{r}^{\ell }}}\widehat{\omega }%
_{N}\right) \left( c_{j}^{k},0\right) \right\vert \lesssim \left( \frac{3}{2}%
\right) ^{k}$, and is a consequence of the fact that the points $\left(
z_{j}^{k},\pm \frac{1}{4\cdot 3^{k+1}}\right) $ are \emph{zeroes} of the
function $R_{1}^{1,2}\widehat{\omega }_{N}$. To see (\ref{claim Poisson})
let $I_{s}^{\ell -1}$ denote the parent of $I_{r}^{\ell }$, and let $%
I_{r+1}^{\ell }$ denote the other child of $I_{s}^{\ell -1}$. Then we have
using $R_{1}^{1,2}\widehat{\omega }_{N}\left( z_{j}^{k},\pm \frac{1}{4\cdot
3^{k+1}}\right) =0$, 
\begin{eqnarray*}
R_{1}^{1,2}\left( \mathbf{1}_{\widehat{L_{r}^{\ell }}}\widehat{\omega }%
_{N}\right) \left( z_{j}^{k},\pm \frac{1}{4\cdot 3^{k+1}}\right)
&=&-R_{1}^{1,2}\left( \mathbf{1}_{\left( \widehat{L_{r}^{\ell }}\right) ^{c}}%
\widehat{\omega }_{N}\right) \left( z_{j}^{k},\pm \frac{1}{4\cdot 3^{k+1}}%
\right) \\
&=&-R_{1}^{1,2}\left( \mathbf{1}_{\left( \widehat{L_{s}^{\ell -1}}\right)
^{c}}\omega \right) \left( z_{j}^{k},\pm \frac{1}{4\cdot 3^{k+1}}\right)
-R_{1}^{1,2}\left( \mathbf{1}_{\widehat{L}_{r+1}^{\ell }}\omega \right)
\left( z_{j}^{k},\pm \frac{1}{4\cdot 3^{k+1}}\right) .
\end{eqnarray*}%
Now we have using $R_{1}^{1,2}\widehat{\omega }_{N}\left( {z_{r}^{\ell }}%
,\pm \frac{1}{4\cdot 3^{\ell +1}}\right) =0$ that%
\begin{eqnarray*}
&&R_{1}^{1,2}\left( \mathbf{1}_{\left( \widehat{L_{s}^{\ell -1}}\right) ^{c}}%
\widehat{\omega }_{N}\right) \left( z_{j}^{k},\pm \frac{1}{4\cdot 3^{k+1}}%
\right) =R_{1}^{1,2}\left( \mathbf{1}_{\left( \widehat{L_{s}^{\ell -1}}%
\right) ^{c}}\widehat{\omega }_{N}\right) \left( {z_{r}^{\ell }},\pm \frac{1%
}{4\cdot 3^{\ell +1}}\right) \\
&&\ \ \ \ \ \ \ \ \ \ \ \ \ \ \ \ \ \ \ \ -\left\{ R_{1}^{1,2}\left( \mathbf{%
1}_{\left( \widehat{L_{s}^{\ell -1}}\right) ^{c}}\widehat{\omega }%
_{N}\right) \left( {z_{r}^{\ell }},\pm \frac{1}{4\cdot 3^{\ell +1}}\right)
-R_{1}^{1,2}\left( \mathbf{1}_{\left( \widehat{L_{s}^{\ell -1}}\right) ^{c}}%
\widehat{\omega }_{N}\right) \left( z_{j}^{k},\pm \frac{1}{4\cdot 3^{k+1}}%
\right) \right\} \\
&=&-R_{1}^{1,2}\left( \mathbf{1}_{\widehat{I_{s}^{\ell -1}}}\widehat{\omega }%
_{N}\right) \left( {z_{r}^{\ell }},\pm \frac{1}{4\cdot 3^{\ell +1}}\right)
-A,
\end{eqnarray*}%
where 
\begin{equation*}
A\equiv R_{1}^{1,2}\left( \mathbf{1}_{\left( \widehat{L_{s}^{\ell -1}}%
\right) ^{c}}\widehat{\omega }_{N}\right) \left( {z_{r}^{\ell }},\pm \frac{1%
}{4\cdot 3^{\ell +1}}\right) -R_{1}^{1,2}\left( \mathbf{1}_{\left( \widehat{%
L_{s}^{\ell -1}}\right) ^{c}}\widehat{\omega }_{N}\right) \left(
z_{j}^{k},\pm \frac{1}{4\cdot 3^{k+1}}\right) .
\end{equation*}%
Combining equalities yields 
\begin{eqnarray*}
&&R_{1}^{1,2}\left( \mathbf{1}_{\widehat{L_{r}^{\ell }}}\widehat{\omega }%
_{N}\right) \left( z_{j}^{k},\pm \frac{1}{4\cdot 3^{k+1}}\right) \\
&=&R_{1}^{1,2}\left( \mathbf{1}_{\left( \widehat{L_{s}^{\ell -1}}\right)
^{c}}\widehat{\omega }_{N}\right) \left( z_{j}^{k},\pm \frac{1}{4\cdot
3^{k+1}}\right) +A-R_{1}^{1,2}\left( \mathbf{1}_{\widehat{L}_{r+1}^{\ell
}}\omega \right) \left( z_{j}^{k},\pm \frac{1}{4\cdot 3^{k+1}}\right) .
\end{eqnarray*}%
We then have for $\left( k,j\right) $ such that $z_{j}^{k}\in I_{r}^{\ell }$,%
\begin{eqnarray*}
&&\left\vert R_{1}^{1,2}\left( \mathbf{1}_{\widehat{L_{s}^{\ell -1}}}%
\widehat{\omega }_{N}\right) \left( {z_{r}^{\ell }},\pm \frac{1}{4\cdot
3^{\ell +1}}\right) \right\vert \lesssim \frac{\widehat{\omega }_{N}(%
\widehat{L_{s}^{\ell -1}})}{\left\vert {I_{s}^{\ell -1}}\right\vert }, \\
&&\left\vert A\right\vert \lesssim \int_{\left( \widehat{L_{s}^{\ell -1}}%
\right) ^{c}}\left\vert K\left( x-\left( {z_{r}^{\ell }},\pm \frac{1}{4\cdot
3^{\ell +1}}\right) \right) -K\left( x-\left( z_{j}^{k},\pm \frac{1}{4\cdot
3^{k+1}}\right) \right) \right\vert \widehat{\omega }_{N}\left( x\right) \\
&&\ \ \ \ \ \ \ \ \ \ \ \ \ \ \ \ \ \ \ \ \ \ \ \ \ \ \ \ \ \ \lesssim
\int_{\left( \widehat{I_{s}^{\ell -1}}\right) ^{c}}\frac{\left\vert
I_{r}^{\ell }\right\vert }{\left\vert x-\left( z_{s}^{\ell -1},\pm \frac{1}{%
4\cdot 3^{\ell }}\right) \right\vert ^{2}}\widehat{\omega }_{N}\left(
x\right) , \\
&&\left\vert R_{1}^{1,2}\left( \mathbf{1}_{\widehat{L_{r+1}^{\ell }}}%
\widehat{\omega }_{N}\right) \left( z_{j}^{k},\pm \frac{1}{4\cdot 3^{k+1}}%
\right) \right\vert \lesssim \frac{\widehat{\omega }_{N}(L_{s}^{\ell -1})}{%
\left\vert {L_{s}^{\ell -1}}\right\vert },
\end{eqnarray*}%
which proves \eqref{claim Poisson}.

Now we compute, using \eqref{claim Poisson} and the estimate $\mathrm{P}%
^{1}\left( I_{r}^{\ell },\omega \right) \lesssim \frac{\widehat{\omega }_{N}(%
\widehat{L_{r}^{\ell }})}{\left\vert {L_{r}^{\ell }}\right\vert }$ proved
above, that%
\begin{eqnarray*}
\widehat{\int_{L_{r}^{\ell }}}\left\vert R_{1}^{1,2}\left( \mathbf{1}_{%
\widehat{L_{r}^{\ell }}}\widehat{\omega }_{N}\right) \right\vert ^{2}d%
\widehat{{\sigma }}_{N} &=&\sum_{\left( k,j\right) :z_{j}^{k}\in I_{r}^{\ell
}}\left\vert R_{1}^{1,2}\left( \mathbf{1}_{\widehat{L_{r}^{\ell }}}\widehat{%
\omega }_{N}\right) \left( z_{j}^{k},\pm \frac{1}{4\cdot 3^{k+1}}\right)
\right\vert ^{2}s_{j}^{k}\leq C\sum_{\left( k,j\right) :z_{j}^{k}\in
I_{r}^{\ell }}\left\vert \mathrm{P}^{1}\left( \widehat{L_{r}^{\ell }},%
\widehat{\omega }_{N}\right) \right\vert ^{2}s_{j}^{k} \\
&\lesssim &\widehat{{\sigma }}_{N}(\widehat{L_{r}^{\ell }})\left( \frac{%
\widehat{\omega }_{N}(L_{r}^{\ell })}{\left\vert L_{r}^{\ell }\right\vert }%
\right) ^{2}\lesssim \mathcal{A}_{2}\widehat{\omega }_{N}(L_{r}^{\ell }).
\end{eqnarray*}%
This is the case $I=I_{r}^{\ell }$ of the dual testing condition\ in (\ref%
{testing conditions}) for the weight pair $\left( \widehat{\sigma }_{N},%
\widehat{\omega }_{N}\right) $ uniformly in $N\geq 1$, and the general case
follows from this, just as for the forward testing condition above, using
the additional argument given in \cite[Subsubsection 5.2.4 \emph{Completion
of the proof for general intervals}]{SaShUr11}, but adapted to the backward
testing condition.

\section{Testing conditions for the second component $R_{2}^{1,2}$}

In this section we establish both forward and backward testing conditions
for $R_{2}^{1,2}$ with respect to the weight pairs $\left( \widehat{\sigma }%
_{N},\widehat{\omega }_{N}\right) $ uniformly in $N\geq 1$. Recall that $%
\mathbf{K}^{1,2}\left( w_{1},w_{2}\right) =\left( \frac{w_{1}}{%
w_{1}^{2}+w_{2}^{2}},\frac{w_{2}}{w_{1}^{2}+w_{2}^{2}}\right) $, and $%
s_{i}^{k}=\left( \frac{1}{3}\right) ^{k}\left( \frac{2}{3}\right) ^{k}$. We
consider first the backward testing condition for $R_{2}^{1,2}$.

\subsection{The backward testing condition}

Recall the estimate (\ref{estimate}), 
\begin{equation*}
R_{2}^{1,2}\widehat{\omega }_{N}\left( z_{i}^{k},\frac{1}{4\cdot 3^{k+1}}%
\right) \approx 2^{-N}3^{k},
\end{equation*}%
and the definitions 
\begin{equation*}
\widehat{\sigma }_{N}^{+}=\sum_{k=0}^{N-1}\sum_{i=1}^{2^{k}}s_{i}^{k}\delta
_{\left( z_{i}^{\ell },\frac{1}{4\cdot 3^{\ell +1}}\right) }\ ,
\end{equation*}%
and%
\begin{eqnarray*}
\widehat{\omega }_{N} &=&2^{-N-2}\sum_{k=0}^{N-1}\sum_{i=1}^{2^{k}}\left(
\delta _{\left( c_{\limfunc{left}}^{N+1}\left( G_{i}^{k}\right) +\frac{1}{%
4\cdot 3^{k+1}},\frac{1}{4\cdot 3^{k+1}}\right) }+\delta _{\left( c_{%
\limfunc{left}}^{N+1}\left( G_{i}^{k}\right) +\frac{1}{4\cdot 3^{k+1}},-%
\frac{1}{4\cdot 3^{k+1}}\right) }\right) \\
&&+2^{-N-2}\sum_{k=0}^{N-1}\sum_{i=1}^{2^{k}}\left( \delta _{\left( c_{%
\limfunc{right}}^{N+1}\left( G_{i}^{k}\right) -\frac{1}{4\cdot 3^{k+1}},%
\frac{1}{4\cdot 3^{k+1}}\right) }+\delta _{\left( c_{\limfunc{right}%
}^{N+1}\left( G_{i}^{k}\right) -\frac{1}{4\cdot 3^{k+1}},-\frac{1}{4\cdot
3^{k+1}}\right) }\right) .
\end{eqnarray*}%
Thus we have the estimate%
\begin{eqnarray*}
\int \int_{\widehat{\left[ 0,1\right] }}\left\vert R_{2}^{1,2}\left( \mathbf{%
1}_{\widehat{\left[ 0,1\right] }}\widehat{\omega }_{N}\right) \right\vert
^{2}d\widehat{\sigma }_{N}
&=&\sum_{k=0}^{N-1}\sum_{i=1}^{2^{k}}s_{i}^{k}\left\vert R_{2}^{1,2}\widehat{%
\omega }_{N}\left( z_{i}^{k},\frac{1}{4\cdot 3^{k+1}}\right) \right\vert ^{2}
\\
&\approx &\sum_{k=0}^{N-1}\sum_{i=1}^{2^{k}}\left( \frac{1}{3}\right)
^{k}\left( \frac{2}{3}\right) ^{k}\left( 2^{-N}3^{k}\right) ^{2} \\
&=&2^{-2N}\sum_{k=0}^{N-1}2^{2k}\approx 1=\left\vert \widehat{\left[ 0,1%
\right] }\right\vert _{\widehat{\omega }_{N}}.
\end{eqnarray*}

More generally, for any square $\widehat{L_{j}^{k}}$, we have%
\begin{equation*}
\int \int_{\widehat{L_{j}^{\ell }}}\left\vert R_{2}^{1,2}\left( \mathbf{1}_{%
\widehat{L_{j}^{\ell }}}\widehat{\omega }_{N}\right) \right\vert ^{2}d%
\widehat{\sigma }_{N}=\sum_{k=\ell }^{N-1}\sum_{i:\ L_{i}^{k}\subset
L_{j}^{\ell }}\left( \frac{1}{3}\right) ^{k}\left( \frac{2}{3}\right)
^{k}\left( 2^{-N}3^{k}\right) ^{2}=\sum_{k=\ell }^{N-1}2^{k-\ell
}2^{k-2N}\approx 2^{-\ell }\approx \left\vert \widehat{L_{j}^{\ell }}%
\right\vert _{\widehat{\omega }_{N}},
\end{equation*}%
and the general case follows easily.

\subsection{The forward testing condition}

Now we turn to the forward testing condition for $R_{2}^{1,2}$, namely the
inequality%
\begin{equation*}
\int \int_{Q}\left\vert R_{2}^{1,2}\left( \mathbf{1}_{Q}\widehat{\sigma }%
_{N}\right) \right\vert ^{2}d\widehat{\omega }_{N}\lesssim \left\vert
Q\right\vert _{\widehat{\sigma }_{N}},\ \ \ \ \ \text{for all squares }%
Q\subset \widehat{\left[ 0,1\right] },
\end{equation*}%
and just as for the first component $R_{1}^{1,2}$ above, we first prove the
forward testing condition for the measure pair $\left( \widehat{\dot{\sigma}}%
_{N},\widehat{\omega }_{N}\right) $,%
\begin{equation*}
\int \int_{Q}\left\vert R_{2}^{1,2}\left( \mathbf{1}_{Q}\widehat{\dot{\sigma}%
}_{N}\right) \right\vert ^{2}d\widehat{\omega }_{N}\lesssim \left\vert
Q\right\vert _{\widehat{\dot{\sigma}}_{N}},\ \ \ \ \ \text{for all squares }%
Q\subset \widehat{\left[ 0,1\right] }.
\end{equation*}%
We have already noted in (\ref{self-similar}) the replication identities
satisfied by the measure pair $\left( \widehat{\dot{\sigma}}_{N},\widehat{%
\omega }_{N}\right) $. Just as for $R_{1}^{1,2}$ we compute for the operator 
$R_{2}^{1,2}$ that 
\begin{equation*}
\int \left\vert R_{2}^{1,2}\widehat{\dot{\sigma}}_{N,2}\right\vert ^{2}%
\widehat{\omega }_{N,1}=\frac{2}{9}\int \left\vert R_{2}^{1,2}\widehat{\dot{%
\sigma}}_{N}\right\vert ^{2}\widehat{\omega }_{N}=\int \left\vert R_{2}^{1,2}%
\widehat{\dot{\sigma}}_{N,2}\right\vert ^{2}\widehat{\omega }_{N,2}\ ,
\end{equation*}%
and then that%
\begin{eqnarray*}
\int \left\vert R_{2}^{1,2}\widehat{\dot{\sigma}}_{N}\right\vert ^{2}\omega
&=&\int \left\vert R_{2}^{1,2}\left( \widehat{\dot{\sigma}}_{N,1}+\delta _{%
\frac{1}{2}}+\widehat{\dot{\sigma}}_{N,2}\right) \right\vert ^{2}\widehat{%
\omega }_{N,1}+\int \left\vert R_{2}^{1,2}\left( \widehat{\dot{\sigma}}%
_{N,1}+\delta _{\frac{1}{2}}+\widehat{\dot{\sigma}}_{N,2}\right) \right\vert
^{2}\widehat{\omega }_{N,2} \\
&=&\left( 1+\varepsilon \right) \left\{ \int \left\vert R_{2}^{1,2}\widehat{%
\dot{\sigma}}_{N,1}\right\vert ^{2}\widehat{\omega }_{N,1}+\int \left\vert
R_{2}^{1,2}\widehat{\dot{\sigma}}_{N,2}\right\vert ^{2}\widehat{\omega }%
_{N,2}\right\} +\mathcal{R}_{\varepsilon },
\end{eqnarray*}%
for $\varepsilon >0$ where the remainder term $\mathcal{R}_{\varepsilon }$
is easily seen to satisfy 
\begin{equation*}
\mathcal{R}_{\varepsilon }\lesssim _{\varepsilon }\mathcal{A}_{2}^{2}\left(
\int \widehat{\dot{\sigma}}_{N}\right) ,
\end{equation*}%
since the supports of $\delta _{\frac{1}{2}}+\widehat{\dot{\sigma}}_{N,2}$
and $\widehat{\omega }_{N,1}$ are well separated, as are those of $\delta _{%
\frac{1}{2}}+\widehat{\dot{\sigma}}_{N,1}$ and $\widehat{\omega }_{N,2}$.
Now we simply proceed as before, and leave the details to the intereseted
reader. Finally, just as we did for $R_{1}^{1,2}$ above, we use a
perturbation argument to obtain the forward testing condition for the
measure pair $\left( \widehat{\sigma }_{N},\widehat{\omega }_{N}\right) $,
uniformly in $N\geq 1$.

\section{The norm inequality}

Here we show that the norm inequality for $\mathbf{R}^{1,2}$ holds with
respect to the weight pair $\left( \widehat{\sigma }_{N},\widehat{\omega }%
_{N}\right) $ uniformly in $N$. We first observe that we have already
established above the following facts for the weight pairs $\left( \widehat{%
\sigma }_{N}^{+},\widehat{\omega }_{N}\right) $ uniformly in $N\geq 1$. Let $%
\widehat{\sigma }_{N}^{-}$ denote the reflection of $\widehat{\sigma }%
_{N}^{+}$ across the $x_{1}$-axis.

\begin{enumerate}
\item The Muckenhoupt/NTV condition $\mathcal{A}_{2}$ holds:%
\begin{equation*}
\sup_{Q\in \mathcal{P}^{2}}\left\{ \frac{\left\vert Q\right\vert _{\widehat{%
\sigma }_{N}^{+}}}{\sqrt{\left\vert Q\right\vert }}\cdot \mathcal{P}^{1}(Q,%
\widehat{\omega }_{N})+\mathcal{P}^{1}(Q,\widehat{\sigma }_{N}^{+})\frac{%
\left\vert Q\right\vert _{\widehat{\omega }_{N}}}{\sqrt{\left\vert
Q\right\vert }}\right\} =\mathcal{A}_{2}<\infty .
\end{equation*}

\item The forward testing condition holds:%
\begin{equation}
\int_{Q}\left\vert \mathbf{R}^{1,2}\left( \mathbf{1}_{Q}\widehat{\sigma }%
_{N}^{+}\right) \right\vert ^{2}d\widehat{\omega }_{N}\lesssim \left\vert
Q\right\vert _{\widehat{\sigma }_{N}^{+}}\ .  \label{ftc}
\end{equation}%
Indeed, if a square $Q$ is symmetric about the $x_{1}$-axis, then both $%
\left\vert Q\right\vert _{\widehat{\sigma }_{N}^{+}}=\left\vert Q\right\vert
_{\widehat{\sigma }_{N}^{-}}=\frac{1}{2}\left\vert Q\right\vert _{\widehat{%
\sigma }_{N}}$ and 
\begin{equation*}
\int_{Q}\left\vert \mathbf{R}^{1,2}\left( \mathbf{1}_{Q}\widehat{\sigma }%
_{N}^{+}\right) \right\vert ^{2}d\widehat{\omega }_{N}=\int_{Q}\left\vert 
\mathbf{R}^{1,2}\left( \mathbf{1}_{Q}\widehat{\sigma }_{N}^{-}\right)
\right\vert ^{2}d\widehat{\omega }_{N}
\end{equation*}%
by symmetry. Since the testing condition holds for the weight pair $\left( 
\widehat{\sigma }_{N},\widehat{\omega }_{N}\right) $, we easily obtain (\ref%
{ftc}) for such symmetric squares $Q$. The general case now follows easily
from this.

\item The backward testing condition holds:%
\begin{equation*}
\int_{Q}\left\vert \mathbf{R}^{1,2}\left( \mathbf{1}_{Q}\widehat{\omega }%
_{N}\right) \right\vert ^{2}d\widehat{\sigma }_{N}^{+}\lesssim \left\vert
Q\right\vert _{\widehat{\omega }_{N}}\ ,
\end{equation*}%
since it holds for the larger measure $\widehat{\sigma }_{N}$ in place of $%
\widehat{\sigma }_{N}^{+}$.

\item The forward energy condition holds:%
\begin{equation*}
\sum_{\overset{\cdot }{\bigcup }_{r=1}^{\infty }Q_{r}\subset R}\left( \frac{%
\mathrm{P}^{1}\left( Q_{r},\mathbf{1}_{R}\widehat{\sigma }_{N}^{+}\right) }{%
\sqrt{\left\vert Q_{r}\right\vert }}\right) ^{2}\left\Vert \mathsf{P}%
_{Q_{r}}^{\widehat{\omega }}x\right\Vert _{L^{2}\left( \widehat{\omega }%
_{N}\right) }^{2}\lesssim \left\vert R\right\vert _{\widehat{\sigma }%
_{N}^{+}}\ .
\end{equation*}

\item The backward energy condition holds:%
\begin{equation*}
\sum_{\overset{\cdot }{\bigcup }_{r=1}^{\infty }Q_{r}\subset R}\left( \frac{%
\mathrm{P}^{1}\left( Q_{r},\mathbf{1}_{R}\widehat{\omega }_{N}\right) }{%
\sqrt{\left\vert Q_{r}\right\vert }}\right) ^{2}\left\Vert \mathsf{P}%
_{Q_{r}}^{\widehat{\sigma }_{N}^{+}}x\right\Vert _{L^{2}\left( \widehat{%
\sigma }_{N}^{+}\right) }^{2}\lesssim \left\vert R\right\vert _{\widehat{%
\omega }_{N}}\ .
\end{equation*}
\end{enumerate}

Now we can apply our $T1$ theorem with an energy side condition\ in \cite%
{SaShUr7} (or see \cite{SaShUr6} or \cite{SaShUr9}) to obtain the dual norm
inequality%
\begin{eqnarray*}
\int \left\vert \mathbf{R}^{1,2}\left( g\widehat{\omega }_{N}\right)
\right\vert ^{2}d\widehat{\sigma }_{N}^{+} &\lesssim &\int \left\vert
g\right\vert ^{2}d\widehat{\omega }_{N}, \\
\text{i.e. }\mathfrak{N}_{\mathbf{R}^{1,2}}\left( \widehat{\sigma }_{N}^{+},%
\widehat{\omega }_{N}\right) &<&\infty .
\end{eqnarray*}%
Consider now the weight pair $\left( \widehat{\sigma }_{N},\widehat{\omega }%
_{N}\right) $. We have $\widehat{\sigma }_{N}=\widehat{\sigma }_{N}^{+}+%
\widehat{\sigma }_{N}^{-}$ and $\mathfrak{N}_{\mathbf{R}^{1,2}}\left( \left( 
\widehat{\sigma }_{N}^{-},\widehat{\omega }_{N}\right) \right) =\mathfrak{N}%
_{\mathbf{R}^{1,2}}\left( \widehat{\sigma }_{N}^{+},\widehat{\omega }%
_{N}\right) $ by symmetry, and so%
\begin{equation*}
\mathfrak{N}_{\mathbf{R}^{1,2}}\left( \left( \widehat{\sigma }_{N},\widehat{%
\omega }_{N}\right) \right) \leq \mathfrak{N}_{\mathbf{R}^{1,2}}\left( 
\widehat{\sigma }_{N}^{+},\widehat{\omega }_{N}\right) +\mathfrak{N}_{%
\mathbf{R}^{1,2}}\left( \left( \widehat{\sigma }_{N}^{-},\widehat{\omega }%
_{N}\right) \right) =2\mathfrak{N}_{\mathbf{R}^{1,2}}\left( \widehat{\sigma }%
_{N}^{+},\widehat{\omega }_{N}\right) <\infty .
\end{equation*}

Thus we have shown that the two weight norm inequality for the Riesz
transform $\mathbf{R}^{1,2}$ holds in the plane with respect to the weight
pair $\left( \widehat{\sigma }_{N},\widehat{\omega }_{N}\right) $ uniformly
in $N\geq 1$, and in Subsubsection \ref{Subsubfails} above, we showed that
the backward energy constants with respect to the weight pairs $\left( 
\widehat{\sigma }_{N},\widehat{\omega }_{N}\right) $ are unbounded in $N\geq
1$. This completes the proof of Theorem \ref{counterexample} in the special
case $\alpha =1$ and $n=2$.

\section{The general case $0\leq \protect\alpha <n$ and $n\geq 2$}

The measure pair $\left( \widehat{\sigma }_{N},\widehat{\omega }_{N}\right) $
just constructed above in the plane serves to show that the energy
conditions are not implied by boundedness of the fractional Riesz transform $%
\mathbf{R}^{n-1,n}$ of order $n-1$ in $\mathbb{R}^{n}$ for $n\geq 2$ -
simply embed the measures in the two-dimensional subspace $\mathbb{R}^{2}$
spanned by the unit coordinate vectors $\mathbf{e}_{1}$ and$\ \mathbf{e}_{2}$%
\textbf{. }The reason for this is that the restriction of the convolution
kernel $\mathbf{K}^{n-1,n}\left( w\right) =\frac{\left(
w_{1}w_{2},...,w_{n}\right) }{\left\vert \left( w_{1}w_{2},...,w_{n}\right)
\right\vert ^{n+1-\left( n-1\right) }}$ to $\mathbb{R}^{2}$ is the kernel $%
\mathbf{K}^{1,2}\left( w\right) =\frac{\left( w_{1}w_{2}\right) }{\left\vert
\left( w_{1}w_{2}\right) \right\vert ^{n+1-\left( n-1\right) }}$. If we
remain in dimension $n=2$, but permit $0\leq \alpha <2$, then the argument
above applies if we take%
\begin{equation*}
s_{i}^{k}=\left( \frac{1}{3}\right) ^{k\left( 3-2\alpha \right) }\left( 
\frac{1}{3}\right) ^{k},
\end{equation*}%
along with similar arithmetic adjustments elsewhere.

In the general case $0\leq \alpha <n$, $n\geq 2$, we start with the
computation that%
\begin{eqnarray*}
&&\frac{d}{dx_{1}}\frac{x_{1}-y_{1}}{\left[ \left( x_{1}-y_{1}\right)
^{2}+\left( x_{2}-y_{2}\right) ^{2}\right] ^{\frac{n+1-\alpha }{2}}} \\
&=&\frac{\left[ \left( x_{1}-y_{1}\right) ^{2}+\left( x_{2}-y_{2}\right) ^{2}%
\right] ^{\frac{n+1-\alpha }{2}}-\frac{n+1-\alpha }{2}2\left(
x_{1}-y_{1}\right) ^{2}\left[ \left( x_{1}-y_{1}\right) ^{2}+\left(
x_{2}-y_{2}\right) ^{2}\right] ^{\frac{n-1-\alpha }{2}}}{\left[ \left(
x_{1}-y_{1}\right) ^{2}+\left( x_{2}-y_{2}\right) ^{2}\right] ^{n+1-\alpha }}
\\
&=&\frac{\left( x_{1}-y_{1}\right) ^{2}+\left( x_{2}-y_{2}\right)
^{2}-\left( n+1-\alpha \right) \left( x_{1}-y_{1}\right) ^{2}}{\left[ \left(
x_{1}-y_{1}\right) ^{2}+\left( x_{2}-y_{2}\right) ^{2}\right] ^{\frac{%
n+3-\alpha }{2}}} \\
&=&\frac{\left( x_{2}-y_{2}\right) ^{2}-\left( n-\alpha \right) \left(
x_{1}-y_{1}\right) ^{2}}{\left[ \left( x_{1}-y_{1}\right) ^{2}+\left(
x_{2}-y_{2}\right) ^{2}\right] ^{\frac{5-\alpha }{2}}}<0,
\end{eqnarray*}%
provided%
\begin{equation*}
\left\vert x_{2}-y_{2}\right\vert <\sqrt{n-\alpha }\left\vert
x_{1}-y_{1}\right\vert .
\end{equation*}

Thus in the subcase $0\leq \alpha <n-1$ and $\left\vert
x_{2}-y_{2}\right\vert <\left\vert x_{1}-y_{1}\right\vert $, the $x_{1}$
derivative of the kernel $K_{1}^{\alpha ,n}\left( x-y\right) $ is negative,
and the above construction of a family of weight pairs in the plane can be
modified in a purely arithmetic way so as to show that the energy conditions
are not necessary for boundedness of the fractional Riesz transform $\mathbf{%
R}^{\alpha ,n}$. The modified measure pair $\left( \widehat{\sigma }_{N},%
\widehat{\omega }_{N}\right) $ lives in the two-dimensional subspace $%
\mathbb{R}^{2}$, and as a consequence, the components $R_{3}^{\alpha
,n},R_{4}^{\alpha ,n},...,R_{n}^{\alpha ,n}$ of $\mathbf{R}^{\alpha ,n}$ are
all trivially bounded since both $R_{j}^{\alpha ,n}\widehat{\sigma }%
_{N}\equiv 0$ and $R_{j}^{\alpha ,n}\widehat{\omega }_{N}\equiv 0$ for $%
j\geq 3$.

However, in the subcase $n-1<\alpha <n$, we must alter the geometry as well,
by translating the point masses of $\omega _{N}$ at an angle less than $%
\theta _{\alpha ,n}$ instead of less than $\frac{\pi }{4}=45^{\circ }$, where%
\begin{equation*}
\tan \theta _{n,\alpha }=\gamma _{n,\alpha }=\sqrt{n-\alpha }.
\end{equation*}%
The angle $\theta _{n,\alpha }$ is less than $\frac{\pi }{4}=45^{\circ }$
precisely when $n-1<\alpha <n$, and with this geometric alteration, the
above construction again goes through with only changes in arithmetic.

\begin{remark}
\label{loc unif}If $\left( \widehat{\sigma }_{N},\widehat{\omega }%
_{N}\right) $ is the weight pair constructed above, then a very lengthy but
straightforward computation shows that the family of localized operators $%
\left\{ \mathbf{R}_{\mathcal{J}}^{\alpha ,n}\Theta _{j}\right\} _{\mathcal{J}%
\in \mathfrak{J}\text{ and }1\leq j\leq N}$ is uniformly bounded from $%
L^{2}\left( \widehat{\sigma }_{N}\right) $ to $L^{2}\left( \widehat{\omega }%
_{N}\right) $. Indeed, the weight pair $\left( \widehat{\sigma }_{N},%
\widehat{\omega }_{N}\right) $ satisfies the Muckenhoupt and energy
conditions uniformly in $N\geq 1$ by Lemma \ref{CZalpha}, and the Calder\'{o}%
n-Zygmund norms of the kernels of $\mathbf{R}_{\mathcal{J}}^{\alpha
,n}\Theta _{j}$ are uniformly bounded for $\mathcal{J}\in \mathfrak{J}$ and $%
1\leq j\leq N$. Finally, the testing constants $\mathfrak{T}_{\mathbf{R}_{%
\mathcal{J}}^{\alpha ,n}\Theta _{j}}\left( \widehat{\sigma }_{N},\widehat{%
\omega }_{N}\right) $ are uniformly bounded for $\in \mathbb{N}$, $\mathcal{J%
}\in \mathfrak{J}$ and $1\leq j\leq M$. Thus from the $T1$ theorem in \cite%
{SaShUr7} with an energy side condition, we obtain the boundedness of the
operators $\mathbf{R}_{\mathcal{J}}^{\alpha ,n}\Theta _{j}$ from $%
L^{2}\left( \widehat{\sigma }_{N}\right) $ to $L^{2}\left( \widehat{\omega }%
_{N}\right) $ uniform in $N\geq 1$. We leave details to the interested
reader.
\end{remark}


\begin{thebibliography}{LaSaShUrWi}
\bibitem[DaJo]{DaJo} \textsc{David, Guy, Journ\'{e}, Jean-Lin,} \textit{A
boundedness criterion for generalized Calder\'{o}n-Zygmund operators,} Ann.
of Math. (2) \textbf{120} (1984), 371--397, MR763911 (85k:42041).

\bibitem[DaJoSe]{DaJoSe} \textsc{David,G.,Journ\'{e},J.-L.,andSemmes,S.,}Op%
\'{e}rateurs de Calder\'{o}n-Zygmund, fonctions para-accr\'{e}tives et
interpolation. Rev. Mat. Iberoamericana 1 (1985), 1--56.

\bibitem[HuMuWh]{HuMuWh} \textsc{R. Hunt, B. Muckenhoupt} \textsc{and R. L.
Wheeden,} \textit{Weighted norm inequalities for the conjugate function and
the Hilbert transform}, Trans. Amer. Math. Soc. \textbf{176} (1973), 227-251.

\bibitem[Hyt2]{Hyt2} \textsc{Hyt\"{o}nen, Tuomas, }\textit{The two weight
inequality for the Hilbert transform with general measures, \texttt{%
arXiv:1312.0843v2}.}

\bibitem[Lac]{Lac} \textsc{Lacey, Michael T.,}\textit{\ Two weight
inequality for the Hilbert transform: A real variable characterization, II},
Duke Math. J. Volume \textbf{163}, Number 15 (2014), 2821-2840.

\bibitem[Lac2]{Lac2} \textsc{Lacey, Michael T.,}\textit{\ The two weight
inequality for the Hilbert transform: a primer}, \texttt{arXiv:1304.5004v1}.

\bibitem[LaMa]{LaMa} \textsc{M. T. Lacey and H. Martikainen,} \textit{Local }%
$Tb$\textit{\ theorem with }$L^{2}$\textit{\ testing conditions and general
measures: Calder\'{o}n--Zygmund operators}, \texttt{arXiv:1310.08531v1.}

\bibitem[LaSaUr1]{LaSaUr1} \textsc{Lacey, Michael T., Sawyer, Eric T.,
Uriarte-Tuero, Ignacio,} \textit{A characterization of two weight norm
inequalities for maximal singular integrals with one doubling measure,}
Analysis \& PDE, Vol. \textbf{5} (2012), No. 1, 1-60.

\bibitem[LaSaUr2]{LaSaUr2} \textsc{Lacey, Michael T., Sawyer, Eric T.,
Uriarte-Tuero, Ignacio,} \textit{A Two Weight Inequality for the Hilbert
transform assuming an energy hypothesis, } Journal of Functional Analysis,
Volume \textbf{263} (2012), Issue 2, 305-363.

\bibitem[LaSaShUr]{LaSaShUr} \textsc{Lacey, Michael T., Sawyer, Eric T.,
Shen, Chun-Yen, Uriarte-Tuero, Ignacio,} \textit{The Two weight inequality
for Hilbert transform, coronas, and energy conditions, }\texttt{arXiv:}
(2011).

\bibitem[LaSaShUr2]{LaSaShUr2} \textsc{Lacey, Michael T., Sawyer, Eric T.,
Shen, Chun-Yen, Uriarte-Tuero, Ignacio,} \textit{Two Weight Inequality for
the Hilbert Transform: A Real Variable Characterization, }\texttt{%
arXiv:1201.4319} (2012).

\bibitem[LaSaShUr3]{LaSaShUr3} \textsc{Lacey, Michael T., Sawyer, Eric T.,
Shen, Chun-Yen, Uriarte-Tuero, Ignacio,} \textit{Two weight inequality for
the Hilbert transform: A real variable characterization I}, Duke Math. J,
Volume \textbf{163}, Number 15 (2014), 2795-2820.

\bibitem[LaSaShUrWi]{LaSaShUrWi} \textsc{Lacey, Michael T., Sawyer, Eric T.,
Shen, Chun-Yen, Uriarte-Tuero, Ignacio, Wick, Brett D.,} \textit{Two weight
inequalities for the Cauchy transform from }$\mathbb{R}$ to $\mathbb{C}_{+}$%
, \textit{\texttt{arXiv:1310.4820v4}}.

\bibitem[LaWi1]{LaWi1} \textsc{Lacey, Michael T., Wick, Brett D.,} \textit{%
Two weight inequalities for the Cauchy transform from }$\mathbb{R}$ to $%
\mathbb{C}_{+}$, \textit{\texttt{arXiv:1310.4820v1}}.

\bibitem[LaWi]{LaWi} \textsc{Lacey, Michael T., Wick, Brett D.,} \textit{Two
weight inequalities for Riesz transforms: uniformly full dimension weights}, 
\textit{\texttt{arXiv:1312.6163v1,v2,v3}}.

\bibitem[NTV4]{NTV4} \textsc{F. Nazarov, S. Treil and A. Volberg,} \textit{%
Two weight estimate for the Hilbert transform and corona decomposition for
non-doubling measures}, preprint (2004) \texttt{arxiv:1003.1596}

\bibitem[Saw1]{MR676801} \textsc{E. Sawyer,} \textit{A characterization of a
two-weight norm inequality for maximal operators}, Studia Math. \textbf{75}
(1982), 1-11, MR\{676801 (84i:42032)\}.

\bibitem[Saw]{Saw3} \textsc{E. Sawyer,} \textit{A characterization of two
weight norm inequalities for fractional and Poisson integrals}, Trans.
A.M.S. \textbf{308} (1988), 533-545, MR\{930072 (89d:26009)\}.

\bibitem[SaShUr2]{SaShUr2} \textsc{Sawyer, Eric T., Shen, Chun-Yen,
Uriarte-Tuero, Ignacio,} A \textit{two weight theorem for }$\alpha $\textit{%
-fractional singular integrals with an energy side condition}, \texttt{%
arXiv:1302.5093v8.}

\bibitem[SaShUr3]{SaShUr3} \textsc{Sawyer, Eric T., Shen, Chun-Yen,
Uriarte-Tuero, Ignacio,} \textit{A geometric condition, necessity of energy,
and two weight boundedness of fractional Riesz transforms}, \texttt{%
arXiv:1310.4484v1.}

\bibitem[SaShUr4]{SaShUr4} \textsc{Sawyer, Eric T., Shen, Chun-Yen,
Uriarte-Tuero, Ignacio,} \textit{A note on failure of energy reversal for
classical fractional singular integrals}, IMRN, Volume \textbf{2015}, Issue
19, 9888-9920.

\bibitem[SaShUr5]{SaShUr5} \textsc{Sawyer, Eric T., Shen, Chun-Yen,
Uriarte-Tuero, Ignacio,} A \textit{two weight theorem for }$\alpha $\textit{%
-fractional singular integrals with an energy side condition and quasicube
testing}, \texttt{arXiv:1302.5093v10.}

\bibitem[SaShUr6]{SaShUr6} \textsc{Sawyer, Eric T., Shen, Chun-Yen,
Uriarte-Tuero, Ignacio,} A \textit{two weight theorem for }$\alpha $\textit{%
-fractional singular integrals with an energy side condition, quasicube
testing and common point masses}, \texttt{arXiv:1505.07816v2,v3.}

\bibitem[SaShUr7]{SaShUr7} \textsc{Sawyer, Eric T., Shen, Chun-Yen,
Uriarte-Tuero, Ignacio,} A \textit{two weight theorem for }$\alpha $\textit{%
-fractional singular integrals with an energy side condition}, Revista Mat.
Iberoam. \textbf{32} (2016), no. 1, 79-174.

\bibitem[SaShUr8]{SaShUr8} \textsc{Sawyer, Eric T., Shen, Chun-Yen,
Uriarte-Tuero, Ignacio,} The \textit{two weight }$T1$ \textit{theorem for
fractional Riesz transforms when one measure is supported on a curve}, 
\texttt{arXiv:1505.07822v4}.

\bibitem[SaShUr9]{SaShUr9} \textsc{Sawyer, Eric T., Shen, Chun-Yen,
Uriarte-Tuero, Ignacio,} A \textit{two weight fractional singular integral
theorem with side conditions, energy and }$k$\textit{-energy dispersed,}
Harmonic Analysis, Partial Differential Equations, Complex Analysis, Banach
Spaces, and Operator Theory (Volume 2) (Celebrating Cora Sadosky's life),
Springer 2017 (see also \texttt{arXiv:1603.04332v2}).

\bibitem[SaShUr10]{SaShUr10} \textsc{Sawyer, Eric T., Shen, Chun-Yen,
Uriarte-Tuero, Ignacio,} \textit{A good-}$\lambda $\textit{\ lemma, two
weight }$T1$\textit{\ theorems without weak boundedness, and a two weight
accretive global }$Tb$\textit{\ theorem,} Harmonic Analysis, Partial
Differential Equations and Applications (In Honor of Richard L. Wheeden),
Birkh\"{a}user 2017 (see also \texttt{arXiv:1609.08125v2}).

\bibitem[SaShUr11]{SaShUr11} \textsc{Sawyer, Eric T., Shen, Chun-Yen,
Uriarte-Tuero, Ignacio,} \textit{A counterexample in the theory of Calder%
\'{o}n-Zygmund operators, }\texttt{arXiv:16079.06071v3}.

\bibitem[SaShUr12]{SaShUr12} \textsc{Sawyer, Eric T., Shen, Chun-Yen,
Uriarte-Tuero, Ignacio,} \textit{A two weight local }$Tb$ theorem for the
Hilbert transform\textit{, }\texttt{arXiv:1709.09595v6}.

\bibitem[SaWh]{SaWh} \textsc{E. Sawyer and R. L. Wheeden,} Weighted
inequalities for fractional integrals on Euclidean and homogeneous spaces, 
\textit{Amer. J. Math. }\textbf{114} (1992), 813-874.

\bibitem[Ste]{Ste} \textsc{E. M. Stein,} \textit{Harmonic Analysis:
real-variable methods, orthogonality, and oscillatory integrals},\textit{\ }%
Princeton University Press, Princeton, N. J., 1993.

\bibitem[Vol]{Vol} \textsc{A. Volberg,} \textit{Calder\'{o}n-Zygmund
capacities and operators on nonhomogeneous spaces,} CBMS Regional Conference
Series in Mathematics (2003), MR\{2019058 (2005c:42015)\}.
\end{thebibliography}
\end{document}